\documentclass[9pt]{amsart}
\usepackage{amsmath}
\usepackage{amsxtra}
\usepackage{amscd}
\usepackage{amsthm}
\usepackage{amsfonts}
\usepackage{amssymb}
\usepackage{eucal}
\usepackage[all]{xy}
\usepackage{graphicx}
\usepackage[hidelinks]{hyperref}

\newtheorem{cor}[subsubsection]{Corollary}
\newtheorem{lem}[subsubsection]{Lemma}
\newtheorem{prop}[subsubsection]{Proposition}

\newtheorem{conj}[subsubsection]{Conjecture}
\newtheorem{thm}[subsubsection]{Theorem}

\theoremstyle{definition}

\theoremstyle{remark}
\newtheorem{rem}[subsubsection]{Remark}

\newcommand{\thmref}[1]{Theorem~\ref{#1}}

\newcommand{\secref}[1]{Sect.~\ref{#1}}
\newcommand{\lemref}[1]{Lemma~\ref{#1}}
\newcommand{\propref}[1]{Proposition~\ref{#1}}
\newcommand{\corref}[1]{Corollary~\ref{#1}}
\newcommand{\conjref}[1]{Conjecture~\ref{#1}}

\numberwithin{equation}{section}

\newcommand{\nc}{\newcommand}
\nc{\renc}{\renewcommand}
\nc{\ssec}{\subsection}
\nc{\sssec}{\subsubsection}
\nc{\on}{\operatorname}

\nc\ol{\overline}
\nc\wt{\widetilde}
\nc\tboxtimes{\wt{\boxtimes}}
\nc\tstar{\wt{\star}}
\nc{\alp}{\alpha}

\nc{\ZZ}{{\mathbb Z}}
\nc{\NN}{{\mathbb N}}
\nc{\OO}{{\mathbb O}}
\renc{\SS}{{\mathbb S}}
\nc{\DD}{{\mathbb D}}
\nc{\GG}{{\mathbb G}}

\nc{\Fq}{{\mathbb F}_q}
\nc{\Fqb}{\ol{{\mathbb F}_q}}
\nc{\Ql}{\ol{{\mathbb Q}_\ell}}
\nc{\id}{\text{id}}
\nc\X{\mathcal X}

\nc{\Hom}{\on{Hom}}
\nc{\Lie}{\on{Lie}}
\nc{\Loc}{\on{Loc}}
\nc{\Pic}{\on{Pic}}
\nc{\Bun}{\on{Bun}}
\nc{\IC}{\on{IC}}
\nc{\Aut}{\on{Aut}}
\nc{\rk}{\on{rk}}
\nc{\Sh}{\on{Sh}}
\nc{\Perv}{\on{Perv}}
\nc{\pos}{{\on{pos}}}
\nc{\Conv}{\on{Conv}}
\nc{\Sph}{\on{Sph}}
\nc{\Sym}{\on{Sym}}
\nc{\BunBb}{\overline{\Bun}_B}
\nc{\BunNb}{\overline{\Bun}_N}
\nc{\BunTb}{\overline{\Bun}_T}
\nc{\BunBbm}{\overline{\Bun}_{B^-}}
\nc{\BunBbel}{\overline{\Bun}_{B,el}}
\nc{\BunBbmel}{\overline{\Bun}_{B^-,el}}
\nc{\Buno}{\overset{o}{\Bun}}
\nc{\BunPb}{{\overline{\Bun}_P}}
\nc{\BunBM}{\Bun_{B(M)}}
\nc{\BunBMb}{\overline{\Bun}_{B(M)}}
\nc{\BunPbw}{{\widetilde{\Bun}_P}}
\nc{\BunBP}{\widetilde{\Bun}_{B,P}}
\nc{\GUb}{\overline{G/U}}
\nc{\GUPb}{\overline{G/U(P)}}

\nc{\Hhom}{\underline{\on{Hom}}}
\nc\syminfty{\on{Sym}^{\infty}}
\nc\lal{\ol{\lambda}}
\nc\xl{\ol{x}}
\nc\thl{\ol{\theta}}
\nc\nul{\ol{\nu}}
\nc\mul{\ol{\mu}}
\nc\Sum\Sigma
\nc{\oX}{\overset{o}{X}{}}
\nc{\hl}{\overset{\leftarrow}h{}}
\nc{\hr}{\overset{\rightarrow}h{}}
\nc{\M}{{\mathcal M}}
\nc{\N}{{\mathcal N}}
\nc{\F}{{\mathcal F}}
\nc{\D}{{\mathcal D}}
\nc{\Q}{{\mathcal Q}}
\nc{\Y}{{\mathcal Y}}
\nc{\G}{{\mathcal G}}
\nc{\E}{{\mathcal E}}
\nc{\CalC}{{\mathcal C}}
\nc\Dh{\widehat{\D}}

\nc{\C}{{\mathcal C}}
\nc{\K}{{\mathcal K}}
\renewcommand{\H}{{\mathcal H}}

\nc{\T}{{\mathcal T}}
\nc{\V}{{\mathcal V}}
\renc{\P}{{\mathcal P}}
\nc{\A}{{\mathcal A}}
\nc{\B}{{\mathcal B}}
\nc{\U}{{\mathcal U}}

\nc{\Gr}{{\on{Gr}}}

\nc{\frn}{{\check{\mathfrak u}(P)}}

\nc{\fC}{\mathfrak C}
\nc{\p}{\mathfrak p}
\nc{\q}{\mathfrak q}
\nc\f{{\mathfrak f}}

\nc{\qo}{{\mathfrak q}}
\nc{\po}{{\mathfrak p}}
\nc{\s}{{\mathfrak s}}
\nc\w{\text{w}}

\nc\mathi\iota
\nc\Spec{\on{Spec}}
\nc\Mod{\on{Mod}}
\nc{\tw}{\widetilde{\mathfrak t}}
\nc{\pw}{\widetilde{\mathfrak p}}
\nc{\qw}{\widetilde{\mathfrak q}}
\nc{\jw}{\widetilde j}

\nc{\grb}{\overline{\Gr}}
\nc{\I}{\mathcal I}

\nc{\lambdach}{{\check\lambda}}
\nc{\Lambdach}{{\check\Lambda}{}}
\nc{\much}{{\check\mu}}
\nc{\omegach}{{\check\omega}}
\nc{\nuch}{{\check\nu}}
\nc{\etach}{{\check\eta}}
\nc{\alphach}{{\check\alpha}}
\nc{\oblvtach}{{\check\oblvta}}
\nc{\rhoch}{{\check\rho}}
\nc{\ch}{{\check h}}

\nc{\Hb}{\overline{\H}}

\nc{\BA}{{\mathbb{A}}}
\nc{\BC}{{\mathbb{C}}}
\nc{\BE}{{\mathbb{E}}}
\nc{\BF}{{\mathbb{F}}}
\nc{\BG}{{\mathbb{G}}}
\nc{\BM}{{\mathbb{M}}}
\nc{\BO}{{\mathbb{O}}}
\nc{\BD}{{\mathbb{D}}}
\nc{\BL}{{\mathbb{L}}}
\nc{\Bl}{{\mathbb{l}}}
\nc{\BN}{{\mathbb{N}}}
\nc{\BP}{{\mathbb{P}}}
\nc{\BQ}{{\mathbb{Q}}}
\nc{\BR}{{\mathbb{R}}}
\nc{\BZ}{{\mathbb{Z}}}
\nc{\BS}{{\mathbb{S}}}
\nc{\BT}{{\mathbb{T}}}

\nc{\CA}{{\mathcal{A}}}
\nc{\CB}{{\mathcal{B}}}

\nc{\CE}{{\mathcal{E}}}
\nc{\CF}{{\mathcal{F}}}
\nc{\CH}{{\mathcal{H}}}

\nc{\CL}{{\mathcal{L}}}
\nc{\CC}{{\mathcal{C}}}
\nc{\CG}{{\mathcal{G}}}
\nc{\CM}{{\mathcal{M}}}
\nc{\CN}{{\mathcal{N}}}
\nc{\CK}{{\mathcal{K}}}
\nc{\CO}{{\mathcal{O}}}
\nc{\CP}{{\mathcal{P}}}
\nc{\CQ}{{\mathcal{Q}}}
\nc{\CR}{{\mathcal{R}}}
\nc{\CS}{{\mathcal{S}}}
\nc{\CT}{{\mathcal{T}}}
\nc{\CU}{{\mathcal{U}}}
\nc{\CV}{{\mathcal{V}}}
\nc{\CW}{{\mathcal{W}}}
\nc{\CX}{{\mathcal{X}}}
\nc{\CY}{{\mathcal{Y}}}
\nc{\CZ}{{\mathcal{Z}}}
\nc{\CI}{{\mathcal{I}}}
\nc{\cD}{{\mathcal{D}}}
\nc{\ocD}{\overset{\circ}{\mathcal{D}}}

\nc{\csM}{{\check{\mathcal A}}{}}
\nc{\oM}{{\overset{\circ}{\mathcal M}}{}}
\nc{\obM}{{\overset{\circ}{\mathbf M}}{}}
\nc{\oCA}{{\overset{\circ}{\mathcal A}}{}}
\nc{\obA}{{\overset{\circ}{\mathbf A}}{}}
\nc{\ooM}{{\overset{\circ}{M}}{}}
\nc{\osM}{{\overset{\circ}{\mathsf M}}{}}
\nc{\vM}{{\overset{\bullet}{\mathcal M}}{}}
\nc{\nM}{{\underset{\bullet}{\mathcal M}}{}}
\nc{\oD}{{\overset{\circ}{\mathcal D}}{}}
\nc{\obD}{{\overset{\circ}{\mathbf D}}{}}
\nc{\oA}{{\overset{\circ}{\mathbb A}}{}}
\nc{\op}{{\overset{\bullet}{\mathbf p}}{}}
\nc{\cp}{{\overset{\circ}{\mathbf p}}{}}
\nc{\oU}{{\overset{\bullet}{\mathcal U}}{}}
\nc{\oZ}{{\overset{\circ}{\mathcal Z}}{}}
\nc{\ofZ}{{\overset{\circ}{\mathfrak Z}}{}}
\nc{\oF}{{\overset{\circ}{\fF}}}

\nc{\fa}{{\mathfrak{a}}}
\nc{\fb}{{\mathfrak{b}}}
\nc{\fc}{{\mathfrak{c}}}
\nc{\fch}{{\mathfrak{ch}}}
\nc{\fd}{{\mathfrak{d}}}
\nc{\ff}{{\mathfrak{f}}}
\nc{\fg}{{\mathfrak{g}}}
\nc{\fgl}{{\mathfrak{gl}}}
\nc{\fh}{{\mathfrak{h}}}
\nc{\fj}{{\mathfrak{j}}}
\nc{\fl}{{\mathfrak{l}}}
\nc{\fm}{{\mathfrak{m}}}
\nc{\fn}{{\mathfrak{n}}}
\nc{\fu}{{\mathfrak{u}}}
\nc{\fp}{{\mathfrak{p}}}
\nc{\fr}{{\mathfrak{r}}}
\nc{\fs}{{\mathfrak{s}}}
\nc{\ft}{{\mathfrak{t}}}
\nc{\fT}{{\mathfrak{T}}}
\nc{\fz}{{\mathfrak{z}}}
\nc{\fsl}{{\mathfrak{sl}}}
\nc{\hsl}{{\widehat{\mathfrak{sl}}}}
\nc{\hgl}{{\widehat{\mathfrak{gl}}}}
\nc{\hg}{{\widehat{\mathfrak{g}}}}
\nc{\htt}{{\widehat{\mathfrak{t}}}}
\nc{\chg}{{\widehat{\mathfrak{g}}}{}^\vee}
\nc{\hn}{{\widehat{\mathfrak{n}}}}
\nc{\chn}{{\widehat{\mathfrak{n}}}{}^\vee}

\nc{\fA}{{\mathfrak{A}}}
\nc{\fB}{{\mathfrak{B}}}
\nc{\fD}{{\mathfrak{D}}}
\nc{\fE}{{\mathfrak{E}}}
\nc{\fF}{{\mathfrak{F}}}
\nc{\fG}{{\mathfrak{G}}}
\nc{\fK}{{\mathfrak{K}}}
\nc{\fL}{{\mathfrak{L}}}
\nc{\fM}{{\mathfrak{M}}}
\nc{\fN}{{\mathfrak{N}}}
\nc{\fP}{{\mathfrak{P}}}
\nc{\fU}{{\mathfrak{U}}}
\nc{\fV}{{\mathfrak{V}}}
\nc{\fZ}{{\mathfrak{Z}}}

\nc{\bb}{{\mathbf{b}}}
\nc{\bc}{{\mathbf{c}}}
\nc{\bd}{{\mathbf{d}}}
\nc{\bbf}{{\mathbf{f}}}
\nc{\be}{{\mathbf{e}}}
\nc{\bg}{{\mathbf{g}}}
\nc{\bi}{{\mathbf{i}}}
\nc{\bj}{{\mathbf{j}}}
\nc{\bn}{{\mathbf{n}}}
\nc{\bp}{{\mathbf{p}}}
\nc{\bq}{{\mathbf{q}}}
\nc{\bu}{{\mathbf{u}}}
\nc{\bv}{{\mathbf{v}}}
\nc{\bx}{{\mathbf{x}}}
\nc{\bs}{{\mathbf{s}}}
\nc{\by}{{\mathbf{y}}}
\nc{\bw}{{\mathbf{w}}}
\nc{\bA}{{\mathbf{A}}}
\nc{\bK}{{\mathbf{K}}}
\nc{\bB}{{\mathbf{B}}}
\nc{\bC}{{\mathbf{C}}}
\nc{\bG}{{\mathbf{G}}}
\nc{\bD}{{\mathbf{D}}}
\nc{\bH}{{\mathbf{He}}}
\nc{\bM}{{\mathbf{M}}}
\nc{\bN}{{\mathbf{N}}}
\nc{\bO}{{\mathbf{O}}}
\nc{\bV}{{\mathbf{V}}}
\nc{\bW}{{\mathbf{Wh}}}
\nc{\bX}{{\mathbf{X}}}
\nc{\bY}{{\mathbf{Y}}}
\nc{\bZ}{{\mathbf{Z}}}
\nc{\bS}{{\mathbf{S}}}
\nc{\bT}{{\mathbf{T}}}

\nc{\sA}{{\mathsf{A}}}
\nc{\sB}{{\mathsf{B}}}
\nc{\sC}{{\mathsf{C}}}
\nc{\sD}{{\mathsf{D}}}
\nc{\sF}{{\mathsf{F}}}
\nc{\sG}{{\mathsf{G}}}
\nc{\sH}{{\mathsf{H}}}
\nc{\sK}{{\mathsf{K}}}
\nc{\sL}{{\mathsf{L}}}
\nc{\sM}{{\mathsf{M}}}
\nc{\sO}{{\mathsf{O}}}
\nc{\sR}{{\mathsf{R}}}
\nc{\sU}{{\mathsf{U}}}
\nc{\sW}{{\mathsf{W}}}
\nc{\sQ}{{\mathsf{Q}}}
\nc{\sP}{{\mathsf{P}}}
\nc{\sY}{{\mathsf{Y}}}
\nc{\sZ}{{\mathsf{Z}}}
\nc{\sfp}{{\mathsf{p}}}
\nc{\sfq}{{\mathsf{q}}}
\nc{\sr}{{\mathsf{r}}}
\nc{\sk}{{\mathsf{k}}}
\nc{\su}{{\mathsf{u}}}
\nc{\sv}{{\mathsf{v}}}
\nc{\sg}{{\mathsf{g}}}
\nc{\sff}{{\mathsf{f}}}
\nc{\sfb}{{\mathsf{b}}}
\nc{\sfc}{{\mathsf{c}}}
\nc{\sd}{{\mathsf{d}}}

\nc{\BK}{{\bar{K}}}

\nc{\tA}{{\widetilde{\mathbf{A}}}}
\nc{\tB}{{\widetilde{\mathcal{B}}}}
\nc{\tg}{{\widetilde{\mathfrak{g}}}}
\nc{\tG}{{\widetilde{G}}}
\nc{\TM}{{\widetilde{\mathbb{M}}}{}}
\nc{\tO}{{\widetilde{\mathsf{O}}}{}}
\nc{\tU}{{\widetilde{\mathfrak{U}}}{}}
\nc{\TZ}{{\tilde{Z}}}
\nc{\tx}{{\tilde{x}}}
\nc{\tbv}{{\tilde{\bv}}}
\nc{\tfP}{{\widetilde{\mathfrak{P}}}{}}
\nc{\tz}{{\tilde{\zeta}}}
\nc{\tmu}{{\tilde{\mu}}}

\nc{\urho}{\underline{\rho}}
\nc{\uB}{\underline{B}}
\nc{\uC}{{\underline{\mathbb{C}}}}
\nc{\ui}{\underline{i}}
\nc{\uj}{\underline{j}}
\nc{\ofP}{{\overline{\mathfrak{P}}}}
\nc{\oB}{{\overline{\mathcal{B}}}}
\nc{\og}{{\overline{\mathfrak{g}}}}
\nc{\oI}{{\overline{I}}}

\nc{\eps}{\varepsilon}
\nc{\hrho}{{\hat{\rho}}}

\nc{\one}{{\mathbf{1}}}
\nc{\two}{{\mathbf{t}}}

\nc{\Rep}{{\mathop{\operatorname{\rm Rep}}}}
\nc{\Tot}{{\mathop{\operatorname{\rm Tot}}}}
\nc{\Ker}{{\mathop{\operatorname{\rm Ker}}}}
\nc{\Hilb}{{\mathop{\operatorname{\rm Hilb}}}}
\nc{\End}{{\mathop{\operatorname{\rm End}}}}
\nc{\Ext}{{\mathop{\operatorname{\rm Ext}}}}
\nc{\CHom}{{\mathop{\operatorname{{\mathcal{H}}\it om}}}}
\nc{\GL}{{\mathop{\operatorname{\rm GL}}}}
\nc{\gr}{{\mathop{\operatorname{\rm gr}}}}
\nc{\Id}{{\mathop{\operatorname{\rm Id}}}}
\nc{\de}{{\mathop{\operatorname{\rm def}}}}
\nc{\length}{{\mathop{\operatorname{\rm length}}}}
\nc{\supp}{{\mathop{\operatorname{\rm supp}}}}

\nc{\Cliff}{{\mathsf{Cliff}}}
\nc{\Fl}{\on{Fl}}
\nc{\Fib}{{\mathsf{Fib}}}
\nc{\Coh}{{\on{Coh}}}
\nc{\QCoh}{{\on{QCoh}}}
\nc{\IndCoh}{{\on{IndCoh}}}
\nc{\FCoh}{{\mathsf{FCoh}}}

\nc{\reg}{{\text{\rm reg}}}

\nc{\cplus}{{\mathbf{C}_+}}
\nc{\cminus}{{\mathbf{C}_-}}
\nc{\cthree}{{\mathbf{C}_*}}
\nc{\Qbar}{{\bar{Q}}}
\nc\Eis{{\on{Eis}}}
\nc\Eisb{\ol\Eis{}}
\nc\Eisr{\on{Eis}^{rat}{}}
\nc\wh{\widehat}
\nc{\Def}{\on{Def_{\check{\fb}}(E)}}
\nc{\barZ}{\overline{Z}{}}
\nc{\barbarZ}{\overline{\barZ}{}}
\nc{\barpi}{\overline\pi}
\nc{\barbarpi}{\overline\barpi}
\nc{\barpip}{\overline\pi{}^+}
\nc{\barpim}{\overline\pi{}^-}

\nc{\fq}{\mathfrak q}

\nc{\fqb}{\ol{\fq}{}}
\nc{\fpb}{\ol{\fp}{}}
\nc{\fpr}{{\fp^{rat}}{}}
\nc{\fqr}{{\fq^{rat}}{}}

\nc{\hattimes}{\wh\otimes}

\nc{\bh}{{{\mathbf h}}}
\nc{\bk}{{{\mathbf k}}}
\nc{\bOmega}{{\overline{\Omega(\check \fn)}}}

\nc{\seq}[1]{\stackrel{#1}{\sim}}

\nc{\cT}{{\check{T}}}
\nc{\cG}{{\check{G}}}
\nc{\cM}{{\check{M}}}
\nc{\cB}{{\check{B}}}
\nc{\cP}{{\check{P}}}

\nc{\ct}{{\check{\mathfrak t}}}
\nc{\cg}{{\check{\fg}}}
\nc{\cb}{{\check{\fb}}}
\nc{\cn}{{\check{\fn}}}

\nc{\cLambda}{{\check\Lambda}}

\nc{\cla}{{\check\lambda}}
\nc{\cmu}{{\check\mu}}
\nc{\cnu}{{\check\nu}}
\nc{\ceta}{{\check\eta}}

\nc{\DefbE}{{\on{Def}_{\cB}(E_\cT)}}

\nc{\imathb}{{\ol{\imath}}}
\nc{\rlr}{\overset{\longrightarrow}{\underset{\longrightarrow}\longleftarrow}}

\nc{\oBun}{\overset{\circ}\Bun}
\nc{\oSht}{\overset{\circ}\Sht}
\nc{\LocSys}{\on{LocSys}}
\nc{\BunBbb}{\ol{\ol{Bun}}_B}
\nc{\BunBr}{\Bun_B^{rat}}
\nc{\BunBrp}{\Bun_B^{rat,polar}}
\nc{\BunTrp}{\Bun_T^{rat,polar}}
\nc{\BunNr}{\Bun_N^{rat}}
\nc{\BunNre}{\Bun_N^{enh,rat}}
\nc{\BunTr}{\Bun_T^{rat}}
\nc{\Vect}{\on{Vect}}
\nc{\Whit}{\on{Whit}}
\nc{\CTb}{\ol{\on{CT}}}
\nc{\Ran}{\on{Ran}}
\nc{\CTr}{\on{CT}^{rat}{}}
\nc\jmathr{\jmath^{rat}{}}
\nc{\ux}{\underline{x}}
\nc{\clambda}{{\check\lambda}}
\nc{\calpha}{{\check\alpha}}
\nc{\ind}{{\mathbf{ind}}}
\nc{\oblv}{{\mathbf{oblv}}}
\nc{\coeff}{\on{W-coeff}}
\nc{\Poinc}{\on{Poinc}}
\nc{\Dmod}{\on{D-mod}}
\nc{\dr}{\on{dR}}
\nc{\oCZ}{\overset{\circ}\CZ}
\nc{\KL}{\on{KL}}
\nc{\triv}{{\mathbf{triv}}}

\nc{\dgSch}{\on{DGSch}}
\nc{\Sch}{\on{Sch}}
\nc{\affdgSch}{\on{DGSch}^{\on{aff}}}
\nc{\affSch}{\on{Sch}^{\on{aff}}}
\nc{\Sing}{\on{Sing}}
\nc{\inftygroup}{\infty\on{-Grpd}}
\renc{\dr}{{\on{dr}}}
\nc\Maps{\on{Maps}}
\nc\Res{\on{Res}}
\nc\bMaps{\mathbf{Maps}}
\nc{\ul}{\underline}
\nc{\bNP}{\mathbf{N(P)}}
\nc{\ofc}{\overset{\circ}\fch}
\nc{\ppart}{(\!(t)\!)}
\nc{\qqart}{[\![t]\!]}
\nc{\crit}{\on{crit}}
\nc{\DGCat}{\on{DGCat}}
\nc{\Shv}{\on{Shv}}
\nc{\bDelta}{\mathbf{\Delta}}
\nc{\genB}{{\overset{\on{gen}}\to B}}
\nc{\genP}{{\underset{\on{gen}}\longrightarrow P}}
\nc{\genN}{{\underset{\on{gen}}\longrightarrow N}}
\nc{\semiinf}{{\frac{\infty}{2}+\bullet}}
\nc{\mmod}{\on{-}\mathbf{mod}}
\nc{\AdFr}{\on{Ad}_{\on{Frob}}}
\nc{\Frob}{\on{Frob}}
\nc{\Tr}{\on{Tr}}
\nc{\Sht}{\on{Sht}}
\nc{\sfe}{\mathsf{e}}
\nc{\tCat}{{2\on{-Cat}}}
\nc{\tIndCoh}{{2\on{-IndCoh}}}
\nc{\tQCoh}{{2\on{-QCoh}}}
\nc{\uShv}{\underline{\Shv}}
\nc{\qLisse}{\on{QLisse}}
\nc{\Nilp}{{\on{Nilp}}}
\nc{\sotimes}{\overset{!}\otimes} 
\nc{\AGCat}{\on{AGCat}}
\nc{\LS}{\on{LS}}
\nc{\Cat}{\on{-Cat}}
\nc{\bo}{\mathbf o}
\nc{\poinc}{\mathsf{poinc}}
\nc{\funct}{\mathsf{funct}}
\nc{\sFunct}{\mathsf{Funct}}

\begin{document}

\vskip1cm

\title[Tempered vs generic automorphic functions]{Tempered vs generic automorphic functions and \\
the canonical filtration on automorphic functions}

\author{Dennis Gaitsgory, Vincent Lafforgue and Sam Raskin} 

\begin{abstract}
We introduce and study the filtration on the space of automorphic functions (in the everywhere unramified situation for
the function field case) obtained by transferring the filtration on the spectral side of the classical Langlands conjecture,
induced by \emph{coherent singular support}. 

\medskip

We propose a number of conjectures that tie this filtration (which, by design, arises from the notion of \emph{cohomological support}) 
to a filtration on the space of $\BC$-valued automorphic functions that arises by considering the \emph{analytic} spectrum of Hecke
operators.

\end{abstract} 

\date{\today}

\maketitle

\tableofcontents

\section*{Introduction}

\ssec{What is this paper about?}

\sssec{}

Let $X$ be a smooth and complete curve and $G$ a reductive group over $\ol\BF_q$.
Let $\Bun_G$ denote the moduli stack of principle $G$-bundles on $X$. 

\medskip

The following version of the geometric Langlands conjecture was established in \cite{GR}:

\begin{equation} \label{e:GLC Intro}
\Shv_\Nilp(\Bun_G) \overset{\BL_G^{\on{restr}}}\simeq \IndCoh_\Nilp({}'\!\LS_\cG^{\on{restr}}),
\end{equation} 
where:

\medskip

\begin{itemize}

\item $\Shv_\Nilp(\Bun_G)\subset \Shv(\Bun_G)$ is the full subcategory consisting of objects whose
singular support is contained in the nilpotent cone (see \cite[Sect. 14.1.3]{AGKRRV1});

\medskip

\item $\cG$ is the Langlands dual of $G$, viewed as an algebraic group over $\ol\BQ_\ell$;

\medskip

\item $\LS_\cG^{\on{restr}}$ is the stack of $\cG$-local systems with \emph{resricted variation} on $X$
(see \cite[Sect. 1.4]{AGKRRV1}); it is a prestack locally of finite type over $\ol\BQ_\ell$

\medskip

\item $\IndCoh_{\Nilp}(\LS_\cG^{\on{restr}})\subset \IndCoh(\LS_\cG^{\on{restr}})$ is the full subcategory
consisting of objects whose \emph{coherent} singular support is contained in the nilpotent cone (see \cite[Sect. 21.2]{AGKRRV1});

\medskip

\item $'\!\LS_\cG^{\on{arithm}}\subset \LS_\cG^{\on{arithm}}$ is the union of \emph{some of} the connected
components of $\LS_\cG^{\on{arithm}}$;

\end{itemize}

\sssec{}

Assume now that both $X$ and $G$ are defined over $\BF_q$; for simplicity we will assume that the resulting $\BF_q$-form of $G$ is split. 
Then $\Bun_G$ also inherits this structure; in particular, it carries the geometric Frobenius endomorphism, to be denoted $\Frob_{\Bun}$. 

\medskip

In addition, the geometric Frobenius on $X$ gives rise to an automorphism, denoted $\Frob$ of $\LS_\cG^{\on{restr}}$.
The equivalence \eqref{e:GLC Intro} intertwines the actions of $(\Frob_{\Bun})_*$ on the left-hand side with the 
action of $\Frob^\IndCoh_*$ on the right-hand side.

\medskip

Taking the traces, we obtain an isomorphism of $\ol\BQ_\ell$-vector spaces:

\begin{equation} \label{e:Tr GLC Intro}
\Tr((\Frob_{\Bun})_*,\Shv_\Nilp(\Bun_G)) \simeq \Tr(\Frob^\IndCoh_*,\IndCoh_\Nilp({}'\!\LS_\cG^{\on{restr}})).
\end{equation}

\medskip

We now translate the isomorphism \eqref{e:Tr GLC Intro} into what can be called the \emph{classical} Langlands conjecture
(for function fields, unramified case).

\sssec{}

We first analyze the left-hand side in \eqref{e:Tr GLC Intro}. 

\medskip

The main theorem of \cite{AGKRRV3} says that the composition
\begin{equation} \label{e:Trace conj}
\Tr((\Frob_{\Bun})_*,\Shv_\Nilp(\Bun_G)) \to \Tr((\Frob_{\Bun})_*,\Shv(\Bun_G)) \overset{\on{LT}}\to \sFunct_c(\Bun_G(\BF_q),\ol\BQ_\ell)
\end{equation}
is an isomorphism, where:

\medskip

\begin{itemize}

\item The first arrow is induced by the embedding $\Shv_\Nilp(\Bun_G)\to \Shv(\Bun_G)$, which gives rise to a well-defined
map at the level of traces thanks to \cite[Theorem 1.1.7]{GR};

\medskip

\item $\Bun_G(\BF_q)\simeq (\Bun_G)^{\Frob_{\Bun}}$ is the groupoid of $\BF_q$-points of $\Bun_G$;

\medskip

\item $\sFunct_c(-,\ol\BQ_\ell)$ is the set of $\ol\BQ_\ell$-valued functions with compact supports on the set of
isomorphism classes of objects of a given groupoid;

\medskip

\item The second arrow is the local term map, which is a souped-up version of Grothendieck's sheaves-functions correspondence,
see \secref{sss:sheaf-functions}.

\end{itemize}

\sssec{}

We now consider the right-hand side of \eqref{e:Tr GLC Intro}.

\medskip

According to \cite[Sect. 6.4.13]{BLR}, the map
\begin{multline} \label{e:Trace IndCoh}
\Tr(\Frob^{\IndCoh}_*,\IndCoh_\Nilp(\LS_\cG^{\on{restr}}))\to
\Tr(\Frob^{\IndCoh}_*,\IndCoh(\LS_\cG^{\on{restr}})) \overset{\sim}\to  \\
\to \Gamma^{\IndCoh}(\LS_\cG^{\on{arithm}},\omega_{\LS_\cG^{\on{arithm}}})
\end{multline}
is also an isomorphism, where:

\medskip

\begin{itemize}

\item $\LS_\cG^{\on{arithm}}:=(\LS^{\on{restr}}_\cG)^{\on{Frob}}$ is the stack of $\cG$-local systems on $X$ equipped with a Weil structure; it is
a quasi-compact algebraic stack over $\ol\BQ_\ell$ (see \cite[Theorem 24.1.4]{AGKRRV1});

\medskip

\item $\omega_{(-)}$ denotes the dualizing sheaf, which is defined as an object of $\IndCoh(-)$
on any prestack locally almost of finite type; 

\medskip

\item $\Gamma^{\IndCoh}(-,-)$ is the functor of $\IndCoh$-global sections, which is defined as a functor on $\IndCoh(-)$
for quasi-compact formal (derived) algebraic stacks; 

\medskip

\item The second arrow is the isomorphism of (see \cite[Sect. 24.6]{AGKRRV1}); it is valid for any formal algebraic stack. 

\end{itemize}

An isomorphism parallel to \eqref{e:Trace IndCoh} holds for $\LS_\cG^{\on{restr}}$ replaced by $'\!\LS_\cG^{\on{restr}}$. 

\medskip

Combining \eqref{e:Tr GLC Intro}, \eqref{e:Trace conj} and \eqref{e:Trace IndCoh}, we obtain an isomorphism 
\begin{equation} \label{e:classical Intro}
\sFunct_c(\Bun_G(\BF_q),\ol\BQ_\ell) \overset{\sL_G}\simeq \Gamma^{\IndCoh}({}'\!\LS_\cG^{\on{arithm}},\omega_{'\!\LS_\cG^{\on{arithm}}}).
\end{equation}

This isomorphism is the starting point of this paper. 

\sssec{}

Now, to every closed $\on{Ad}$-invariant subset $Y$ of the nilpotent cone of $\cg$ (which is automatically conical), 
we can associate a full subcategory
$$\IndCoh_Y(\LS_\cG^{\on{restr}})\subset \IndCoh_\Nilp(\LS_\cG^{\on{restr}}),$$
stable under the action of $\Frob^\IndCoh_*$, see \cite[Sect. 21.1]{AGKRRV1}.

\medskip

The above embedding induces a map
\begin{multline} \label{e:Y Intro}
\Tr(\Frob^{\IndCoh}_*,\IndCoh_Y(\LS_\cG^{\on{restr}}))\to
\Tr(\Frob^{\IndCoh}_*,\IndCoh_\Nilp(\LS_\cG^{\on{restr}}))\simeq \\
\simeq \Gamma^{\IndCoh}(\LS_\cG^{\on{arithm}},\omega_{\LS_\cG^{\on{arithm}}}).
\end{multline}

In \secref{ss:filtr spec} we formulate a conjecture\footnote{This is our \conjref{c:gr}.}  
(which in in fact a theorem-in-progress by the authors of the present paper) that
implies that for any $Y$, the left-hand side in \eqref{e:Y Intro} is a \emph{classical} vector space (i.e., is concentrated
in cohomological degree $0$) and the map \eqref{e:Y Intro} (of classical vector spaces) is injective. 

\medskip

For the duration of this introduction we will assume the validity of this conjecture. It follows formally that a parallel assertion 
holds for $\LS_\cG^{\on{restr}}$ replaced by $'\!\LS_\cG^{\on{restr}}$. 

\sssec{}

Thus, we obtain that 
\begin{equation} \label{e:filtr Intro}
Y\rightsquigarrow \Tr(\Frob^{\IndCoh}_*,\IndCoh_Y({}'\!\LS_\cG^{\on{restr}}))
\end{equation}
defines a filtration on $\Gamma^{\IndCoh}({}'\!\LS_\cG^{\on{arithm}},\omega_{'\!\LS_\cG^{\on{arithm}}})$,
indexed by the poset of closed $\on{Ad}$-invariant subsets of the nilpotent cone of $\cg$.

\medskip

The term corresponding to $Y=\{0\}$ is the (image of)
$$\Gamma({}'\!\LS_\cG^{\on{arithm}},\CO_{{}'\!\LS_\cG^{\on{arithm}}})\simeq
\Tr(\Frob_*,\QCoh({}'\!\LS_\cG^{\on{restr}}))=\Tr(\Frob^{\IndCoh}_*,\IndCoh_{\{0\}}({}'\!\LS_\cG^{\on{restr}}))$$
under
\begin{multline} 
\Tr(\Frob^{\IndCoh}_*,\IndCoh_{\{0\}}({}'\!\LS_\cG^{\on{restr}}))\to \Tr(\Frob^{\IndCoh}_*,\IndCoh_\Nilp({}'\!\LS_\cG^{\on{restr}})\overset{\sim}\to \\
\to \Tr(\Frob^{\IndCoh}_*,\IndCoh({}'\!\LS_\cG^{\on{restr}}) \simeq \Gamma^{\IndCoh}({}'\!\LS_\cG^{\on{arithm}},\omega_{'\!\LS_\cG^{\on{arithm}}}).
\end{multline}
induced by the embedding
$$\Xi_{{}'\!\LS_\cG^{\on{restr}}}:\QCoh({}'\!\LS_\cG^{\on{restr}})\to \IndCoh({}'\!\LS_\cG^{\on{restr}}).$$

\sssec{}

Let us transfer the filtration \eqref{e:filtr Intro} to the left-hand side of the isomorphism \eqref{e:classical Intro}. We obtain a filtration
\begin{equation} \label{e:filtr autom Intro}
Y \rightsquigarrow \sFunct_c(\Bun_G(\BF_q),\ol\BQ_\ell)_Y
\end{equation}
on the space $\sFunct_c(\Bun_G(\BF_q),\ol\BQ_\ell)$ of automorphic functions.

\medskip

The goal of this paper is to initiate the study this filtration. 

\ssec{What is done in this paper?}

This paper consists of two parts that are logically almost disconnected from one another. 

\sssec{}

In the first part we study the term
of the filtration \eqref{e:filtr autom Intro} corresponding to $Y=\{0\}$. We denote it by
$$\sFunct_c(\Bun_G(\BF_q),\ol\BQ_\ell)_{\on{temp}}:=\sFunct_c(\Bun_G(\BF_q),\ol\BQ_\ell)_{\{0\}}.$$

It can be defined (without assuming \conjref{c:gr}) as the image under $\sL_G$ of the image of the map
\begin{multline} \label{e:map in tempered part}
\Gamma({}'\!\LS_\cG^{\on{arithm}},\CO_{{}'\!\LS_\cG^{\on{arithm}}})\simeq
\Tr(\Frob^{\IndCoh}_*,\IndCoh_{\{0\}}({}'\!\LS_\cG^{\on{restr}})) \to \\
\to \Tr(\Frob^{\IndCoh}_*,\IndCoh({}'\!\LS_\cG^{\on{restr}}) \simeq \Gamma^{\IndCoh}({}'\!\LS_\cG^{\on{arithm}},\omega_{'\!\LS_\cG^{\on{arithm}}}),
\end{multline}
since one knows a priori that both the left-hand side and the right-hand side are classical
(the former thanks to \cite[Corollary 3.7.3]{GaLiRe}, and the latter thanks to \eqref{e:classical Intro}).

\sssec{}

We prove (see \corref{c:temp=nondeg}) that the subspace 
$$\sFunct_c(\Bun_G(\BF_q),\ol\BQ_\ell)_{\on{temp}}\subset \sFunct_c(\Bun_G(\BF_q),\ol\BQ_\ell)$$
equals 
$$\sFunct_c(\Bun_G(\BF_q),\ol\BQ_\ell)_{\on{non-degen}}\subset \sFunct_c(\Bun_G(\BF_q),\ol\BQ_\ell),$$
where the latter is the submodule of $\sFunct_c(\Bun_G(\BF_q),\ol\BQ_\ell)$ generated by the \emph{vacuum Poincar\'e function}
$\poinc^{\on{Vac}}_!$ (see \secref{sss:vac poinc funct}) under the action of the \emph{excursion algebra} 
$$\on{Exc}(X,\cG):=\Gamma(\LS_\cG^{\on{arithm}},\CO_{\LS_\cG^{\on{arithm}}})$$
(see \secref{sss:act of exc}). 

\sssec{}

In fact, we prove that the isomorphism $\sL_G$ sends the element 
\begin{equation} \label{e:poinc funct Intro}
\poinc^{\on{Vac}}_!\in  \sFunct_c(\Bun_G(\BF_q),\ol\BQ_\ell)
\end{equation} 
to the image of
$$1\in \Gamma({}'\!\LS_\cG^{\on{arithm}},\CO_{{}'\!\LS_\cG^{\on{arithm}}})$$
under the map \eqref{e:map in tempered part}.

\medskip

The latter is the main result of the present paper, which appears as \thmref{t:main}. 

\sssec{}

Let us comment briefly on how we prove \thmref{t:main}. 

\medskip

At the first glance, this theorem should be easy to prove: we understand the equivalence \eqref{e:GLC Intro} rather
explicitly, and $\poinc^{\on{Vac}}_!$ is obtained as the \emph{class} of a Weil sheaf, namely, the object
$$\on{Poinc}^{\on{Vac}}_!\in \Shv(\Bun_G)^c,$$
which is involved in the construction of the functor $\BL_G$. 

\medskip

However, we cannot apply the functor $\BL_G$ to $\on{Poinc}^{\on{Vac}}_!$ directly because it does not lie in 
$\Shv_\Nilp(\Bun_G)$. 

\medskip

So the proof of \thmref{t:main} consists of manipulating $\on{Poinc}^{\on{Vac}}_!$ to ``make it belong" to
$\Shv_\Nilp(\Bun_G)$. We do so by applying a version of Beilinson's projector for a large enough 
algebraic substack $\CZ_a\subset \LS_\cG^{\on{restr}}$, see \thmref{t:Z a}. 

\medskip

In fact, the latter theorem answers the following general question: given an object $\CF\in \Shv(\Bun_G)^c$
equipped with a Weil structure, how to find an object $\CF'\in \Shv_\Nilp(\Bun_G)^c$ (also equipped with a Weil
structure) so that\footnote{More precisely, \thmref{t:Z a} makes the equality below hold up to a universal scalar.}
\begin{equation} \label{e:replace F}
\funct(\CF)=\funct(\CF').
\end{equation} 

\sssec{}

The second part of the paper consists of conjectures pertaining the properties of the filtration 
\eqref{e:filtr autom Intro}. There are three \emph{types} of conjectures that we propose:

\medskip

\begin{enumerate}

\item Rationality conjectures;

\medskip

\item Properties of \eqref{e:filtr autom Intro} with respect to the action of $\on{Exc}(X,\cG)$ and the Hecke algebra(s);

\medskip

\item Interaction of the two.

\end{enumerate}

\sssec{}

Note that the space $\sFunct_c(\Bun_G(\BF_q),\ol\BQ_\ell)$ comes equipped with a natural rational structure:
$$\sFunct_c(\Bun_G(\BF_q),\ol\BQ_\ell)\simeq \ol\BQ_\ell\underset{\BQ}\otimes \sFunct_c(\Bun_G(\BF_q),\BQ).$$

Our key rationality conjecture (made jointly with D.~Kazhdan), \conjref{c:filtr rat}, says that the filtration \eqref{e:filtr autom Intro} is defined
over $\BQ$.

\sssec{}

Another rationality conjecture (\conjref{c:exc rat}) says that the action of the excursion algebra $\on{Exc}(X,G)$ 
(which is equipped with a rational structure thanks to \cite{GaLiRe}) on $\sFunct_c(\Bun_G(\BF_q),\ol\BQ_\ell)$ ic
compatible with the rational structures.

\sssec{}

Note that Conjectures \conjref{c:filtr rat} and \conjref{c:exc rat} have a common consequence, namely, that the subspace
$$\sFunct_c(\Bun_G(\BF_q),\ol\BQ_\ell)_{\on{temp}}\subset \sFunct_c(\Bun_G(\BF_q),\ol\BQ_\ell)$$
is defined over $\BQ$. This is our \conjref{c:temp rat}.

\medskip

The deduction  \conjref{c:exc rat} $\Rightarrow$ \conjref{c:temp rat} uses the equality 
$$\sFunct_c(\Bun_G(\BF_q),\ol\BQ_\ell)_{\on{temp}}=\sFunct_c(\Bun_G(\BF_q),\ol\BQ_\ell)_{\on{non-degen}}$$
and the fact that the element \eqref{e:poinc funct Intro} is rational. 

\ssec{Spectral conjectures}

We propose a whole web of conjectures pertaining to the spectral properties of the action of $\on{Exc}(X,\cG)$
on $\sFunct_{\on{cusp}}(\Bun_G(\BF_q),\ol\BQ_\ell)$. This entire discussion can be viewed as 
an elaboration of \cite{Ras}. 

\sssec{}

Our key spectral conjecture is \conjref{c:Ramanujan Q ell}, which is an $\ell$-adic version of the Ramanujan-Arthur conjecture.
It says that the subspace of
$$\sFunct_{\on{cusp}}(\Bun_G(\BF_q),\ol\BQ_\ell)\subset \sFunct_c(\Bun_G(\BF_q),\ol\BQ_\ell)$$
of cuspidal autmorphic functions splits as a direct sum
$$\underset{\bO}\oplus\, \sFunct_{\on{cusp}}(\Bun_G(\BF_q),\ol\BQ_\ell)_\bO,$$
where $\bO$ runs over the set of nilpotent orbits in $\cg$. 

\medskip

The direct summand $\sFunct_{\on{cusp}}(\Bun_G(\BF_q),\ol\BQ_\ell)_\bO$ is characterized by the property that
the eigenvalues of the Hecke operators acting on it are given by Weil numbers of weight $0$, up to a correction given by $\bO$,
see \secref{sss:Weil as eigen} for the precise formulation. 

\sssec{}

We note that \conjref{c:Ramanujan Q ell} implies its rational version, \conjref{c:Ramanujan Q}, which in turn implies
a version of the Ramanujan-Arthur conjecture with $\BC$-coefficients, \conjref{c:Ram-Pet}. 

\sssec{}

In its turn, \conjref{c:Ramanujan Q ell} follows from a more general statement, given by \conjref{c:support O discr Q ell}:

\medskip 

For a given nilpotent orbit $\bO$, consider the subspace
\begin{multline*}
\sFunct_{\on{discr}}(\Bun_G(\BF_q),\ol\BQ_\ell)_{\bO}\subset 
\sFunct_c(\Bun_G(\BF_q),\ol\BQ_\ell)_{\bO}:=\\
=\sFunct_c(\Bun_G(\BF_q),\ol\BQ_\ell)_{\ol\bO}/\sFunct_c(\Bun_G(\BF_q),\ol\BQ_\ell)_{\ol\bO-\bO}
\end{multline*}
that consist of elements that are \emph{locally finite} with respect to the action of the Hecke operators.

\medskip

Our \conjref{c:support O discr Q ell} says that the
eigenvalues of the Hecke operators acting on the space $\sFunct_{\on{discr}}(\Bun_G(\BF_q),\ol\BQ_\ell)_{\bO}$ 
are given by \emph{mock} Weil numbers\footnote{See \cite[Sect. 2.1.2]{GaLiRe} for what this means. In fact, according to \conjref{c:mock Weil}, these
are actual Weil numbers.} of weight $0$,
up to a correction given by $\bO$.

\medskip

In fact, we propose an even more precise conjecture, \conjref{c:locally finite O}, 
which says that the support of $\sFunct_{\on{discr}}(\Bun_G(\BF_q),\ol\BQ_\ell)_{\bO}$
in
$$\LS^{\on{arithm,coarse}}_\cG:=\Spec(\on{Exc}(X,\cG))$$
consists of \emph{elliptic Arthur parameters}, whose $SL_2$ component corresponds to $\bO$.

\medskip

It appears, however, that \conjref{c:locally finite O} is within reach. It is the subject of the work-in-progress \cite{GaLaRa}. As was metioned
above, once this conjecture becomes as theorem, so will the other conjectures mentioned previously in the current subsection. 

\sssec{}

Let us now assume the rationality conjecture, i.e., \conjref{c:filtr rat}. This will allow us to connect the filtration \eqref{e:filtr autom Intro}
to a decomposition of $L^2(\Bun_G(\BF_q))$, conjectured by J.~Arthur. 

\medskip

On the one hand, by \conjref{c:filtr rat}, the filtration \eqref{e:filtr autom Intro} induces a filtration 
$$Y\rightsquigarrow \sFunct_c(\Bun_G(\BF_q),\BQ)_Y$$
on $\sFunct_c(\Bun_G(\BF_q),\BQ)$,
which in turn induces a filtration 
$$Y\rightsquigarrow \sFunct_c(\Bun_G(\BF_q),\BC)_Y$$
on $\sFunct_c(\Bun_G(\BF_q),\BC)$.

\medskip

On the other hand, Arthur's conjecture predicts a decomposition of $L^2(\Bun_G(\BF_q))$ as
\begin{equation} \label{e:Arth Intro}
L^2(\Bun_G(\BF_q))\simeq \underset{\bO}\oplus\, L^2(\Bun_G(\BF_q))_\bO,
\end{equation}
where each direct summand is characterized by its spectral support with respect to the Hecke algebra,
see \conjref{c:Arthur} for the precise statement. (Note \eqref{e:Arth Intro} also implies the Ramanujan-Arthur conjecture,
\conjref{c:Ram-Pet}.) 

\medskip

Jointly with D.~Kazhdan, we propose \conjref{c:Arth alg}, which says that for a given closed Ad-invariant subset $Y$ of the nilpotent 
cone of $\cG$, the closure of $\sFunct_c(\Bun_G(\BF_q),\BC)_Y$ in $L^2(\Bun_G(\BF_q))$ equals 
$$\underset{\bO\,|\, \bO\subset Y}\oplus\, L^2(\Bun_G(\BF_q))_\bO.$$

\medskip

We note that \conjref{c:Arth alg} implies \conjref{c:support O discr Q ell} mentioned above. 

\sssec{}

A part of our system of conjectures can be depicted by the following diagram:

$$
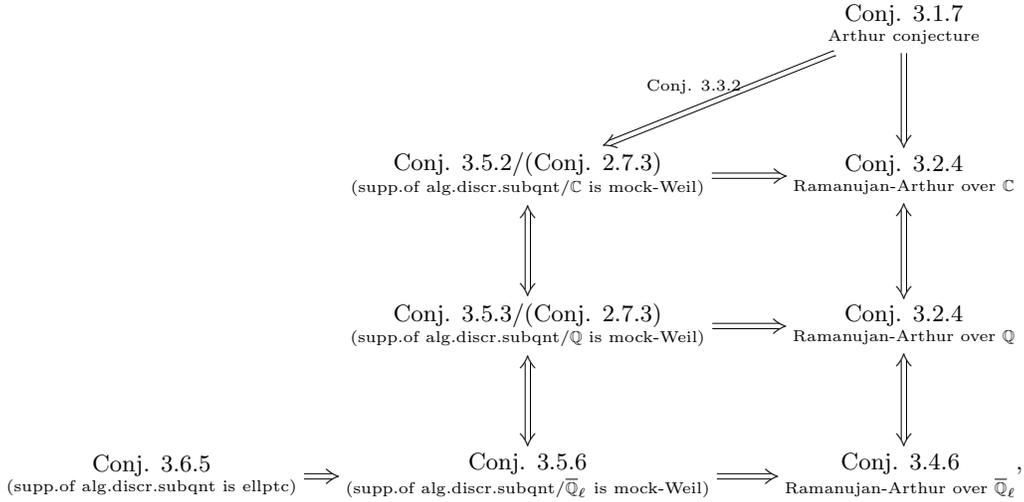
\xy
(-50,0)*+{\underset{\text{(supp.of alg.discr.subqnt is ellptc)}}{\text{Conj. \ref{c:locally finite O}}}}="A";
(0,0)*+{\underset{\text{(supp.of alg.discr.subqnt}/\ol\BQ_\ell \text{ is mock-Weil)}}{\text{Conj. \ref{c:support O discr Q ell}}}}="B";
(50,0)*+{\underset{\text{Ramanujan-Arthur over }\ol\BQ_\ell}{\text{Conj. \ref{c:Ramanujan Q ell}}},}="C";
(0,20)*+{\underset{\text{(supp.of alg.discr.subqnt}/\BQ \text{ is mock-Weil)}}{\text{Conj. \ref{c:support O discr Q}/(Conj. \ref{c:filtr rat})}}}="BQ";
(50,20)*+{\underset{\text{Ramanujan-Arthur over }\BQ}{\text{Conj. \ref{c:Ram-Pet}}}}="CQ";
(0,40)*+{\underset{\text{(supp.of alg.discr.subqnt}/\BC \text{ is mock-Weil)}}{\text{Conj. \ref{c:support O discr C}/(Conj. \ref{c:filtr rat})}}}="BC";
(50,40)*+{\underset{\text{Ramanujan-Arthur over }\BC}{\text{Conj. \ref{c:Ram-Pet}}}}="CC";
(50,60)*+{\underset{\text{Arthur conjecture}}{\text{Conj. \ref{c:Arthur}}}}="D";
{\ar@{=>} "A";"B"};
{\ar@{=>} "B";"C"};
{\ar@{=>} "BQ";"CQ"};
{\ar@{=>} "BC";"CC"};
{\ar@{<=>} "B";"BQ"};
{\ar@{<=>} "BQ";"BC"};
{\ar@{<=>} "C";"CQ"};
{\ar@{<=>} "CQ";"CC"};
{\ar@{=>} "D";"CC"};
{\ar@{=>}_{\text{Conj. \ref{c:Arth alg}}} "D";"BC"};
\endxy
$$
where:

\begin{itemize}

\item Conj. X/(Conj. Y) means that Conjecture X makes sense once we accept Conjecture Y;

\medskip

\item $\overset{\text{Conj. Z}}\Leftarrow$ means that the implication uses Conjecture Z. 

\end{itemize} 

\ssec{Plan of the paper}

\sssec{}

In \secref{s:poinc} we state and prove \thmref{t:main}, mentioned above, which describes the image of the function
$\poinc^{\on{Vac}}_!$ under the isomorphism $\sL_G$.

\medskip

The proof uses a certain technical assertion, \thmref{t:present}, which we deal with in \secref{s:proof of present}.  

\sssec{}

In \secref{s:filtr} we introduce the filtration by nilpotent orbits, first on the spectral side, and then on the automorphic
side of the isomorphism \eqref{e:classical Intro}. 

\medskip

The existence of this filtration relies on \conjref{c:gr}, which is work-in-progress by the authors of the present paper.
We assume this conjecture throughout the paper. 

\medskip

Using \thmref{t:main}, we prove that the term of this filtration corresponding to the $0$-orbit equals the subspace of
non-degenerate functions, i.e., the subspace of functions generated by $\poinc^{\on{Vac}}_!$ under the action of the
excursion algebra. 

\medskip

Finally, we formulate our rationality conjectures, \ref{c:filtr rat} and \ref{c:exc rat}. 

\sssec{}

In \secref{s:Arth} we present our maze of conjectures about the spectral properties 
of the action of $\on{Exc}(X,\cG)$ on $\sFunct_{\on{cusp}}(\Bun_G(\BF_q),\ol\BQ_\ell)$,
and the corresponding system of logical implications. 

\ssec{Conventions and notation}

The conventions and notations in this paper follow those adopted in \cite{GR}. 

\sssec{} We fix $X$ a smooth and complete curve and $G$ a reductive group over $\ol\BF_q$, but 
both assumed to be defined over $\BF_q$, so that they carry an action of the geometric Frobenius.
For simplicity, we will assume that the $\BF_q$-form of $G$ is split.

\medskip

We let $\cG$ be the Langlands dual of $G$, which we regard as a reductive group over the field
of coefficients $\ol\BQ_\ell$. 

\sssec{} 

We will deal with two kinds of algebraic geometry. One is taking place over the \emph{ground field}, which
is $\ol\BF_q$, and we will be interested in the category of constructible sheaves on schemes or algebraic stacks over 
this $\ol\BQ_\ell$. As this category is not sensitive to infinitesimals, we can stay within the realm of \emph{classical}
(as opposed to derived) algebraic geometry. 

\medskip

The second kind of algebraic geometry is taking place over the \emph{field of coefficients}, which is $\ol\BQ_\ell$. 
We will be dealing with coherent (or ind-coherent) sheaves on prestacks over $\ol\BQ_\ell$. As these categories
are sensitive to the derived structure, we have to place ourselves in the context of \emph{derived} algebraic geometry
(but in the simplified context\footnote{Foundations of this theory can be found, e.g., in \cite[Chapter 2]{GaRo1}.},
since $\ol\BQ_\ell$ has characteristic $0$). 

\sssec{}

We will be dealing with \emph{higher algebra} over $\ol\BQ_\ell$, in which the basic object is the $\infty$-category
of $\ol\BQ_\ell$-linear DG-categories, denoted $\DGCat_{\ol\BQ_\ell}$ (see, e.g.,  \cite[Chapter 1, Sect. 10]{GaRo1}). 
This category carries a symmetric monoidal structure, given by the Lurie tensor product, for which the unit is the
category $\Vect_{\ol\BQ_\ell}$ of chain complexes of $\ol\BQ_\ell$-vector spaces. 

\medskip

In particular, for a \emph{dualizable} object $\bC\in \DGCat_{\ol\BQ_\ell}$ equipped with an endofunctor $F$, we have a
well-defined object
$$\Tr(F,\bC)\in \Vect_{\ol\BQ_\ell}.$$

\sssec{}

As was mentioned above, the main source of objects in $\DGCat_{\ol\BQ_\ell}$ in this paper is provided by the construction
$$\CY \rightsquigarrow \Shv(\CY),$$
where $\CY$ is a scheme (or, more generally, a prestack locally of finite type) over $\ol\BF_q$; see, e.g., 
\cite[Sects 1.1 and F.1]{AGKRRV1}

\ssec{Acknowledgements}

Many of the conjectures in Sects. \ref{s:filtr} and \ref{s:Arth} were proposed in the course of discussions with D.~Kazhdan. 
We are grateful to him for sharing his ideas and for allowing us to make them public. 

\medskip

We wish to thank A.~Eteve, M.~Harris, K.~Lin, A.~Okounkov, W.~Reeves and P.~Scholze for some very valuable discussions. 

\section{The image of the Poincar\'e function} \label{s:poinc}

In this section we formulate and prove the main result of this paper, \thmref{t:main}, that
describes the image of the vacuum Poincar\'e (a.k.a., basic Whittaker) function under the
\emph{classical} Langlands isomorphism. 

\ssec{Sheaves-functions dictionary} \label{ss:sheaf-functions}

\sssec{} 

Let $\CY$ be a quasi-compact algebraic stack over $\ol\BF_q$, but defined over $\BF_q$, so that it carries the geometric 
Frobenius, denoted $\Frob_\CY$. 

\medskip

Note that the functor $(\Frob_\CY)_*$ is continuous 
(see, e.g., \cite[Lemma 5.1.3]{GRV}\footnote{The potential issue could have been that $\CY$ is a stack, the morphism
$\Frob_\CY$ is not schematic.}).

\sssec{}

By a \emph{weak Weil sheaf} on $\CY$ we will mean an object
$\CF\in \Shv(\CY)$, equipped with a map
$$\alpha:\CF\to (\Frob_\CY)_*(\CY).$$

Note that a usual Weil sheaf is a similar data, where the map $\alpha$ is required to be an isomorphism.

\sssec{}

Assume that $\CF$ belongs to $\Shv(\CY)^c$. Then, following Grothendieck
 to the data of $(\CF,\alpha)$ one associates an element
$$\funct(\CF,\alpha)\in \sFunct(\CF(\BF_q),\ol\BQ_\ell).$$

When no ambiguity is likely to occur, we will simply write $\funct(\CF)$. 

\medskip

There are (at least) three ways to think about this construction.

\sssec{}

For $y\in \CF(\BF_q)$, let $i_y$ denote the corresponding morphism
$$\Spec(\BF_q) \to \CY.$$

By adjunction, $\alpha$ gives rise to a map
$$\alpha':\Frob_\CY^*(\CF)\to \CF.$$

Pulling back by means of $i_y$, and using the fact that $\Frob_\CY\circ i_y\simeq i_y$ from $\alpha'$ we obtain a map, to be denoted $\alpha'_y$
$$\CF_y:=i_y^*(\CF)\simeq i_y^*\circ (\Frob_\CY)^* \overset{\alpha'}\to i_y^*(\CF)=\CF_y.$$

The value of $\funct(\CF)$ at $y$ is by definition
$$\Tr(\alpha'_y,\CF_y)\in \ol\BQ_\ell.$$

This is Grothendieck's original definition.

\sssec{}

The above definition can be rewritten (almost tautologically as follows):

\medskip

The functor 
$$i_y^*:\Shv(\CY)\to \Vect_{\ol\BQ_\ell}$$ satisfies
$$i_y^* \simeq i_y^*\circ (\Frob_\CY)_*,$$
and admits a colimit-preserving right adjoint. 

\medskip

Hence, by \cite[Sect. 3.4.1]{GKRV}, $i_y^*$ gives rise to a map
$$\Tr((\Frob_\CY)_*,\Shv(\CY))\to \Tr(\on{Id},\Vect_{\ol\BQ_\ell})\simeq \ol\BQ_\ell,$$
to be denoted $\on{LT}^{\on{naive}}_y$.

\medskip

The maps $\on{LT}^{\on{naive}}_y$ combine to a map
$$\on{LT}_\CY^{\on{naive}}:\Tr((\Frob_\CY)_*,\Shv(\CY))\to \sFunct(\CF(\BF_q),\ol\BQ_\ell).$$

\medskip

Now, to $(\CF,\alpha)$ as above with $\CF\in \Shv(\CY)^c$ one attaches an element
$$\on{cl}(\CF,\alpha)\in \Tr((\Frob_\CY)_*,\Shv(\CY)),$$
see \cite[Sect. 3.4.3]{GKRV}.

\medskip

Unwinding, we obtain:
$$\funct(\CF)=\on{LT}_\CY^{\on{naive}}(\CF,\alpha).$$

\sssec{}

Finally, following \cite[Sect. 22.2]{AGKRRV1}, for any endomorphism $\phi$ on $\CY$, there is a canonically defined map
$$\Tr(\phi_*,\Shv(\CY))\overset{\on{LT}_\phi}\to \on{C}^\cdot(\CY^\phi,\omega_{\CY^\phi}).$$

Take $\phi=\Frob_\CY$ and let us notice that 
$$\CY^{\Frob_\CY}\simeq \underset{y\in \CY(\BF_q)}\sqcup\, \on{pt}/\on{Aut}_y(\BF_q).$$

Hence, 
$$\on{C}^\cdot(\CY^{\Frob_\CY},\omega_{\CY^{\Frob_\CY}})\simeq \sFunct(\CF(\BF_q),\ol\BQ_\ell).$$

Denote the resulting map $\on{LT}_{\Frob_\CY}$ 
$$\Tr((\Frob_\CY)_*,\Shv(\CY))\to \sFunct(\CF(\BF_q),\ol\BQ_\ell)$$
by $\on{LT}_\CY^{\on{true}}$.

\sssec{} \label{sss:sheaf-functions} 

The main result of the paper \cite{GV} says that we have
$$\on{LT}_\CY^{\on{naive}}=\on{LT}_\CY^{\on{true}}.$$

Hence, we can talk about a map
$$\on{LT}_\CY:\Tr((\Frob_\CY)_*,\Shv(\CY))\to \sFunct(\CY(\BF_q),\ol\BQ_\ell),$$
omitting the superscript.

\sssec{}

The above discussion generalizes to algebraic stacks that are not necessarily quasi-compact,
see \cite[Sect. 22.2.11]{AGKRRV1}, with the difference that now the target is the space
$$\sFunct_c(\CY(\BF_q),\ol\BQ_\ell)$$
of functions with finite support. 

\sssec{}  \label{sss:sheaf-functions funct}

The following two features of the map $\on{LT}_\CY$ are obvious from its definition as $\on{LT}_\CY^{\on{naive}}$:

\begin{itemize}

\item For a map $f:\CY_1\to \CY_2$ and a weak Weil sheaf $(\CF,\alpha)$ on $\CY_2$ with $\CF$ compact, 
we have
$$\funct(f^*(\CF))=f^*(\funct(\CF)),$$
where:

\begin{itemize}

\item In the right-hand side, $f^*$ denotes the pullback map
$$\sFunct(\CY_2(\BF_q),\ol\BQ_\ell)\to \sFunct(\CY_1(\BF_q),\ol\BQ_\ell);$$

\item $f^*(\CF)$ is equipped with a weak Weil structure induced by that on $\CF$;

\end{itemize}

\medskip

\item For $(\CF_1,\alpha_1)$ and $(\CF_2,\alpha_2)$ two Weil sheaves on $\CY$, 
$$\funct(\CF_1\otimes \CF_2)=\funct(\CF_1)\cdot \funct(\CF_2),$$
where $\CF_1\otimes \CF_2$ is equipped with a weak Weil structure, induced by $\alpha_1$ 
and $\alpha_2$.

\end{itemize} 

\medskip

The next property (see \cite{GV}) is the Grothendieck-Lefschetz formula:

\medskip

\begin{itemize}

\item For a map $f:\CY_1\to \CY_2$ and a weak Weil sheaf $(\CF,\alpha)$ on $\CY_1$ with $\CF$ compact, 
we have
$$\funct(f_!(\CF))=f_!(\funct*\CY),$$
where:

\begin{itemize}

\item In the right-hand side, 
$$f:\sFunct(\CY_1(\BF_q),\ol\BQ_\ell)\to \sFunct(\CY_2(\BF_q),\ol\BQ_\ell)$$
denotes the operation of summation along the fibers of the map of groupoids
$$(\CY_1)^{\Frob_{\CY_1}}\to (\CY_2)^{\Frob_{\CY_2}};$$

\item $f_!(\CF)$ is equipped with a weak Weil structure induced by that on $\CF$. 

\end{itemize}

\end{itemize} 

\ssec{The vacuum Poincar\'e function}

\sssec{} \label{sss:vac poinc sheaf}

Consider the object 
$$\on{Poinc}^{\on{Vac}}_!\in \Shv(\Bun_G)^c,$$
see \cite[Sect. 1.1.1]{GR}; the definition is recalled below. 

\medskip

It is equipped with a naturally defined Weil structure. Set  
$$\poinc^{\on{Vac}}_!:=\funct(\on{Poinc}^{\on{Vac}}_!)\in \sFunct_c(\Bun_G(\BF_q),\ol\BQ_\ell).$$

\sssec{} \label{sss:vac poinc funct}

Explicitly, the function $\poinc^{\on{Vac}}_!$ can be described as follows. Consider the stack
$$\Bun_{N,\rho(\omega_X)}:=\Bun_B\underset{\Bun_T}\times \{\rho(\omega_X)\}.$$

It is equipped with a map 
$$\chi:\Bun_{N,\rho(\omega_X)}\to \BG_a,$$
see \cite[Sect. 1.3.6]{GLC1}. 

\medskip

Let $\fp$ denote the natural projection
$$\Bun_{N,\rho(\omega_X)}\to \Bun_G.$$

Then
$$\on{Poinc}^{\on{Vac}}_!:=\fp_!\circ \chi^*(\on{AS}),$$
where $\on{AS}\in \Shv(\BG_a)$ is the Artin-Schreier sheaf corresponding to a 
(chosen once and for all) character $\mathsf{exp}:\BF_q\to \ol\BQ_\ell^\times$, so that
$$\funct(\on{AS})=\mathsf{exp}.$$

\medskip

In its turn
$$\poinc^{\on{Vac}}_!=\fp_!\circ \chi^*(\mathsf{exp}),$$
see \secref{sss:sheaf-functions funct} for the symbols $(-)_!$ and $(-)^*$, respectively.

\sssec{}

Equivalently, the function $\poinc^{\on{Vac}}_!$ ``represents" the function of 
\emph{vacuum Whittaker coefficient}:

\medskip

For $\sff\in \sFunct_c(\Bun_G(\BF_q),\ol\BQ_\ell)$,
$$\langle \poinc^{\on{Vac}}_!,\sff \rangle_{\Bun_G(\BF_q)}=\underset{\Bun_{N,\rho(\omega_X)}(\BF_q)}\int\, \fp^*(\sff)\cdot \chi^*(\mathsf{exp}),$$
where for a quasi-compact algebraic stack $\CY$

\medskip 

\begin{itemize}

\item For $\sg\in \sFunct_c(\CY(\BF_q),\ol\BQ_\ell)$, 
$$\underset{\CY(\BF_q)}\int\, \sg:=\underset{y\in \CY(\BF_q)/\sim}\Sigma\, \sg(y)\cdot \frac{1}{|\on{Aut}(y)|};$$

\item For $\sg_1,\sg_2\in \sFunct_c(\CY(\BF_q),\ol\BQ_\ell)$,
$$\langle \sg_1,\sg_2\rangle_{\CY(\BF_q)}:=\underset{\CY(\BF_q)}\int\, \sg_1\cdot \sg_2.$$

\end{itemize} 

\ssec{The statement}

\sssec{}

Recall that according to \cite[Theorem 1.3.9(a)]{GR}, there is a canonically defined equivalence
\begin{equation} \label{e:GLC}
\Shv_\Nilp(\Bun_G) \overset{\BL_G^{\on{restr}}}\simeq \IndCoh_\Nilp({}'\!\LS_\cG^{\on{restr}}),
\end{equation} 
where $'\!\LS_\cG^{\on{restr}}$ is the union of \emph{some of} the connected components of $\LS_\cG^{\on{restr}}$.

\begin{rem}

Recall also that according to Theorem 1.3.9(c), the inclusion 
$$'\!\LS_\cG^{\on{restr}}\subseteq \LS_\cG^{\on{restr}}$$
is known to be an equality if $G=GL_n$, and is conjectured to be such for any $G$.

\end{rem}

\sssec{} \label{sss:char L}

A defining feature of the equivalence $\BL_G^{\on{restr}}$ is that the functor
$$\CHom_{\Shv(\Bun_G)}(\on{Poinc}^{\on{Vac}}_!,-):\Shv_\Nilp(\Bun_G)\to \Vect_{\ol\BQ_\ell}$$
corresponds to the functor
$$\Gamma^{\IndCoh}({}'\!\LS^{\on{restr}}_\cG,-):\IndCoh_\Nilp({}'\!\LS_\cG^{\on{restr}}) \to \Vect_{\ol\BQ_\ell}.$$

Note, however, that it \emph{does not} make sense to ask about the image of $\on{Poinc}^{\on{Vac}}_!$ under
$\BL_G^{\on{restr}}$, because $\on{Poinc}^{\on{Vac}}_!$ \emph{does not} belong to $\Shv_\Nilp(\Bun_G)$.

\sssec{}

The equivalence $\BL_G^{\on{restr}}$ is compatible with the Frobenius endofunctor on both sides, where:

\begin{itemize}

\item On the geometric side, this is the *-direct image with respect to the geometric Frobenius of $\Bun_G$;

\medskip

\item On the spectral side, this is the !-pullback with respect to the automorphism of $\LS^{\on{restr}}_\cG$,
given by pullback along the geometric Frobenius on $X$.

\end{itemize}

\sssec{}

Denote
$$\LS_\cG^{\on{arithm}}:=(\LS_\cG^{\on{restr}})^{\Frob}.$$
According to \cite[Theorem 24.1.4]{AGKRRV1}, $\LS_\cG^{\on{arithm}}$ is a quasi-compact algebraic stack locally 
almost of finite type. Similarly, denote
$${}'\!\LS_\cG^{\on{arithm}}:=({}'\!\LS_\cG^{\on{restr}})^{\Frob}.$$

\medskip

As is explained in \cite[Sect. 1.5]{GR}, by taking the categorical trace of Frobenius on either side of $\BL_G^{\on{restr}}$, 
the equivalence \eqref{e:GLC} induces an isomorphism of objects of $\Vect_{\ol\BQ_\ell}$:

\begin{equation} \label{e:classical L}
\sFunct_c(\Bun_G(\BF_q),\ol\BQ_\ell) \overset{\sL_G}\simeq \Gamma^{\IndCoh}({}'\!\LS_\cG^{\on{arithm}},\omega_{'\!\LS_\cG^{\on{arithm}}}).
\end{equation}

\sssec{}

We claim now that there is a canonically defined map
\begin{equation} \label{e:O to omega}
\Gamma(\LS_\cG^{\on{arithm}},\CO_{\LS_\cG^{\on{arithm}}})\to
\Gamma^{\IndCoh}(\LS_\cG^{\on{arithm}},\omega_{\LS_\cG^{\on{arithm}}})
\end{equation} 
and similarly for $\LS_\cG^{\on{arithm}}$ replaced by $'\!\LS_\cG^{\on{arithm}}$. 

\medskip

In fact, such a map is defined for an algebraic stack of the form $\CY^\phi$, where $\CY$ is a quasi-smooth
formal algebraic stack, and $\phi$ is its endomorphism. 

\medskip

Consider the functor
$$\Xi_\CY: \QCoh(\CY)\to \IndCoh(\CY)$$
equal to $(\CL_\CY\otimes -)\circ \Upsilon_\CY$, where:

\begin{itemize}

\item $\Upsilon_\CY$ is the canonical natural transformation $\QCoh(-)\to \IndCoh(\CY)$, given by tensoring with the dualizing sheaf;

\medskip

\item $\CL_\CY$ is the inverse of the determinant line bundle on $\CY$, i.e., $\det(T^*(\CY))^{\otimes -1}[-\dim(\CY)]$. 

\end{itemize}

Since $\CY$ is quasi-smooth, this functor preserves compactness. Note that we have a diagram 
\begin{equation} \label{e:Xi diag}
\CD
\QCoh(\CY) @>{\Xi_\CY}>> \IndCoh(\CY) \\
@A{\phi^*}AA @AA{\phi^!}A \\
\QCoh(\CY) @>>{\Xi_\CY}> \IndCoh(\CY)
\endCD
\end{equation} 
that commutes up to a natural transformation, given by tensoring with $\CL_Y\otimes \phi^*(\CL^{\otimes -1}_Y)$. 

\medskip

Hence (see \cite[Sect. 3.4]{GKRV}), the diagram \eqref{e:Xi diag} induces a map
\begin{equation} \label{e:Xi Tr}
\Tr(\phi^*,\QCoh(\CY))\to \Tr(\phi^!,\IndCoh(\CY)).
\end{equation} 

We have:
$$\Tr(\phi^*,\QCoh(\CY))\simeq \Gamma(\CY^\phi,\CO_{\CY^\phi})$$
and
$$\Tr(\phi^!,\IndCoh(\CY))\simeq \Gamma^{\IndCoh}(\CY^\phi,\omega_{\CY^\phi}).$$

Hence, \eqref{e:Xi Tr} gives the sought-for map
\begin{equation} \label{e:Xi Tr bis}
\Gamma(\CY^\phi,\CO_{\CY^\phi}) \to \Gamma^{\IndCoh}(\CY^\phi,\omega_{\CY^\phi}).
\end{equation}

\begin{rem}
In fact the map \eqref{e:Xi Tr bis} can be obtained by applying $\Gamma(\CY^\phi,-)$ to a map 
\begin{equation} \label{e:Xi Tr bis refined}
\CO_{\CY^\phi} \to \Psi(\omega_{\CY^\phi}),
\end{equation}
where $\Psi_{(-)}$ is the natural transformation $\IndCoh(-)\to \QCoh(-)$ (defined for algebraic stacks or formal algebraic 
stacks).

\medskip

Indeed, the map \eqref{e:Xi Tr bis refined} is obtained from \eqref{e:Xi diag} by considering the enhanced trace 
$\Tr^{\on{enh}}_{\QCoh(\CY)}(-,-)$ relative to the $\QCoh(\CY)$-action on both sides (see \cite[Sect. 3.8.2]{GKRV}),
instead of the absolute categorical trace.

\end{rem}

\begin{rem}

The map \eqref{e:Xi Tr bis refined} can also be interpreted as follows. Note that $\CY^\phi$ is
\emph{quasi-quasi-smooth}. 

\medskip

We expect that for any quasi-quasi-smooth $\CZ$, there is a canonical map
\begin{equation} \label{e:Xi Tr bis refined qq}
\det(T^*(\CZ))[-\dim(\CZ)]\to \Psi_\CZ(\omega_\CZ).
\end{equation}
(It is easy to construct this map for $\CZ$ exhibited as a fiber product $\CY_1\underset{\CY_0}\times \CY_2$,
where $\CY_i$ are quasi-smooth; the problem of constructing \eqref{e:Xi Tr bis refined qq} consists of gluing.)

\medskip

Now, for $\CZ$ of the form $\CY^\phi$, we have $\dim(\CZ)=0$ and $\det(T^*(\CZ))\simeq \CO_\CZ$. 
Hence, in this case, \eqref{e:Xi Tr bis refined qq} is a map
\begin{equation} \label{e:Xi Tr bis refined another}
\CO_{\CY^\phi} \to \Psi(\omega_{\CY^\phi}).
\end{equation}

We expect that the map \eqref{e:Xi Tr bis refined another} equals the map \eqref{e:Xi Tr bis refined}.

\end{rem} 

\sssec{}

We are now ready to state the main result of this section (which is also the main technical result
of the paper):

\begin{thm} \label{t:main}
Under the isomorphism \eqref{e:classical L}, the element
$$\poinc^{\on{Vac}}_!\in \sFunct_c(\Bun_G(\BF_q),\ol\BQ_\ell)$$
corresponds to the image of
$$1\in \Gamma({}'\!\LS_\cG^{\on{arithm}},\CO_{'\!\LS_\cG^{\on{arithm}}})$$
under the map \eqref{e:O to omega}.
\end{thm} 

\sssec{} \label{sss:describe image}

The rest of the section will be devoted to the proof of this theorem. In the course of the proof, we will answer the following general question:

\medskip

Given $\CF\in \Shv(\Bun_G)^c$, equipped with a weak Weil structure, how to describe the element 
$\Gamma^{\IndCoh}({}'\!\LS_\cG^{\on{arithm}},\omega_{'\!\LS_\cG^{\on{arithm}}})$
that corresponds to $\funct(\CF)$ under $\sL_G$? 

\medskip

The above element will \emph{almost} be described as the \emph{class} (see \cite[Sect. 3.4.3]{GKRV} for the notion of class)
$$\on{cl}(\CM,\alpha)\in \Tr(\Frob,\IndCoh_\Nilp({}'\!\LS_\cG^{\on{restr}}))\simeq \Gamma^{\IndCoh}({}'\!\LS_\cG^{\on{arithm}},\omega_{'\!\LS_\cG^{\on{arithm}}})$$
for some $\CM\in \IndCoh_\Nilp(\LS_\cG^{\on{restr}})$ and $\alpha:\CM\to \Frob(\CM)$. However, the pair $(\CM,\alpha)$, rather than being unique, will
depend on a certain set of choices. 

\begin{rem}

Assume that $\CF$ belongs to $\Shv_\Nilp(\Bun_G)^c$. Then we can take $\CM=\BL_G^{\on{restr}}(\CF)$, with $\alpha$ 
corresponding to the given weak Weil structure on $\CF$. 

\medskip

Note, however, the fact that
$$\on{cl}(\CF,\alpha)\in \Tr(\Frob,\Shv_\Nilp(\Bun_G))\simeq \sFunct_c(\Bun_G(\BF_q),\ol\BQ_\ell)$$
equals $\funct(\CF)$ is a \emph{valid but a non-trivial} assertion. In fact, this is \cite[Theorem 5.2.3]{AGKRRV3}
(it relies on the refinement of the Grothedieck-Lefschetz formula developed in \cite{GV}). 

\end{rem} 

\ssec{Spectral action: recollections}

\sssec{}

The starting point of the proof is the fact that the category $\Shv_\Nilp(\Bun_G)$ carries a canonical action 
of $\QCoh(\LS^{\on{restr}}_\cG)$. 

\sssec{}

Recall that $\LS^{\on{restr}}_\cG$ is a \emph{formal} algebraic stack. By \cite[Sect. 7.9.8]{AGKRRV1}, we can write $\LS^{\on{restr}}_\cG$ as
$$\underset{a\in A}{\on{colim}}\, \CZ_a,$$
where:

\begin{itemize} 

\item The category $A$ of indices is a filtered set;

\item The terms $\CZ_a$ are quasi-smooth and quasi-compact algebraic stacks;

\item The maps $i_a:\CZ_a\to \LS^{\on{restr}}_\cG$ are \emph{regular}\footnote{For a morphism between prestacks locally almost of finite type, 
\emph{regular} means that the (derived) normal is an (unshifted) vector bundle.}  
schematic closed embeddings.

\end{itemize} 

\medskip

Moreover, we claim:

\begin{thm} \label{t:present} 
One can choose a presentation as above in a way compatible with the action of the Frobenius, i.e., 
$$a\mapsto \CZ_a$$
is a functor with values in the category of prestacks equipped with an endomorphism. Moreover,
one ensure that: 

\medskip

\noindent{\em(a)} Each $\CZ_a$ is invariant under the Chevalley involution;

\smallskip

\noindent{\em(b)} For every index $a$, the map 
$$(\CZ_a)^{\Frob}\to (\LS^{\on{restr}}_\cG)^{\Frob}=\LS^{\on{arithm}}_\cG$$
is an isomorphism.

\end{thm}

This theorem is not explicitly stated in \cite{AGKRRV1}. We supply a proof in \secref{s:proof of present}. 

\begin{rem}

Note that point (b) of \propref{t:present} is a strengthening of \cite[Theorem 24.1.4]{AGKRRV1}. The proof
that we will give is simpler than the one in {\it loc. cit.}

\end{rem} 

\sssec{} \label{sss:det funct prel}

Choose a presentation as in \thmref{t:present}, and let $a\in A$ be an index.

\medskip

Let $i_a$ denote the (schematic) closed embedding $\CZ_a\to \LS^{\on{restr}}_\cG$. The functor 
$$(i_a)_*:\QCoh(\CZ_a)\to \QCoh(\LS^{\on{restr}}_\cG)$$
admits a left adjoint, $(i_a)^*$, and also a right adjoint $(i_a)^!$ (thanks to the regularity assumption on the embeddings 
$i_a$). 

\medskip

Note that the composition
$$(i_a)^*\circ (i_a)_*:\QCoh(\CZ_a)\to \QCoh(\CZ_a)$$
admits a filtration with associated graded given by tensoring with $\Sym(N^*_a[1])$, 
where $N_a$ is the normal to $\CZ_a$ inside $\LS^{\on{restr}}_\cG$. 

\sssec{}

Denote 
$$\CE_a:=(i_a)^*\circ (i_a)_*(\CO_{\CZ_a}),$$
so that
\begin{equation} \label{e:gr Ea}
\on{gr}(\CE_a)\simeq \Sym(N^*_a[1]).
\end{equation} 

\sssec{} \label{sss:det funct}

Note that $\CE_a$ is naturally equipped with a structure of equivariance with respect to $\Frob$; denote it by $\alpha_a$. 
Consider the corresponding element
$$\sfe_a:=\on{cl}(\CE_a,\alpha_a)\in \Tr(\Frob,\QCoh(\CZ_a))\simeq \Gamma((\CZ_a)^{\Frob},\CO_{(\CZ_a)^{\Frob}}).$$

According to \cite[Proposition 2.2.3]{KP} (see also \cite[Proposition 3.5.7]{GKRV}), the element $\sfe_a$ is explicitly given by pointwise traces of $\alpha_a$
acting on the fibers of $\CE_a$ at points of $(\CZ_a)^{\Frob}$.

\sssec{}

Note that by \thmref{t:present}(b), the restriction map
$$\Gamma(\LS^{\on{arithm}}_\cG,\CO_{\LS^{\on{arithm}}_\cG})\to \Gamma((\CZ_a)^{\Frob},\CO_{(\CZ_a)^{\Frob}})$$
is an isomorphism, so we can view $\sfe_a$ as an element of $\Gamma(\LS^{\on{arithm}}_\cG,\CO_{\LS^{\on{arithm}}_\cG})$.

\medskip

We claim:

\begin{lem} \label{l:invertible}
The function 
$$\sfe_a\in \Gamma(\LS^{\on{arithm}}_\cG,\CO_{\LS^{\on{arithm}}_\cG})$$
is invertible.
\end{lem} 

\begin{proof}

By \eqref{e:gr Ea} and \secref{sss:det funct}, the value of $\sfe_a$ at a closed point 
$$\sigma\in (\CZ_a)^{\Frob}=\LS^{\on{arithm}}_\cG$$
is
$$\det(\on{id}-\alpha_a,(N^*_a)_\sigma).$$

However, the fact that $\CZ_a \to \LS_\cG^{\on{restr}}$ induces an isomorphism 
of $\Frob$-fixed points implies that $\alpha_a$ does not have eigenvalue $1$ when acting on 
$(N^*_a)_\sigma$: indeed, since
$$(T_\sigma(\CZ_a))^{\Frob}\to (T_\sigma(\LS_\cG^{\on{restr}}))^{\Frob}$$
is an isomorphism, we have
$$((N_a)_\sigma)^{\Frob}=0.$$

\end{proof} 

\sssec{}

We now come back to the action of $\QCoh(\LS^{\on{restr}}_\cG)$ on $\Shv_\Nilp(\Bun_G)$. 

\medskip

For an index $a$, consider the category 
$$\QCoh(\CZ_a)\underset{\QCoh(\LS^{\on{restr}}_\cG)}\otimes \Shv_\Nilp(\Bun_G).$$

Denote by $\iota_a$ the functor
$$\QCoh(\CZ_a)\underset{\QCoh(\LS^{\on{restr}}_\cG)}\otimes \Shv_\Nilp(\Bun_G)
\overset{((i_a)_*\otimes \on{Id})}\longrightarrow \Shv_\Nilp(\Bun_G)\overset{\iota}\hookrightarrow \Shv(\Bun_G).$$

According to \cite[Corollary 15.5.5]{AGKRRV1}, the functor $\iota_a$ admits a \emph{left} adjoint, given by
the Beilinson projector $\sP_a$. 

\begin{prop} \label{p:iota comp}
The functor $\iota_a$ preserves compactness. 
\end{prop} 

\begin{proof}

It suffices to show that both functors $((i_a)_*\otimes \on{Id})$ and $\iota$ preserve compactness. The functor $((i_a)_*\otimes \on{Id})$
admits a continuous right adjoint given by $((i_a)^!\otimes \on{Id})$, where
$$(i_a)^!:\QCoh(\LS^{\on{restr}}_\cG)\to \QCoh(\CZ_a)$$
is defined since $i_a$ is a regular embedding. 

\medskip

The fact that $\iota$ preserves compactness is \cite[Theorem 1.1.7]{GR}.

\end{proof} 

\ssec{Mechanism of the proof}

\sssec{}

We now state a general assertion that stands behind \thmref{t:main}: 

\begin{thm} \label{t:Z a}
For $\CF\in \Shv(\Bun_G)^c$ equipped with a weak Weil structure,
$$\funct(\CF)=\sfe_a^{-1}\cdot \funct(\iota_a\circ \sP_a(\CF)),$$
where:

\begin{itemize}

\item $\iota_a\circ \sP_a(\CF)$ is compact thanks to \propref{p:iota comp};

\item it is equipped with the weak Weil structure induced by that on $\CF$ 
(using the fact that the functors $\iota_a$ and $\sP_a$ commute with $\Frob$);

\item $\sfe_a$ is viewed as an endomorphism of $\Tr(\Frob,\Shv_\Nilp)\simeq \sFunct_c(\Bun_G(\BF_q),\ol\BQ_\ell)$ 
via the action of $\Tr(\Frob,\QCoh(\LS^{\on{restr}}_\cG))\simeq \Gamma(\LS^{\on{arithm}}_\cG,\CO_{\LS^{\on{arithm}}_\cG})$. 

\end{itemize}

\end{thm} 

\begin{rem} 

Note that \thmref{t:Z a} provides an answer to the question in \secref{sss:describe image}, \emph{up to} the multiplication by
$\sfe_a$. 

\medskip

Indeed, we can take $\CM:=\BL_G^{\on{restr}}(((i_a)_*\otimes \on{Id})\circ \sP_a(\CF))$ with the induced 
structure of weak $\Frob$-equivariance.

\end{rem}

\begin{rem}

Our proof of \thmref{t:main} uses the substacks $\CZ_a$ (and the corresponding projectors $\sP_a$) instead of the entire 
$\LS^{\on{restr}}_\cG$ and the universal Beilinson's projector $\sP$ (see below) because this is how we produce ensure
compactness: 

\medskip

The functors $\sP_a$ preserve compactness, whereas $\sP$ does not.

\end{rem} 

\sssec{}

In the rest of this subsection we will show how \thmref{t:Z a} implies \thmref{t:main}. This will be more or less tautological.

\medskip 

Let
$$\sP:\Shv(\Bun_G)\to \Shv_\Nilp(\Bun_G)$$ denote
the universal Beilinson's projector, which, thanks to the combination of \cite[Proposition 17.2.3]{AGKRRV1} and 
\cite[Theorem 1.1.7]{GR} identifies with the right adjoint of the functor 
$$\iota:  \Shv_\Nilp(\Bun_G)\hookrightarrow \Shv(\Bun_G).$$

\medskip

Denote
$$\on{Poinc}^{\on{Vac}}_{!,\Nilp}:=\sP(\on{Poinc}^{\on{Vac}}_!).$$

Recall that the equivalence $\BL_G^{\on{restr}}$ satisfies
\begin{equation} \label{e:L of Poinc}
\BL_G^{\on{restr}}(\on{Poinc}^{\on{Vac}}_{!,\Nilp})\simeq \Xi_{'\!\LS^{\on{restr}}_\cG}(\CO_{'\!\LS^{\on{restr}}_\cG}),
\end{equation} 
see \cite[Sect. 1.1.4]{GR}. Moreover, this identification is compatible with the structure of Frobenius-equivariance on the two sides. 

\begin{rem}
In fact, according to \cite[Proposition 1.2.5]{GR}, the identification \eqref{e:L of Poinc} is equivalent to
the characterization of $\BL_G^{\on{restr}}$ given in \secref{sss:char L}.
\end{rem} 

\sssec{}

We claim:

\begin{lem} \label{l:L Poinc a}
For any index $a$,
$$\BL_G^{\on{restr}}(((i_a)_*\otimes \on{Id})\circ \sP_a(\on{Poinc}^{\on{Vac}}_!))\simeq (i_a)_*(\CO_{\CZ_a})\otimes 
\Xi_{'\!\LS^{\on{restr}}_\cG}(\CO_{'\!\LS^{\on{restr}}_\cG}).$$
Moreover, this isomorphism is compatible with the structure of Frobenius-equivariance on the two sides. 
\end{lem}

\begin{proof}

First, by \cite[Corollary 15.5.3]{AGKRRV1}, we have
$$((i_a)_*\otimes \on{Id})\circ \sP_a(-) \simeq ((i_a)_*\otimes \on{Id})(\CO_{\CZ_a})\otimes \sP(-)$$
as functors
$$\Shv(\Bun_G)\to \Shv_\Nilp(\Bun_G).$$

Moreover, this isomorphism, being natural, is compatible with the Frobenius actions on the two sides. 

\medskip

Evaluating on $\on{Poinc}^{\on{Vac}}_!$, we obtain
$$((i_a)_*\otimes \on{Id})\circ \sP_a(\on{Poinc}^{\on{Vac}}_!) \simeq ((i_a)_*\otimes \on{Id})(\CO_{\CZ_a})\otimes \sP(\on{Poinc}^{\on{Vac}}_!).$$

Hence,
\begin{multline*} 
\BL_G^{\on{restr}}(((i_a)_*\otimes \on{Id}))\circ \sP_a(\on{Poinc}^{\on{Vac}}_!))\simeq
\BL_G^{\on{restr}}(((i_a)_*\otimes \on{Id})(\CO_{\CZ_a})\otimes \sP(\on{Poinc}^{\on{Vac}}_!))\simeq \\
\simeq 
(i_a)_*(\CO_{\CZ_a})\otimes \BL_G^{\on{restr}}(\sP(\on{Poinc}^{\on{Vac}}_!))\overset{\text{\eqref{e:L of Poinc}}}\simeq 
(i_a)_*(\CO_{\CZ_a})\otimes 
\Xi_{'\!\LS^{\on{restr}}_\cG}(\CO_{'\!\LS^{\on{restr}}_\cG}),
\end{multline*} 
where the second isomorphism is due to the fact that the functor $\BL_G^{\on{restr}}$ is
$\QCoh(\LS^{\on{restr}}_\cG)$-linear.

\end{proof} 

\sssec{}

Taking the classes of the two sides in \lemref{l:L Poinc a} with respect to their natural Frobenius-equivariance 
structures, we obtain that the isomorphism $\sL_G$ sends the class of $((i_a)_*\otimes \on{Id})\circ \sP_a(\on{Poinc}^{\on{Vac}}_!)$ 
in 
$$\Tr(\Frob,\Shv_\Nilp(\Bun_G))\simeq \sFunct_c(\Bun_G(\BF_q),\ol\BQ_\ell)$$
to the class of $(i_a)_*(\CO_{\CZ_a})\otimes \Xi_{'\!\LS^{\on{restr}}_\cG}(\CO_{'\!\LS^{\on{restr}}_\cG})$
in 
$$\Tr(\Frob, \IndCoh_\Nilp({}'\!\LS_\cG^{\on{restr}}))\simeq \Gamma^{\IndCoh}({}'\!\LS_\cG^{\on{restr}},\omega_{'\!\LS_\cG^{\on{restr}}}).$$

\sssec{}

By \cite[Theorem 5.2.3]{AGKRRV3}, the former class equals 
\begin{equation} \label{e:funct a}
\funct(\iota\circ ((i_a)_*\otimes \on{Id})\circ \sP_a(\on{Poinc}^{\on{Vac}}_!))
=\funct(\iota_a\circ \sP_a(\on{Poinc}^{\on{Vac}}_!)).
\end{equation}

Applying \thmref{t:Z a}, we obtain that the expression in \eqref{e:funct a} equals
$$\sfe_a\cdot \funct(\on{Poinc}^{\on{Vac}}_!)=\sfe_a\cdot \poinc^{\on{Vac}}_!.$$

\sssec{} \label{sss:e a}

Since the isomorphism $\sL_G$ is linear with respect to $\Gamma(\LS^{\on{arithm}}_\cG,\CO_{\LS^{\on{arithm}}_\cG})$,
it remains to show that for $a\geq a_0$, the class of $(i_a)_*(\CO_{\CZ_a})\otimes \Xi_{'\!\LS^{\on{restr}}_\cG}(\CO_{'\!\LS^{\on{restr}}_\cG})$ in
$\Gamma^{\IndCoh}({}'\!\LS_\cG^{\on{restr}},\omega_{'\!\LS_\cG^{\on{restr}}})$ (with respect to its natural Frobenius-equivariance structure) 
equals the image of 
$$\sfe_a|_{'\!\LS_\cG^{\on{arithm}}}\in \Gamma({}'\!\LS_\cG^{\on{arithm}},\CO_{'\!\LS_\cG^{\on{arithm}}})$$
under the map \eqref{e:O to omega}. 

\sssec{}

By the definition of the map \eqref{e:O to omega}, it suffices to show that the class of the object 
$(i_a)_*(\CO_{\CZ_a})$ in
$$\Tr(\Frob, \QCoh(\LS_\cG^{\on{restr}}))\simeq \Gamma(\LS_\cG^{\on{restr}},\CO_{\LS_\cG^{\on{restr}}})$$
(with respect to its natural Frobenius-equivariance structure) equals to $\sfe_a$.  
 
\medskip

Since the functor 
$$(i_a)^*:\QCoh(\LS_\cG^{\on{restr}})\to \QCoh(\CZ_a)$$
induces an isomorphism on $\Tr(\Frob,-)$ (by the assumption on $a$), it suffices to show that the class of 
$(i_a)^*\circ (i_a)_*(\CO_{\CZ_a})$
in
$$\Tr(\Frob, \QCoh(\CZ_a))\simeq \Gamma(\CZ_a,\CO_{\CZ_a})$$
equals $\sfe_a$.

\medskip

However, the latter is the definition of $\sfe_a$. 

\qed[\thmref{t:main}]

\ssec{Proof of \thmref{t:Z a}}

\sssec{}

Let $\CY$ be an algebraic stack. For $\CF\in \Shv(\CY)^c$, let $\Phi_\CF$ denote the functor 
$$\Shv(\CY)\to \Vect_{\ol\BQ_\ell}, \quad \CF'\mapsto \on{C}^\cdot_c(\CY,\CF\overset{*}\otimes \CF').$$

Note that this functor preserves compactness.

\sssec{}

Suppose that $\CY$ is defined over $\BF_q$, so that we have a well-defined geometric Frobenius endomorphism
$\Frob_\CY$. Recall (see \cite[Sect. 22.2]{AGKRRV1}) that we have a canonically defined map
$$\Tr((\Frob_\CY)_*,\Shv(\CY)) \overset{\on{LT}}\to \sFunct_c(\CY(\BF_q),\ol\BQ_\ell).$$

\medskip

Let $\CF$ as above be equipped with a weak Weil structure. This structure makes the diagram
\begin{equation} \label{e:Frob funct}
\CD
\Shv(\CY) @>{\Phi_\CF}>> \Vect_{\ol\BQ_\ell} \\
@V{(\Frob_\CY)_*}VV @VV{\on{Id}}V \\
\Shv(\CY)@>{\Phi_\CF}>>   \Vect_{\ol\BQ_\ell}
\endCD
\end{equation}
commute \emph{up to a natural transformation}. Hence, \eqref{e:Frob funct} gives rise to a map
$$\Tr((\Frob_\CY)_*,\Shv(\CY)) \overset{\Tr(\Frob,\Phi_\CF)}\longrightarrow \Tr(\on{Id}, \Vect_{\ol\BQ_\ell})=\ol\BQ_\ell.$$

\medskip

The following results from \cite[Theorem 0.4]{GV}:

\begin{thm} \label{t:Frob}
The diagram
\begin{equation} \label{e:Frob}
\CD
\Tr((\Frob_\CY)_*,\Shv(\CY))  @>{\Tr(\Frob,\Phi_\CF)}>> \ol\BQ_\ell \\
@V{\on{LT}}VV @VV{\on{id}}V \\
\sFunct_c(\CY(\BF_q),\ol\BQ_\ell) @>{\langle -,\funct(\CF)\rangle_{\CY(\BF_q)}}>> \ol\BQ_\ell 
\endCD
\end{equation}
commutes.
\end{thm} 

\sssec{}

Recall by \cite[Theorem 5.2.3]{AGKRRV3}, the isomorphism
$$\Tr(\Frob,\Shv_\Nilp(\Bun_G))\simeq \sFunct_c(\Bun_G(\BF_q),\ol\BQ_\ell)$$
equals the composition
$$\Tr(\Frob,\Shv_\Nilp(\Bun_G)) \overset{\Tr(\Frob,\iota)}\longrightarrow 
\Tr(\Frob,\Shv(\Bun_G)) \overset{\on{LT}}\longrightarrow \sFunct_c(\Bun_G(\BF_q),\ol\BQ_\ell).$$

From \thmref{t:Frob} we obtain:

\begin{cor} \label{c:Frob Nilp}
Let $\CF$ be an object of $\Shv(\Bun_G)^c$ equipped with a weak Weil structure. Then the following digram commutes
\begin{equation} \label{e:Frob Nilp}
\CD
\Tr(\Frob,\Shv_\Nilp(\Bun_G))  @>{\Tr(\Frob,\Phi_\CF\circ \iota)}>> \ol\BQ_\ell \\
@V{\simeq}VV @VV{\on{id}}V \\
\sFunct_c(\Bun_G(\BF_q),\ol\BQ_\ell) @>{\langle -,\funct(\CF)\rangle_{\Bun_G(\BF_q)}}>> \ol\BQ_\ell 
\endCD
\end{equation}
commutes.
\end{cor} 

\sssec{}

Applying \corref{c:Frob Nilp}, we obtain that in order to prove \thmref{t:Z a}, we have to show that the map
\begin{equation} \label{e:LHS}
\Tr(\Frob,\Shv_\Nilp(\Bun_G)) \overset{\Tr(\Frob,\Phi_{\iota_a\circ \sP_a(\CF)}\circ \iota)}\longrightarrow \ol\BQ_\ell
\end{equation} 
equals the map
\begin{equation} \label{e:RHS}
\Tr(\Frob,\Shv_\Nilp(\Bun_G)) \overset{\sfe_a\cdot (-)}\longrightarrow \Tr(\Frob,\Shv_\Nilp(\Bun_G)) 
\overset{\Tr(\Frob,\Phi_\CF\circ \iota)}\longrightarrow \ol\BQ_\ell.
\end{equation} 

\sssec{}

Note that the explicit description of $\iota_a\circ \sP_a$ as an integral Hecke functor (see \cite[Sect. 15.3.3]{AGKRRV1})
implies that 
for any $\CF_1,\CF_2\in \Shv(\Bun_G)$, we have a canonical isomorphism
$$\on{C}^\cdot_c(\Bun_G,\CF_1\overset{*}\otimes (\iota_a\circ \sP_a(\CF_2)))\simeq
\on{C}^\cdot_c(\Bun_G,(\iota_{a^\tau}\circ \sP_{a^\tau}(\CF_1))\overset{*}\otimes\CF_2),$$
Moreover, this is isomorphism is compatible with the Frobenius actions.

\medskip

In the above formula, $a^\tau$ is such that
$$\CZ_{a^\tau}=\tau(\CZ_a),$$
where $\tau$ is the Chevalley involution. Recall, however, that according to \thmref{t:present}, we can assume that $Z_a$ 
is $\tau$-invariant, so that $a^\tau=a$. 

\medskip

From here we obtain that there is a canonical isomorphism of functors
$$\Phi_{\iota_a\circ \sP_a(\CF)}\simeq \Phi_\CF \circ \iota_a\circ \sP_a,$$
compatible with the Frobenius actions.

\sssec{}

Note now there is also a canonical isomorphism
$$\iota_a\circ \sP_a\circ \iota \simeq \iota\circ ((i_a)_*(\CO_{\CZ_a})\otimes (-)),$$
compatible with the Frobenius actions.

\medskip

Combining, we obtain an isomorphism
$$\Phi_{\iota_a\circ \sP_a(\CF)}\circ \iota \simeq \Phi_\CF \circ \iota \circ ((i_a)_*(\CO_{\CZ_a})\otimes (-)),$$
compatible with the Frobenius actions.

\medskip

Thus, we can rewrite the map in \eqref{e:LHS} as
%\begin{equation} \label{e:LHS bis}
$$
\Tr(\Frob,\Shv_\Nilp(\Bun_G)) \overset{\Tr(\Frob,((i_a)_*(\CO_{\CZ_a})\otimes (-)))}\longrightarrow 
\Tr(\Frob,\Shv_\Nilp(\Bun_G)) 
\overset{\Tr(\Frob,\Phi_\CF\circ \iota)}\longrightarrow \ol\BQ_\ell.$$
%\end{equation} 

\sssec{}

Hence, in order to show that \eqref{e:LHS} equals \eqref{e:RHS}, it remains to show that the map
$$\Tr(\Frob,\Shv_\Nilp(\Bun_G)) \overset{\Tr(\Frob,((i_a)_*(\CO_{\CZ_a})\otimes (-)))}\longrightarrow 
\Tr(\Frob,\Shv_\Nilp(\Bun_G))$$
is given by the action of $\sfe_a$.

\medskip

This would follow from the assertion that the class of the object $(i_a)_*(\CO_{\CZ_a})$ in
$$\Tr(\Frob, \QCoh(\LS_\cG^{\on{restr}}))\simeq \Gamma(\LS_\cG^{\on{restr}},\CO_{\LS_\cG^{\on{restr}}})$$
(with respect to its natural Frobenius-equivariance structure) equals to $\sfe_a$.  However, this has been
already established in \secref{sss:e a}. 

\qed[\thmref{t:Z a}]

\section{Filtration by nilpotent orbits } \label{s:filtr}

In this section we introduce the filtration on the space of automorphic functions, indexed by
the poset of nilpotent orbits in $\cg$. 

\medskip

We define this filtration geometrically (by taking the trace
of Frobenius on an appropriate category) on the spectral side, and then transfer it to
the geometric side by the isomorphism $\sL_G$. 

\medskip

We prove that the $0$-th term of the filtration corresponds to the subspace of functions
generated by the basic Whittaker (a.k.a. vacuum Poincar\'e) function. 

\medskip

We formulate a key conjecture that says that this filtration is defined over $\BQ$.

\ssec{The filtration on the spectral side} \label{ss:filtr spec}

\sssec{}

Recall (see \cite[Sect. 21.2.4]{AGKRRV1}) that points of $\Sing(\LS_\cG^{\on{restr}})$ are in bijection with pairs
$$(\sigma,A),\quad \sigma\in \LS_\cG^{\on{restr}},\,\, A\in H^0(X,\cg^*_\sigma).$$

Hence, to every closed conical $\on{Ad}$-invariant subset $Y\subset \cg^*$, we can associate a closed conical 
subset of $\Sing(\LS_\cG^{\on{restr}})$; by a slight abuse of notation, we will denote it by the same symbol $Y$. 

\medskip

In what follows, we will identify $\cg^*$ with $\cg$ using an $\on{Ad}$-invariant bilinear form. So, we can think
of $Y$ as above as a conical $\on{Ad}$-invariant subset $Y\subset \cg$. We will be interested in closed 
$\on{Ad}$-invariant subsets contained in the nilpotent cone of $\cg$; note that any such is automatically conical. 

\sssec{}

For $Y$ as above consider the corresponding full subcategory
\begin{equation} \label{e:embed Y}
\IndCoh_Y(\LS_\cG^{\on{restr}})\subset \IndCoh_\Nilp(\LS_\cG^{\on{restr}}).
\end{equation}

By \cite[Sect. 11.1.6]{AG1}, this subcategory is compactly generated, and the above embedding preserves
compact objects.

\medskip

The subcategory \eqref{e:embed Y} is stable under the action of the Frobenius endofunctor. 
Hence, we can consider
\begin{equation} \label{e:Y Tr}
\Tr(\Frob,\IndCoh_Y(\LS_\cG^{\on{restr}}))\in \Vect_{\ol\BQ_\ell}.
\end{equation} 

For every $Y_1\subset Y_2$ we have an embedding
$$\IndCoh_{Y_1}(\LS_\cG^{\on{restr}})\subset \IndCoh_{Y_2}(\LS_\cG^{\on{restr}})$$
that preserves compactness. Hence, it induces a map
\begin{equation} \label{e:Y12 Tr}
\Tr(\Frob,\IndCoh_{Y_1}(\LS_\cG^{\on{restr}}))\to \Tr(\Frob,\IndCoh_{Y_2}(\LS_\cG^{\on{restr}})).
\end{equation} 

In particular, the embedding \eqref{e:embed Y} induces a map
\begin{multline} \label{e:embed Y Tr}
\Tr(\Frob,\IndCoh_Y(\LS_\cG^{\on{restr}}))\to \Tr(\Frob,\IndCoh_{\Nilp}(\LS_\cG^{\on{restr}}))\simeq \\
\simeq \Gamma^{\IndCoh}(\LS^{\on{arithm}}_\cG,\omega_{\LS^{\on{arithm}}_\cG})\simeq 
\Gamma(\LS^{\on{arithm}}_\cG,\Psi_{\LS^{\on{arithm}}_\cG}(\omega_{\LS^{\on{arithm}}_\cG})).
\end{multline} 

\begin{rem}
If we consider the enhanced trace relative to the action of $\QCoh(\LS_\cG^{\on{restr}})$ instead of the absolute trace, to every $Y$
as above we can attach an object 
$$\Tr^{\on{enh}}_{\QCoh(\LS_\cG^{\on{restr}})}(\Frob,\IndCoh_Y(\LS_\cG^{\on{restr}}))\in 
\on{HH}(\Frob,\QCoh(\LS_\cG^{\on{restr}}))\simeq \QCoh(\LS_\cG^{\on{arithm}}),$$
so that
$$\Tr(\Frob,\IndCoh_Y(\LS_\cG^{\on{restr}})) \simeq \Gamma(\LS_\cG^{\on{arithm}},\Tr^{\on{enh}}_{\QCoh(\LS_\cG^{\on{restr}})}(\Frob,\IndCoh_Y(\LS_\cG^{\on{restr}}))).$$

For example, for $Y=\Nilp$, we have 
$$\Tr^{\on{enh}}_{\QCoh(\LS_\cG^{\on{restr}})}(\Frob,\IndCoh_Y(\LS_\cG^{\on{restr}}))=\Psi_{\LS_\cG^{\on{arithm}}}(\omega_{\LS_\cG^{\on{arithm}}})$$ 
and for $Y=\{0\}$, we have 
$$\Tr^{\on{enh}}_{\QCoh(\LS_\cG^{\on{restr}})}(\Frob,\IndCoh_Y(\LS_\cG^{\on{restr}}))=\CO_{\LS_\cG^{\on{arithm}}}.$$

\medskip

Furthermore, the maps \eqref{e:Y12 Tr} upgrade to maps
\begin{equation} \label{e:Y12 Tr upgrade}
\Tr^{\on{enh}}_{\QCoh(\LS_\cG^{\on{restr}})}(\Frob,\IndCoh_{Y_1}(\LS_\cG^{\on{restr}}))\to
\Tr^{\on{enh}}_{\QCoh(\LS_\cG^{\on{restr}})}(\Frob,\IndCoh_{Y_2}(\LS_\cG^{\on{restr}})).
\end{equation}

For $Y_1=\{0\}$ and $Y_2=\Nilp$, the map \eqref{e:Y12 Tr upgrade} is, by construction, 
the map \eqref{e:Xi Tr bis refined}. 

\end{rem} 

\sssec{}

Let $\bO\subset \CN$ be a nilpotent orbit, and let $\ol\bO$ be its closure. Consider the embedding
$$\IndCoh_{\ol\bO-\bO}(\LS_\cG^{\on{restr}})\subset \IndCoh_{\ol\bO}(\LS_\cG^{\on{restr}}),$$
and the quotient
$$\IndCoh_\bO(\LS_\cG^{\on{restr}}):=\IndCoh_{\ol\bO}(\LS_\cG^{\on{restr}})/\IndCoh_{\ol\bO-\bO}(\LS_\cG^{\on{restr}}).$$

The following conjecture was proposed jointly with D.~Kazhdan: 

\begin{conj} \label{c:gr} 
The object
$$\Tr(\Frob,\IndCoh_\bO(\LS_\cG^{\on{restr}}))\in \Vect_{\ol\BQ_\ell}$$
is concentrated in cohomological degree $0$.
\end{conj} 

In fact, this conjecture is a theorem-in-progress, and is the subject of the forthcoming work \cite{GaLaRa}.

\sssec{}

A particular case of \conjref{c:gr} for $Y=\{0\}$ is much simpler, and has been recently settled in \cite[Theorem 3.7.2]{GaLiRe}. I.e., we have:

\begin{thm} \label{t:exc} 
The object 
$$\Tr(\Frob,\IndCoh_{\{0\}}(\LS_\cG^{\on{restr}}))\simeq \Tr(\Frob,\QCoh(\LS_\cG^{\on{restr}}))\simeq \Gamma(\LS_\cG^{\on{arithm}},\CO_{\LS_\cG^{\on{arithm}}})$$
is concentrated in cohomological degree $0$.
\end{thm} 

\begin{rem} 
Recall that $\LS_\cG^{\on{arithm}}$ can be exhibited as the quotient of an affine (derived) scheme 
by a reductive group. Hence, $\Gamma(\LS_\cG^{\on{arithm}},\CO_{\LS_\cG^{\on{arithm}}})$ is 
a priori concentrated in non-positive cohomological degrees. Thus, the content of \thmref{t:exc}
is that this algebra does not have negative cohomology either. 

\medskip

Note, however, that (unless $\cG$ is a torus), the stack $\LS_\cG^{\on{arithm}}$ is not classical
(and \emph{not} eventually coconnective). So, the structure sheaf of $\LS_\cG^{\on{arithm}}$ does
have negative cohomology, but all of it vanishes after taking global sections.

\end{rem} 

\sssec{}

For the rest of this subsection we will assume \conjref{c:gr}. We claim:

\begin{cor} \label{c:filtr}  For every closed $\on{Ad}$-invariant subset $Y\subset \Nilp$, the object \eqref{e:Y Tr} is concentrated in cohomological degree $0$, and the map
\eqref{e:embed Y Tr} is injective. 
\end{cor}

\medskip

In other words, \corref{c:filtr} says that the assignment
$$Y\rightsquigarrow \Tr(\Frob,\IndCoh_Y(\LS_\cG^{\on{restr}}))$$
defines a filtration on (the classical vector space) $\Tr(\Frob,\IndCoh_\Nilp(\LS_\cG^{\on{restr}}))$,
indexed by the poset of $\on{Ad}$-invariant closed subsets of $\Nilp$. 

\medskip

As particular case of \corref{c:filtr}, we obtain:

\begin{cor} \label{c:filtr part} \hfill

\smallskip

\noindent{\em(a)} The object 
$$\Gamma(\LS^{\on{arithm}}_\cG,\Psi_{\LS^{\on{arithm}}_\cG}(\omega_{\LS^{\on{arithm}}_\cG}))\in \Vect_{\ol\BQ_\ell}$$
is concentrated in cohomological degree $0$;

\smallskip

\noindent{\em(b)} The map \eqref{e:O to omega} is injective.

\end{cor} 

\sssec{}

We will also prove:

\begin{cor} \label{c:filtr bis}
For a pair of closed subsets $Y_1$, $Y_2$, we have
$$\Tr(\Frob,\IndCoh_{Y_1}(\LS_\cG^{\on{restr}}))\cap \Tr(\Frob,\IndCoh_{Y_2}(\LS_\cG^{\on{restr}}))=
\Tr(\Frob,\IndCoh_{Y_1\cap Y_2}(\LS_\cG^{\on{restr}}))$$
and
$$\Tr(\Frob,\IndCoh_{Y_1}(\LS_\cG^{\on{restr}}))+ \Tr(\Frob,\IndCoh_{Y_2}(\LS_\cG^{\on{restr}}))=
\Tr(\Frob,\IndCoh_{Y_1\cup Y_2}(\LS_\cG^{\on{restr}}))$$
as subspaces of $\Tr(\Frob,\IndCoh_\Nilp(\LS_\cG^{\on{restr}}))$.
\end{cor} 

\ssec{Proof of Corollaries \ref{c:filtr} and \ref{c:filtr bis}}

\sssec{}

To prove \corref{c:filtr}, we will argue by induction. Assume that for all $Y$ with\footnote{Recall that nilpotent orbits have even dimension.}  
$\dim(Y)\leq 2n$ the following two statements hold:

\begin{enumerate}

\item $\Tr(\Frob,\IndCoh_Y(\LS_\cG^{\on{restr}}))$ is concentrated in cohomological degree $0$;

\item For all $Y'\subseteq Y$, the map 
$$\Tr(\Frob,\IndCoh_{Y'}(\LS_\cG^{\on{restr}}))\to \Tr(\Frob,\IndCoh_Y(\LS_\cG^{\on{restr}}))$$
is injective.

\end{enumerate} 

\sssec{}

The base of the induction is when $n=0$, in which case the assertion coincides with \thmref{t:exc}. Let us
perform the induction step. I.e., let $Y$ be of dimension $2(n+1)$. Let $Y_0\subset Y$ be the union of
orbits of dimensions $\leq 2n$. 

\medskip

We claim that 
$$\on{coFib}\left(\Tr(\Frob,\IndCoh_{Y_0}(\LS_\cG^{\on{restr}}))\to \Tr(\Frob,\IndCoh_Y(\LS_\cG^{\on{restr}}))\right)$$
is concentrated in cohomological degree $0$. 

\medskip

Indeed, let $\bC_0\hookrightarrow \bC$ be a fully faithful embedding of dualizable categories
that admits a continuous right adjoint. Let $F$ be an endofunctor of $\bC$ that preserves $\bC'$. 
In this case, we have a fiber sequence in $\Vect_{\ol\BQ_\ell}$:
\begin{equation} \label{e:fiber traces} 
\Tr(F|_{\bC_0},\bC')\to \Tr(F,\bC)\to \Tr(F|_{\bC/\bC_0},\bC/\bC_0).
\end{equation} 

\medskip

Hence, it is enough to show that
$$\Tr(\Frob,\IndCoh_Y(\LS_\cG^{\on{restr}})/\IndCoh_{Y_0}(\LS_\cG^{\on{restr}}))$$ is
concentrated in cohomological degree $0$. 

\sssec{}

Write 
$$\CY-\CY_0=\underset{i}\sqcup\, \bO_i,$$
where $\bO_i$ are nilpotent orbits (of dimension $2(n+1)$). We claim that
$$\IndCoh_Y(\LS_\cG^{\on{restr}})/\IndCoh_{Y'}(\LS_\cG^{\on{restr}})\simeq
\underset{i\in I}\Pi\, \IndCoh_{\bO_i}(\LS_\cG^{\on{restr}}).$$

Indeed, this follows from the localization of $\IndCoh(\CZ)/\QCoh(\CZ)$ onto $\BP(\Sing(\CZ))_\dr$, see \cite[Theorem 1.4.2]{AG2}. 

\medskip

Hence, the required assertion follows from \conjref{c:gr}.

\sssec{}

Since, by the inductive hypothesis, $\Tr(\Frob,\IndCoh_{Y_0}(\LS_\cG^{\on{restr}}))$ is concentrated in cohomological
degree $0$, we obtain that so is $\Tr(\Frob,\IndCoh_Y(\LS_\cG^{\on{restr}}))$; moreover, the map 
$$\Tr(\Frob,\IndCoh_{Y_0}(\LS_\cG^{\on{restr}}))\to \Tr(\Frob,\IndCoh_Y(\LS_\cG^{\on{restr}})$$
is injective.

\sssec{}

Let now $Y'\subset Y$ be an arbitrary closed $\on{Ad}$-invariant subset. We already know that 
$\Tr(\Frob,\IndCoh_{Y'}(\LS_\cG^{\on{restr}}))$ is concentrated in cohomological
degree $0$, and it remains to show that the map
$$\Tr(\Frob,\IndCoh_{Y'}(\LS_\cG^{\on{restr}}))\to \Tr(\Frob,\IndCoh_Y(\LS_\cG^{\on{restr}}))$$
is injective. 

\medskip

Set $Y'_0=Y'\cap Y_0$. By what we have proved above, the map
$$\Tr(\Frob,\IndCoh_{Y'_0}(\LS_\cG^{\on{restr}}))\to \Tr(\Frob,\IndCoh_{Y'}(\LS_\cG^{\on{restr}}))$$
is injective.

\medskip

By the inductive hypothesis, the map
$$\Tr(\Frob,\IndCoh_{Y'_0}(\LS_\cG^{\on{restr}}))\to \Tr(\Frob,\IndCoh_{Y_0}(\LS_\cG^{\on{restr}}))$$
is injective. Hence, the map
$$\Tr(\Frob,\IndCoh_{Y'_0}(\LS_\cG^{\on{restr}}))\to \Tr(\Frob,\IndCoh_Y(\LS_\cG^{\on{restr}}))$$
is injective. 

\medskip

Hence, it suffices to show that the map
\begin{multline*} 
\on{coFib}\left(\Tr(\Frob,\IndCoh_{Y'_0}(\LS_\cG^{\on{restr}}))\to \Tr(\Frob,\IndCoh_{Y'}(\LS_\cG^{\on{restr}}))\right)\to \\
\to \on{coFib}\left(\Tr(\Frob,\IndCoh_{Y'_0}(\LS_\cG^{\on{restr}}))\to \Tr(\Frob,\IndCoh_{Y}(\LS_\cG^{\on{restr}}))\right)
\end{multline*}
is injective. 

\medskip

For that, it suffices to show that the map
\begin{multline} \label{e:fewer} 
\on{coFib}\left(\Tr(\Frob,\IndCoh_{Y'_0}(\LS_\cG^{\on{restr}}))\to \Tr(\Frob,\IndCoh_{Y'}(\LS_\cG^{\on{restr}}))\right)\to \\
\to \on{coFib}\left(\Tr(\Frob,\IndCoh_{Y_0}(\LS_\cG^{\on{restr}}))\to \Tr(\Frob,\IndCoh_{Y}(\LS_\cG^{\on{restr}}))\right)
\end{multline}
is injective. 

\sssec{}

By the above, we identify
$$\on{coFib}\left(\Tr(\Frob,\IndCoh_{Y_0}(\LS_\cG^{\on{restr}}))\to \Tr(\Frob,\IndCoh_{Y}(\LS_\cG^{\on{restr}}))\right)
\simeq \underset{i\in I}\Pi\, \Tr(\Frob,\IndCoh_{\bO_i}(\LS_\cG^{\on{restr}})).$$

Similarly, 
\begin{multline*} \
\on{coFib}\left(\Tr(\Frob,\IndCoh_{Y'_0}(\LS_\cG^{\on{restr}}))\to \Tr(\Frob,\IndCoh_{Y'}(\LS_\cG^{\on{restr}}))\right)\simeq \\
\simeq \underset{i\in I'}\Pi\, \Tr(\Frob,\IndCoh_{\bO_i}(\LS_\cG^{\on{restr}})),
\end{multline*}
where $I'$ is a subset of $I$.

\medskip

Hence, \eqref{e:fewer} is an inclusion of a direct summand.

\qed[\corref{c:filtr}]

\sssec{}

We now prove \corref{c:filtr bis}. Consider the diagram of categories
$$
\CD
\IndCoh_{Y_1}(\LS_\cG^{\on{restr}}) @>>> \IndCoh_{Y_1\cup Y_2}(\LS_\cG^{\on{restr}}) @>>> \IndCoh_{Y_1\cup Y_2}(\LS_\cG^{\on{restr}})/\IndCoh_{Y_1}(\LS_\cG^{\on{restr}}) \\
@AAA @AAA @AAA \\
\IndCoh_{Y_1\cap Y_2}(\LS_\cG^{\on{restr}}) @>>> \IndCoh_{Y_2}(\LS_\cG^{\on{restr}}) @>>> \IndCoh_{Y_2}(\LS_\cG^{\on{restr}})/\IndCoh_{Y_1\cap Y_2}(\LS_\cG^{\on{restr}}),
\endCD
$$
where the left horizontal functors admit continuous right adjoints. 

\medskip

By \eqref{e:fiber traces}, it suffices to show that the right vertical arrow induces an isomorphism after taking $\Tr(\Frob,-)$. 
We claim, however, that the corresponding functor is an equivalence. Indeed, this follows from the localization of
$\IndCoh(-)$ on $\BP(\Sing(-))$ in \cite{AG2}.

\qed[\corref{c:filtr bis}]

\ssec{The filtration on the automorphic side}

\sssec{}

The discussion in \secref{ss:filtr spec} is equally applicable to $\LS^{\on{restr}}_\cG$ replaced by $'\!\LS^{\on{restr}}_\cG$.

\medskip

Let us temporarily assume \conjref{c:gr}. Then the filtration on 
$$\Tr(\Frob,\IndCoh_\Nilp({}'\!\LS_\cG^{\on{restr}}))\simeq \Gamma^{\IndCoh}({}'\!\LS^{\on{arithm}}_\cG,\omega_{'\!\LS^{\on{arithm}}_\cG})$$
given by \corref{c:filtr} induces a filtration on 
$$\sFunct_c(\Bun_G(\BF_q),\ol\BQ_\ell) \overset{\sL_G}\simeq \Gamma^{\IndCoh}({}'\!\LS^{\on{arithm}}_\cG,\omega_{'\!\LS^{\on{arithm}}_\cG}).$$

Denote this filtration by
\begin{equation} \label{e:autom filtr}
Y\rightsquigarrow \sFunct_c(\Bun_G(\BF_q),\ol\BQ_\ell)_Y.
\end{equation} 

Moreover, by \corref{c:filtr bis}, we have
$$\sFunct_c(\Bun_G(\BF_q),\ol\BQ_\ell)_{Y_1}\cap \sFunct_c(\Bun_G(\BF_q),\ol\BQ_\ell)_{Y_2}=
\sFunct_c(\Bun_G(\BF_q),\ol\BQ_\ell)_{Y_1\cap Y_2}$$
and
$$\sFunct_c(\Bun_G(\BF_q),\ol\BQ_\ell)_{Y_1}+ \sFunct_c(\Bun_G(\BF_q),\ol\BQ_\ell)_{Y_2}=
\sFunct_c(\Bun_G(\BF_q),\ol\BQ_\ell)_{Y_1\cup Y_2}.$$

\sssec{}

For a given nilpotent orbit $\bO$, denote
$$\sFunct_c(\Bun_G(\BF_q),\ol\BQ_\ell)_\bO:=\sFunct_c(\Bun_G(\BF_q),\ol\BQ_\ell)_{\ol\bO}/\sFunct_c(\Bun_G(\BF_q),\ol\BQ_\ell)_{\ol\bO-\bO}.$$

Note that by construction, the isomorphism $\sL_G$ identifies
$$\sFunct_c(\Bun_G(\BF_q),\ol\BQ_\ell)_\bO\simeq \Tr(\Frob,\IndCoh_\bO({}'\!\LS_\cG^{\on{restr}})).$$

\begin{rem} \label{r:fake}

We will refer to 
$$\on{gr}(\sFunct_c(\Bun_G(\BF_q),\ol\BQ_\ell)):=\underset{\bO}\oplus\, \sFunct_c(\Bun_G(\BF_q),\ol\BQ_\ell)_\bO$$
as the space of \emph{fake} automorphic functions. 

\end{rem} 

\sssec{}

In the special case of $Y=\{0\}$, denote the corresponding term of the filtration by
$$\sFunct_c(\Bun_G(\BF_q),\ol\BQ_\ell)_{\{0\}}=:\sFunct_c(\Bun_G(\BF_q),\ol\BQ_\ell)_{\on{temp}}.$$

Note that by definition, $\sFunct_c(\Bun_G(\BF_q),\ol\BQ_\ell)_{\on{temp}}$ is isomorphic under $\sL_G$ to the image of the map
\begin{equation} \label{e:temp to all}
\Gamma({}'\!\LS^{\on{arithm}}_\cG,\CO_{'\!\LS^{\on{arithm}}_\cG})\overset{\text{\eqref{e:O to omega}}}\longrightarrow
%\Gamma({}'\!\LS^{\on{arithm}}_\cG,\Psi_{'\!\LS^{\on{arithm}}_\cG}(\omega_{'\!\LS^{\on{arithm}}_\cG}))\simeq
\Gamma^{\IndCoh}({}'\!\LS^{\on{arithm}}_\cG,\omega_{'\!\LS^{\on{arithm}}_\cG}).
\end{equation}

\sssec{}

Note that the latter gives an unconditional definition of the subspace 
$$\sFunct_c(\Bun_G(\BF_q),\ol\BQ_\ell)_{\on{temp}} \subset \sFunct_c(\Bun_G(\BF_q),\ol\BQ_\ell).$$

Indeed, both sides in \eqref{e:temp to all} are concentrated in cohomological degree $0$: 

\begin{itemize}

\item For the right-hand side, this is true thanks to the existence of the isomorphism $\sL_G$;

\item For the left-hand side, this is true thanks to \thmref{t:exc}.

\end{itemize} 

\begin{rem}

What \conjref{c:gr} ``buys" us is the knowledge that the map \eqref{e:temp to all} is injective, see \corref{c:filtr part}(a).

\end{rem}

\ssec{The non-degenerate subspace}

\sssec{} \label{sss:act of exc}

Denote 
$$\on{Exc}(X,\cG):=\Gamma(\LS_\cG^{\on{arithm}},\CO_{\LS_\cG^{\on{arithm}}}).$$

We call it the \emph{excursion algebra}. Note that according to \thmref{t:exc}, it is \emph{classical}.

\medskip

The algebra $\Gamma(\LS_\cG^{\on{arithm}},\CO_{\LS_\cG^{\on{arithm}}})$ acts on
$$\Gamma^{\IndCoh}({}'\!\LS_\cG^{\on{arithm}},\omega_{'\!\LS_\cG^{\on{arithm}}}).$$

Transferring this action using the identification $\sL_G$, we obtain an action of $\on{Exc}(X,\cG)$ on the
space $\sFunct_c(\Bun_G(\BF_q),\ol\BQ_\ell)$ of automorphic functons. 

\begin{rem}
The above action is a version of the main construction in \cite[Sect. 11]{VLaf} (on cuspidal functions), 
extended in \cite{Xue} to all automorphic functions.
\end{rem} 

\sssec{}

Recall the element 
$$\poinc^{\on{Vac}}_!\in \sFunct_c(\Bun_G(\BF_q),\ol\BQ_\ell),$$
see \secref{sss:vac poinc funct}.

\medskip

Let
\begin{equation} \label{e:non-degen}
\sFunct_c(\Bun_G(\BF_q),\ol\BQ_\ell)_{\on{non-degen}}\subset \sFunct_c(\Bun_G(\BF_q),\ol\BQ_\ell)
\end{equation}
be the subspace generated by $\poinc^{\on{Vac}}_!$ under the action of $\on{Exc}(X,\cG)$. 

\sssec{}

We claim:

\begin{cor} \label{c:temp=nondeg}
The subspaces
$$\sFunct_c(\Bun_G(\BF_q),\ol\BQ_\ell)_{\on{non-degen}} \subset \sFunct_c(\Bun_G(\BF_q),\ol\BQ_\ell)
\supset \sFunct_c(\Bun_G(\BF_q),\ol\BQ_\ell)_{\on{temp}}$$
coincide.
\end{cor}

\begin{proof}

By \thmref{t:main}, the subspace \eqref{e:non-degen} corresponds under $\sL_G$ to the subspace of
$$\Gamma^{\IndCoh}({}'\!\LS^{\on{arithm}}_\cG,\omega_{'\!\LS^{\on{arithm}}_\cG})$$
generated under the action of the excursion algebra 
by the image of the element 
$$1\in \Gamma({}'\!\LS_\cG^{\on{arithm}},\CO_{'\!\LS_\cG^{\on{arithm}}})$$
under \eqref{e:O to omega}. 

\medskip

However, the latter space is the same as the image of the map \eqref{e:temp to all}. 

\end{proof} 

\ssec{The support of non-degenerate cuspidal functions}

In this subsection we assume that $G$ is semi-simple. 

\sssec{}

The following strengthening of \thmref{t:exc} has also been obtained in \cite[Corollary 3.7.3]{GaLiRe}: 

\begin{thm} \label{t:Exc red}
The excursion algebra 
$$\on{Exc}(X,\cG)=\Gamma(\LS_\cG^{\on{arithm}},\CO_{\LS_\cG^{\on{arithm}}})$$
is nilpotent-free.
\end{thm} 

In this subsection we will show how \thmref{t:Exc red} implies a particular case of the Ramanujan conjecture
(for function fields, everywhere unramified case). 

\sssec{}

Let 
$$\LS^{\on{arithm}}_{\cG,\on{irred}}\subset \LS^{\on{arithm}}_\cG$$
be the locus of \emph{irreducible} Weil $\cG$-local systems (i.e., those local systems that
do not admit reductions to proper parabolics). 

\medskip

Recall that according to \cite[Theorem 24.1.6]{AGKRRV1}, $\LS^{\on{arithm}}_{\cG,\on{irred}}$ is a union 
of connected components of $\LS^{\on{arithm}}_\cG$, each of the form
$$\on{pt}/\Gamma,$$
where $\Gamma$ is a finite group. 

\medskip

Thus, 
$$\omega_{\LS^{\on{arithm}}_\cG}|_{\LS^{\on{arithm}}_{\cG,\on{irred}}} \simeq
\CO_{\LS^{\on{arithm}}_{\cG,\on{irred}}}.$$

In particular, 
$$\Gamma(\LS^{\on{arithm}}_{\cG,\on{irred}},\CO_{\LS^{\on{arithm}}_{\cG,\on{irred}}})\simeq
\underset{\sigma\in \LS^{\on{arithm}}_{\cG,\on{irred}}/\sim}\oplus\, \ol\BQ_\ell$$
is a direct summand of 
$$\Gamma^\IndCoh(\LS^{\on{arithm}}_\cG,\omega_{\LS^{\on{arithm}}_\cG}).$$

\sssec{}

Denote
$$'\!\LS^{\on{arithm}}_{\cG,\on{irred}}:=\LS^{\on{arithm}}_{\cG,\on{irred}}\cap {}'\!\LS^{\on{arithm}}_\cG.$$

We obtain the corresponding direct summand 
\begin{equation} \label{e:irred autom}
\Gamma({}'\!\LS^{\on{arithm}}_{\cG,\on{irred}},\CO_{'\!\LS^{\on{arithm}}_{\cG,\on{irred}}})\simeq
\underset{\sigma\in {}'\!\LS^{\on{arithm}}_{\cG,\on{irred}}/\sim}\oplus\, \ol\BQ_\ell
\end{equation} 
of 
$$\Gamma^\IndCoh({}'\!\LS^{\on{arithm}}_\cG,\omega_{'\!\LS^{\on{arithm}}_\cG}).$$

\medskip

Let 
$$\sFunct_c(\Bun_G(\BF_q),\ol\BQ_\ell)_{\on{irred}}\subset \sFunct_c(\Bun_G(\BF_q),\ol\BQ_\ell)$$
be the subspace corresponding to \eqref{e:irred autom} under the isomorphism $\sL_G$.

\sssec{}

Equivalently,
$$\sFunct_c(\Bun_G(\BF_q),\ol\BQ_\ell)_{\on{irred}}=
\Gamma(\LS^{\on{arithm}}_{\cG,\on{irred}},\CO_{\LS^{\on{arithm}}_{\cG,\on{irred}}})
\underset{\Gamma(\LS^{\on{arithm}}_\cG,\CO_{\LS^{\on{arithm}}_\cG})}\otimes 
\sFunct_c(\Bun_G(\BF_q),\ol\BQ_\ell),$$
where we view $\Gamma(\LS^{\on{arithm}}_{\cG,\on{irred}},\CO_{\LS^{\on{arithm}}_{\cG,\on{irred}}})$ as a direct factor of
$\Gamma(\LS^{\on{arithm}}_\cG,\CO_{\LS^{\on{arithm}}_\cG})$.

\sssec{}

Let 
\begin{equation} \label{e:cusp}
\sFunct_{\on{cusp}}(\Bun_G(\BF_q),\ol\BQ_\ell)\subset \sFunct_c(\Bun_G(\BF_q),\ol\BQ_\ell)
\end{equation} 
be the subspace of cuspidal functions.

\medskip

Denote
$$\sFunct_{\on{cusp}}(\Bun_G(\BF_q),\ol\BQ_\ell)_{\on{non-degen}}:=
\sFunct_{\on{cusp}}(\Bun_G(\BF_q),\ol\BQ_\ell)\cap \sFunct_c(\Bun_G(\BF_q),\ol\BQ_\ell)_{\on{non-degen}}.$$

\sssec{}

We will prove:

\begin{thm} \label{t:non-deg Ramanujan}
The subspaces 
$$\sFunct_{\on{cusp}}(\Bun_G(\BF_q),\ol\BQ_\ell)_{\on{non-degen}} \subset \sFunct_c(\Bun_G(\BF_q),\ol\BQ_\ell)\supset 
\sFunct_c(\Bun_G(\BF_q),\ol\BQ_\ell)_{\on{irred}}$$
coincide.
\end{thm} 

In other words, \thmref{t:non-deg Ramanujan} says that the non-degenerate subspace of cuspidal
automorphic functions is characterized by its support with respect to the action of the excursion algebra
on all automorphic functions. Namely, this is the subspace supported over the irreducible locus. 

\begin{rem} 

Let us emphasize that the proof of \thmref{t:non-deg Ramanujan}, given below, is conditional on \conjref{c:gr}: we are using that the map
\eqref{e:O to omega}, and hence \eqref{e:temp to all}, is injective. 

\end{rem}

\sssec{}

In \secref{ss:cl Ramanujan} we will explain how \thmref{t:non-deg Ramanujan} relates to a particular
case of the Ramanujan-Arthur conjecture. 

\medskip

Note also that \thmref{t:non-deg Ramanujan} is a special case of \cite[Conjecture 12.7]{VLaf}. 
In particular,  \thmref{t:non-deg Ramanujan} says that for elements in $\sFunct_{\on{cusp}}(\Bun_G(\BF_q),\ol\BQ_\ell)_{\on{non-degen}}$,
the Arthur $SL_2$ is trivial. 

\ssec{Proof of \thmref{t:non-deg Ramanujan}}

\sssec{} \label{sss:Ram 1}

Recall (see \cite[Proposition 8.23]{VLaf}) that the subspace \eqref{e:cusp} is characterized as the subspace consisting 
of vectors, on which the excursion algebra acts locally finitely (i.e., the action factors through an ideal
of finite codimension).

\medskip 

Hence, taking into account the isomorphism $\sL_G$, \thmref{t:exc} and the injectivity of \eqref{e:O to omega}  
the assertion of \thmref{t:non-deg Ramanujan} can be strengthened as follows:

\medskip

\noindent{\bf Claim 1:} {\it The subspace of 
$$\Gamma(\LS^{\on{arithm}}_\cG,\CO_{\LS^{\on{arithm}}_\cG})=\on{Exc}(X,\cG)$$ consisting of vectors, on which the action of $\on{Exc}(X,\cG)$
is locally finite, coincides with
$$\Gamma(\LS^{\on{arithm}}_{\cG,\on{irred}},\CO_{\LS^{\on{arithm}}_{\cG,\on{irred}}}).$$} 

\sssec{}

Since $\Gamma(\LS^{\on{arithm}}_\cG,\CO_{\LS^{\on{arithm}}_\cG})$ is nilpotent-free (by \thmref{t:Exc red}), the above subspace
equals 
$$\Gamma(\LS^{\on{arithm,coarse}}_{\cG,\on{discr}},\CO_{\LS^{\on{arithm,coarse}}_{\cG,\on{discr}}}),$$
where:

\begin{itemize}

\item $\LS^{\on{arithm,coarse}}_\cG:=\Spec(\Gamma(\LS^{\on{arithm}}_\cG,\CO_{\LS^{\on{arithm}}_\cG}))$;

\item $\LS^{\on{arithm,coarse}}_{\cG,\on{discr}}\subset \LS^{\on{arithm,coarse}}_\cG$ is the subscheme equal
to the union of connected components that are isomorphic to the point-scheme.

\end{itemize} 

\medskip

Let $\fr^{\on{arithm}}$ denote the tautological projection
$$\LS^{\on{arithm}}_\cG\to \LS^{\on{arithm,coarse}}_\cG.$$

We can reformulate the Claim in \secref{sss:Ram 1} as follows: 

\medskip

\noindent{\bf Claim 2:} {\it The map $\fr^{\on{arithm}}$ induces a bijection between connected components of $\LS^{\on{arithm}}_{\cG,\on{irred}}$
and $\LS^{\on{arithm,coarse}}_{\cG,\on{discr}}$.}

\sssec{} \label{sss:sigma}

One direction is clear: a connected component of  $\LS^{\on{arithm}}_\cG$ of the form 
$\on{pt}/\Gamma$ maps to a connected component of $\LS^{\on{arithm,coarse}}_{\cG,\on{discr}}$ isomorphic
to $\on{pt}$.

\medskip

For the opposite implication we have to show that if $\sigma$ is a Weil $\cG$-local system on $X$, such that
$\fr^{\on{artihm}}(\sigma)$ belongs to a point component of $\LS^{\on{arithm,coarse}}_\cG$, then $\sigma$
is in fact irreducible.

\sssec{}

Recall (see, e.g., \cite[Lemma 11.9]{VLaf}\footnote{The quoted reference treats the moduli space of $n$-tuples in $\cG$ up to 
conjugation. However, this case implies that the description of the points of the coarse moduli space for $\bMaps(H,\cG)/\on{Ad}(\cG)$,
for a (pro-)algebraic group $H$, of which $\LS^{\on{arithm,coarse}}_\cG$ is a particular case.}) that $\ol\BQ_\ell$-points of $\LS^{\on{arithm,coarse}}_\cG$ are in bijection
with isomorphism classes of \emph{semi-simple} Weil $\cG$-local systems on $X$, and the map 
$\fr^{\on{arithm}}$ sends a given local system to the isomorphism class of its semi-simplification.

\medskip

Hence, we can assume that the local system $\sigma$ in \secref{sss:sigma} is semi-simple. Let us view
it as a homomorphism
$$\on{Weil}(X)\to \cG.$$

Suppose that $\sigma$ is \emph{not} irreducible. Then its image in $\cG$ is contained in a proper Levi 
subgroup; denote it $\cM$. In particular, the image of $\sigma$ is centralized by the connected 
center $Z^0_\cM$ of $\cM$. 

\medskip

Note that we can produce a map from $Z^0_\cM$ to $\LS^{\on{arithm}}_\cG$, i.e., a family of homomorphisms
$$\on{Weil}(X)\to \cG$$
parameterized by points of $Z^0_\cM$, where each such homomorphism 
is equal to $\sigma$ on $\pi_{1,\on{geom}}(X)$ and where we 
modify the image of the Frobenius element in 
$$\on{Weil}(X)/\pi_{1,\on{geom}}(X)\simeq \BZ$$
by the given element of $Z^0_\cM$. 

\medskip

We claim that the resulting map from $Z^0_\cM$ to $\LS^{\on{arithm,coarse}}_\cG$ is \emph{non-constant}.
(That would imply that the image of $\sigma$ could not have been a point component of $\LS^{\on{arithm,arithm}}_\cG$.)

\sssec{}

To prove the above claim, it is enough to show that the family of elements of $\cG/\!/\on{Ad}(\cG)$ we obtain by composing
$$\{*\} \to \on{Weil}(X) \overset{\text{translate by an element of }Z^0_\cM}\longrightarrow \cG\to \cG/\!/\on{Ad}(\cG)$$
(where the first arrow corresponds to \emph{some} choice of the Frobenius element in $\on{Weil}(X)$) is non-constant. 

\medskip

Since the map $\cM/\!/\on{Ad}(\cM)\to \cG/\!/\on{Ad}(\cG)$ is finite, it suffices to show that the resulting family of maps
$$\{*\} \to \on{Weil}(X) \to \cM\to \cM/\!/\on{Ad}(\cM)$$
is non-constant. 

\medskip

However, this is evident since the map
$$\{*\} \to \on{Weil}(X) \to \cM\to \cM/\!/\on{Ad}(\cM)\to \cM/[\cM,\cM]$$
is equal (up to translation) to the composition
$$\{*\} \overset{\text{given element of }Z^0_\cM}\longrightarrow  Z^0_\cM\to \cM/[\cM,\cM],$$
where the second arrow is the natural projection. 

\qed[\thmref{t:non-deg Ramanujan}]

\ssec{The rationality conjecture}

In this subsection we continue to assume the validity of \conjref{c:gr} and hence of \corref{c:filtr}. 

\sssec{}

Consider the filtration
\begin{equation} \label{e:filtr}
Y\rightsquigarrow \sFunct_c(\Bun_G(\BF_q),\ol\BQ_\ell)_Y, \quad Y\subset \Nilp
\end{equation} 
on $\sFunct_c(\Bun_G(\BF_q),\ol\BQ_\ell)$.

\medskip

Note that the $\ol\BQ_\ell$-vector space $\sFunct_c(\Bun_G(\BF_q),\ol\BQ_\ell)$ has a natural $\BQ$-structure:
$$\sFunct_c(\Bun_G(\BF_q),\ol\BQ_\ell)\simeq
\ol\BQ_\ell\underset{\BQ}\otimes \sFunct_c(\Bun_G(\BF_q),\BQ).$$

\sssec{}

The following conjecture was proposed jointly with 
D.~Kazhdan\footnote{Note that all nilpotent orbits in $\cg$ are defined over $\BQ$, and hence so are all subsets $Y$.}:

\begin{conj} \label{c:filtr rat}
The filtration \eqref{e:filtr} is defined over $\BQ$.
\end{conj} 

\sssec{} \label{sss:filtr rat}

Assuming \conjref{c:filtr rat}, we will denote the resulting filtration on $\sFunct_c(\Bun_G(\BF_q),\BQ)$ by
$$Y\rightsquigarrow \sFunct_c(\Bun_G(\BF_q),\BQ)_Y.$$ 

For a given nilpotent orbit $\bO$, denote
$$\sFunct_c(\Bun_G(\BF_q),\BQ)_\bO:=\sFunct_c(\Bun_G(\BF_q),\BQ)_{\ol\bO}/\sFunct_c(\Bun_G(\BF_q),\BQ)_{\ol\bO-\bO}.$$

\medskip

For a field $\BQ\subset \sk$, we will denote by
$$\sFunct_c(\Bun_G(\BF_q),\sk)_Y \text{ and } \sFunct_c(\Bun_G(\BF_q),\sk)_\bO$$
the corresponding subspace/subquotient of $\sFunct_c(\Bun_G(\BF_q),\sk)$. 

\medskip

The cases of interest are $\sk=\ol\BQ$ and $\sk=\BC$. 

\sssec{}

A particular case of \conjref{c:filtr rat} is:

\begin{conj} \label{c:temp rat}
The subspace 
$$\sFunct_c(\Bun_G(\BF_q),\ol\BQ_\ell)_{\on{temp}} \overset{\text{\corref{c:temp=nondeg}}}= \sFunct_c(\Bun_G(\BF_q),\ol\BQ_\ell)_{\on{non-degen}}$$
of $\sFunct_c(\Bun_G(\BF_q),\ol\BQ_\ell)$
is defined over $\BQ$.
\end{conj} 

\sssec{}

%We will now relate \conjref{c:temp rat} to a (particular case of the) conjecture stated in \cite[???]{VLaf}. 
Recall (assuming \conjref{c:gr}, and hence the injectivity of \eqref{e:temp to all}) that the action of
$$\Gamma({}'\!\LS^{\on{arithm}}_\cG,\CO_{'\!\LS^{\on{arithm}}_\cG})$$
on $\poinc^{\on{Vac}}_!$ defines an isomorphism
\begin{equation} \label{e:act on poinc}
\Gamma({}'\!\LS^{\on{arithm}}_\cG,\CO_{'\!\LS^{\on{arithm}}_\cG})\simeq \sFunct_c(\Bun_G(\BF_q),\ol\BQ_\ell)_{\on{non-degen}}.
\end{equation} 

It is easy to see (e.g., by the method of \cite[Sect. 7.1]{GR}) that the element $\poinc^{\on{Vac}}_!$ in fact belongs to $\sFunct_c(\Bun_G(\BF_q),\BQ)$. 

\medskip

Thus, \conjref{c:temp rat} admits the following strengthening: 
 
\begin{conj} \label{c:exc rat prel}
The algebra $\Gamma({}'\!\LS^{\on{arithm}}_\cG,\CO_{'\!\LS^{\on{arithm}}_\cG})$ possesses a rational structure, 
uniquely characterized by the property that the action of a rational element sends $\poinc^{\on{Vac}}_!$ to 
a rational element of $\sFunct_c(\Bun_G(\BF_q),\ol\BQ_\ell)$.
\end{conj}

\sssec{}

We now note that a rational structure on 
$$\on{Exc}(X,\cG)\simeq \Gamma(\LS^{\on{arithm}}_\cG,\CO_{\LS^{\on{arithm}}_\cG})$$
has been constructed in \cite[Theorem 5.1.2]{GaLiRe}. 

\medskip

Following \cite[Conjecture 12.12]{VLaf}, we propose the following strengthening of \conjref{c:exc rat prel}:

\begin{conj} \label{c:exc rat}
The action of $\on{Exc}(X,\cG)$ on $\sFunct_c(\Bun_G(\BF_q),\ol\BQ_\ell)$ is compatible with the rational structures.
\end{conj}

The following results from \cite[Theorem 5.1.9]{GaLiRe}:

\begin{thm}
The statement of \conjref{c:exc rat} holds for $G=GL_n$.
\end{thm}

\section{Relation to the Arthur and Ramanujan-Arthur conjectures} \label{s:Arth}

The material in this subsection consists entirely of conjectures, and should be viewed as a more detailed version 
of \cite{Ras}. 

\medskip

The main motif is the interaction of the filtration \eqref{e:autom filtr}, which is of cohomological nature, 
with the spectral\footnote{Here ``spectral" means "in the sense of eigenvalues".} properties of the action 
of $\on{Exc}(\cG,X)$ on the space of automorphic functions. 

\medskip

Some of our conjectures are of \emph{algebraic} nature, and as such, they pertain to the subspace of
automorphic functions (or subquotients thereof) on which $\on{Exc}(\cG,X)$ acts locally finitely. 
One of these conjectures is the Ramanujan-Arthur conjecture.

\medskip

In order to be able to talk about spectral properties of the entire space of automorphic functions,
we need to place ourselves in the analytic context. 

\ssec{The Arthur conjecture}

In this subsection we recall the statement of Arthur's conjecture on the decomposition of $L^2(\Bun_G(\BF_q))$.

\sssec{}

Let $\bO\in \cg$ be a nilpotent orbit. Choose the corresponding Jacobson-Morozov 
homomorphism
$$\phi_\bO:SL_2\to \cG;$$
it is well-defined up to conjugacy. Denote $H_\bO:=Z_\cG(\phi_\bO)$. 

\medskip

Denote by $q_\bO\in \cG(\BC)$ the image under $\phi_\bO$ of the element $q^{\frac{1}{2}}\in \BG_m(\BC)\subset SL_2(\BC)$.
Let $H^c_\bO(\BC)\subset H_\bO(\BC)$ be a maximal compact subgroup. Consider the subset
\begin{equation} \label{e:shifted compact}
q_\bO\cdot H^c_\bO(\BC)\subset \cG(\BC);
\end{equation}
it is well-defined up to conjugacy by elements of $\cG(\BC)$. 

\medskip

For example, when $\bO=\{0\}$, we have $q_\bO=1$, $H_\bO=\cG$ and \eqref{e:shifted compact} is $\cG^c(\BC)$. In the other extreme,
when $\bO$ is the regular nilpotent, we have $H_\bO=Z(\cG)$, and (say, if $G$ is semi-simple)
the subset \eqref{e:shifted compact} is $q_\bO\cdot Z(\cG)(\BC)$.

\sssec{} \label{sss:subset R}

Let 
$$\sR_{q,\bO,\BC}\subset (\cG/\!/\on{Ad}(\cG))(\BC)$$
be the image of \eqref{e:shifted compact} under the projection
$$\cG(\BC)\to (\cG/\!/\on{Ad}(\cG))(\BC).$$

\begin{lem} 
For $\bO_1\neq \bO_2$, the subsets $\sR_{q,\bO_1,\BC}$ and $\sR_{q,\bO_2,\BC}$
are disjoint. 
\end{lem}

\begin{proof}

The subset \eqref{e:shifted compact} consists of semi-simple elements. So it is enough to show that 
any element $g$ from it determines the nilpotent conjugacy class $\bO$ uniquely. 

\medskip

Let $\fg^2$ (resp., $\fg^0$) be the subspace of $\fg$ consisting of elements on which the eigenvalues of $g$
have absolute value $q$ (resp., $1$). The subspace $\fg^0$ is closed under Jordan decomposition; let
$G^0\subset G$ be the corresponding algebraoc subgroup; it acts on $\fg^2$ by conjugaction. 

\medskip

Then $G^0$ has a unique open orbit on $\fg^2$. Elements on this orbit belong to $\bO$. 

\end{proof} 

\sssec{}

For a place $x\in |X|$, let $\Sph(G)_{x,\BQ}$ denote the spherical Hecke algebra, i.e.,
$$\sFunct_c(G(\CO_x)\backslash G(\CK_x)/G(\CO_x),\BQ),$$
where $\CO_x\subset \CK_x$ are the local ring and field at $x$, respectively. For $K\supset \BQ$, 
denote
$$\Sph(G)_{x,K}:=K\underset{\BQ}\otimes \Sph(G)_{x,\BQ}.$$

\medskip

The algebra $\Sph(G)_{x,\BC}$ is commutative and acts on $L^2(\Bun_G(\BF_q))$ by bounded normal operators. 
(For $x_1\neq x_2$, the actions of $\Sph(G)_{x_1,\BC}$ and $\Sph(G)_{x_2,\BC}$ on $L^2(\Bun_G(\BF_q))$
commute.)

\medskip

Hence, we talk about the joint spectrum of the elements from $\Sph(G)_{x,\BC}$ acting on $L^2(\Bun_G(\BF_q))$. 

\sssec{}

Recall also Satake isomorphism
\begin{equation} \label{e:Sat}
\Sph(G)_{x,\BQ}\simeq \CO_{\cG/\!/\on{Ad}(\cG),\BQ}.
\end{equation} 

Thus, we can think of the spectrum of $\Sph(G)_{x,\BC}$ on $L^2(\Bun_G(\BF_q))$
as a subset of $(\cG/\!/\on{Ad}(\cG))(\BC)$. 

\sssec{}

Arthur's conjecture (see \cite[Conjectures 2 and 4]{Clo}) reads:

\begin{conj} \label{c:Arthur}
There exists an orthogonal decomposition 
\begin{equation} \label{e:Arthur}
L^2(\Bun_G(\BF_q))\simeq \underset{\bO}\oplus\, L^2(\Bun_G(\BF_q))_\bO,
\end{equation} 
uniquely characterized by the property that for every $x\in |X|$, the spectrum of $\Sph(G)_{x,\BC}$ acting on $L^2(\Bun_G(\BF_q))$
is contained in the subset
$$\sR_{q_x,\bO,\BC}\subset (\cG/\!/\on{Ad}(\cG))(\BC),$$
where $q_x$ is the cardinality of the residue field at $x$.
\end{conj} 

In what follows we will use the notation
$$L^2(\Bun_G(\BF_q))_{\on{temp}}:=L^2(\Bun_G(\BF_q))_{\{0\}}.$$

\begin{rem}

A special case of \conjref{c:Arthur} on the subspace of $L^2(\Bun_G(\BF_q))$ generated by principal Eisenstein 
series (from characters of the torus trivial on compact ideles) has been settled by D.~Kazhdan and A.~Okounkov
in \cite{KO}.

\end{rem}

\ssec{The discrete spectrum and the Ramanujan-Arthur conjecture}

In this subsection we continue to work in the analytic context. 

\medskip

We let $G$ be semi-simple and we assume the validity of \conjref{c:Arthur}. 

\sssec{}

Let 
$$L^2_{\on{discr}}(\Bun_G(\BF_q))\subset L^2(\Bun_G(\BF_q))$$
be the subspace of \emph{discrete automorphic functions}
(some people refer to this subspace as \emph{discrete series}). 

\medskip

I.e., this is the maximal Hilbert subspace on which the Hecke algebras $\Sph(G)_{x,\BC}$, $x\in |X|$
act locally finitely\footnote{Note that \emph{if} we impose this locally finiteness condition on \emph{compactly supported}
functions, we obtain the subspace of \emph{cuspidal} functions, see, e.g., \cite[Proposition 8.23]{VLaf}.}.

\medskip

It follows formally from \conjref{c:Arthur} that we have a decomposition
$$L^2_{\on{discr}}(\Bun_G(\BF_q))\simeq \underset{\bO}\oplus\, L^2_{\on{discr}}(\Bun_G(\BF_q))_\bO,$$
where
\begin{equation} \label{e:decomp disc}
L^2_{\on{discr}}(\Bun_G(\BF_q))_\bO=L^2_{\on{discr}}(\Bun_G(\BF_q))\cap L^2(\Bun_G(\BF_q))_\bO.
\end{equation}

\sssec{} \label{sss:cusp O}

Consider now the subspace
$$\sFunct_{\on{cusp}}(\Bun_G(\BF_q),\BC)\subset \sFunct_c(\Bun_G(\BF_q),\BC)\subset L^2(\Bun_G(\BF_q)).$$

Since the Hecke algebras $\Sph(G)_{x,\BC}$ acts locally finitely on $\sFunct_{\on{cusp}}(\Bun_G(\BF_q),\BC)$, we have
$$\sFunct_{\on{cusp}}(\Bun_G(\BF_q),\BC)\subset L^2_{\on{discr}}(\Bun_G(\BF_q)).$$

\medskip

Denote 
$$\sFunct_{\on{cusp}}(\Bun_G(\BF_q),\BC)_\bO:=\sFunct_{\on{cusp}}(\Bun_G(\BF_q),\BC)\cap L^2_{\on{discr}}(\Bun_G(\BF_q))_\bO.$$

By definition, $\sFunct_{\on{cusp}}(\Bun_G(\BF_q),\BC)_\bO$ is the subspace of $\sFunct_{\on{cusp}}(\Bun_G(\BF_q),\BC)$ consisting 
of elements, whose support in $\on{Specm}(\Sph(G)_{x,\BC})\simeq (\cG/\!/\on{Ad}(\cG))(\BC)$ is contained in $\sR_{q_x,\bO,\BC}$
for all $x\in |X|$. 

\sssec{}

Hence, from \eqref{e:decomp disc} we obtain:

\begin{conj} \label{c:Ram-Pet}
The inclusion
$$\underset{\bO}\oplus\, \sFunct_{\on{cusp}}(\Bun_G(\BF_q),\BC)_\bO\to \sFunct_{\on{cusp}}(\Bun_G(\BF_q),\BC)$$
is an isomorphism. 
\end{conj} 

Note that \conjref{c:Ram-Pet} is nothing but the Ramanujan-Arthur conjecture in the
everywhere unramified case.

\ssec{Relation between Conjectures \ref{c:Arthur} and \ref{c:filtr rat}} 

In this subsection we continue to work in the analytic context. 

\medskip

We will assume the validity Conjectures \ref{c:Arthur} and \ref{c:filtr rat}, and formulate
one more conjecture pertaining to the interaction between the two. 

\sssec{}

Recall the notations from \secref{sss:filtr rat}.  The following conjecture was proposed jointly with D.~Kazhdan:

\begin{conj} \label{c:Arth alg}
The closure of $\sFunct_c(\Bun_G(\BF_q),\BC)_Y$ in $L^2(\Bun_G(\BF_q))$ equals 
$$\underset{\bO\subset Y}\oplus\, L^2(\Bun_G(\BF_q))_\bO.$$
\end{conj} 

Note that we can reformulate \conjref{c:Arth alg} as follows:

\begin{conj} \hfill  \label{c:Arth alg prime}

\smallskip

\noindent{\em(a)}
The subspace $\sFunct_c(\Bun_G(\BF_q),\BC)_Y\subset \sFunct_c(\Bun_G(\BF_q),\BC)$
equals the intersection 
$$ \sFunct_c(\Bun_G(\BF_q),\BC) \cap \left(\underset{\bO\subset Y}\oplus\, L^2(\Bun_G(\BF_q))_\bO\right),$$
viewed as subsets of $L^2(\Bun_G(\BF_q))$.

\smallskip

\noindent{\em(b)}
For every nilpotent orbit $\bO$, the map
$$\sFunct_c(\Bun_G(\BF_q),\BC)_\bO=\sFunct_c(\Bun_G(\BF_q),\BC)_{\ol\bO}/\sFunct_c(\Bun_G(\BF_q),\BC)_{\ol\bO-\bO}\to 
L^2(\Bun_G(\BF_q))_\bO$$
has a dense image.

\end{conj} 

\sssec{}

For the rest of this subsection, we will assume that $G$ is semi-simple. 

\medskip

Fix an nilpotent orbit $\bO$ and let
$$\sFunct_{c,\on{discr}}(\Bun_G(\BF_q),\BQ)_\bO\subset \sFunct_c(\Bun_G(\BF_q),\BQ)_\bO$$
be the subspace consisting of elements that are locally finite with respect to the action of
$\Sph(G)_{x,\BQ}$, $x\in |X|$. 

\begin{rem} \label{r:fake discr}

We will refer to the space
$$\underset{\bO}\oplus\, \sFunct_{c,\on{discr}}(\Bun_G(\BF_q),\BQ)_\bO$$
as \emph{fake} discrete automorphic functions, cf. Remark \ref{r:fake}. 

\end{rem} 

\sssec{}

Denote
$$\sFunct_{c,\on{discr}}(\Bun_G(\BF_q),\BC)_\bO:=\BC\underset{\BQ}\otimes \sFunct_{c,\on{discr}}(\Bun_G(\BF_q),\BQ)_\bO.$$

In other words,
$$\sFunct_{c,\on{discr}}(\Bun_G(\BF_q),\BC)_\bO\subset \sFunct_c(\Bun_G(\BF_q),\BC)_\bO$$
is the subspace consisting of elements that are locally finite with respect to the action of
$\Sph(G)_{x,\BC}$, $x\in |X|$. 

\sssec{}

By \conjref{c:Arth alg prime}(a), we have an embedding
$$\sFunct_c(\Bun_G(\BF_q),\BC)_\bO\hookrightarrow L^2(\Bun_G(\BF_q))_\bO.$$

This map automatically restricts to an embedding
\begin{equation} \label{e:Arth alg}
\sFunct_{c,\on{discr}}(\Bun_G(\BF_q),\BC)_\bO\to L^2_{\on{discr}}(\Bun_G(\BF_q))_\bO.
\end{equation} 

\sssec{}

We propose:

\begin{conj} \label{c:alg discr}
The map \eqref{e:Arth alg} is an isomorphism.
\end{conj} 

\sssec{}

We now consider the interaction of the subspace of cuspidal functions with the filtration. Namely, 
from \conjref{c:Arth alg prime}(a) we obtain:

\medskip

\begin{conj} \label{c:filtr and Arthur and cusp C} 
Let $Y\subset \Nilp$ be a closed Ad-invariant subset. Then the subspace 
$$\sFunct_{\on{cusp}}(\Bun_G(\BF_q),\BC)\cap \sFunct_c(\Bun_G(\BF_q),\BC)_Y\subset \sFunct_{\on{cusp}}(\Bun_G(\BF_q),\BC)$$
equals $$\underset{\bO\subset Y}\oplus\, \sFunct_{\on{cusp}}(\Bun_G(\BF_q),\BC)_\bO,$$
where $\sFunct_{\on{cusp}}(\Bun_G(\BF_q),\BC)_\bO$ is as in \secref{sss:cusp O}. 
\end{conj} 

\begin{rem}

Note that from \conjref{c:filtr and Arthur and cusp C} we obtain that the composition
\begin{multline*}
\sFunct_{\on{cusp}}(\Bun_G(\BF_q),\BC)_\bO\to \sFunct_c(\Bun_G(\BF_q),\BC)_{\ol\bO}\twoheadrightarrow \\
\to \sFunct_c(\Bun_G(\BF_q),\BC)_{\ol\bO}/\sFunct_c(\Bun_G(\BF_q),\BC)_{\ol\bO-\bO}=
\sFunct_c(\Bun_G(\BF_q),\BC)_{\bO}
\end{multline*} 
is injective. Its image is clearly contained in the subspace
$$\sFunct_{c,\on{discr}}(\Bun_G(\BF_q),\BC)_{\bO}\subset \sFunct_c(\Bun_G(\BF_q),\BC)_{\bO},$$
so we obtain an embedding
\begin{equation} \label{e:cusp to alg discr}
\sFunct_{\on{cusp}}(\Bun_G(\BF_q),\BC)_\bO\hookrightarrow \sFunct_{c,\on{discr}}(\Bun_G(\BF_q),\BC)_{\bO}.
\end{equation}

This embedding is \emph{not} necessarily an isomorphism: an element in $\sFunct_c(\Bun_G(\BF_q),\BC)_{\bO}$ 
that is locally finite with respect to $\Sph(G)_x$ may not be liftable to an element of 
$\sFunct_c(\Bun_G(\BF_q),\BC)_{\ol\bO}$ with the same property.

\medskip

Note also that the composition of \eqref{e:cusp to alg discr} with the embedding 
$$\sFunct_{c,\on{discr}}(\Bun_G(\BF_q),\BC)_{\bO}
\hookrightarrow L^2_{\on{discr}}(\Bun_G(\BF_q))_\bO$$
(which, according to \conjref{c:alg discr}, is an isomorphism) is 
the tautological embedding
$$\sFunct_{\on{cusp}}(\Bun_G(\BF_q),\BC)_\bO\hookrightarrow L^2_{\on{discr}}(\Bun_G(\BF_q))_\bO.$$

Thus, the failure of \eqref{e:cusp to alg discr} to be an isomorphism is the difference between cuspidal
functions and the space of discrete automorphic functions (within the given $\bO$-direct summand). 

\medskip

(For example, for the $\bO$ being the regular nilpotent orbit, the space $L^2_{\on{discr}}(\Bun_G(\BF_q))_\bO$
is one-dimensional and is spanned by the constant function, which is obviously not cuspidal.) 

\end{rem} 

\ssec{Rational and \texorpdfstring{$\ell$}{ell}-adic versions of the Ramanujan-Arthur conjecture}

In this subsection we will assume that $G$ is semi-simple.

\medskip

We will assume \conjref{c:Arthur}. 
We will deduce consequences pertaining to the structure of $\sFunct_{\on{cusp}}(\Bun_G(\BF_q),\BQ)$ and 
$\sFunct_{\on{cusp}}(\Bun_G(\BF_q),\ol\BQ_\ell)$.

\sssec{} \label{sss:Weil as eigen}

Let $H$ be a reductive group defined over $\BQ$. We let
$$H^c(\ol\BQ)\subset H(\ol\BQ)$$
be the subset consisting of semi-simple elements, such that for every embedding $\ol\BQ\hookrightarrow \BC$,
the resulting element of $H(\BC)$ is compact (i.e., is conjugate to an element of the maximal compact subgroup).

\medskip

Let 
$$H^{c,\on{Weil}}(\ol\BQ)\subset H^c(\ol\BQ)$$
be the subset obtained by imposing the compactness condition for every $\ol\BQ\hookrightarrow \ol\BQ_\ell$, $\ell\neq p$. 

\medskip

For example, for $H=\BG_m$, the subset $\BG_m^{c,\on{Weil}}(\ol\BQ)$ is that of Weil numbers of weight $0$. The subset 
$\BG_m^c(\ol\BQ)$ is what is called in \cite{GaLiRe} ``mock Weil numbers" of weight $0$ (i.e., we impose the norm $=1$ condition 
only at the Archimedean places). 

\medskip

In general, an element $h\in H(\ol\BQ)$ belongs to $H^c(\ol\BQ)$ (resp., $\BG_m^{c,\on{Weil}}(\ol\BQ)$) 
if and only for every/some maximal
torus $h\in T_H\subset H$ and every character $\chi:T_H\to \BG_m$, the image of $h$ under $\chi$ is a 
mock Weil number (resp., Weil number) of weight $0$. 

\sssec{}

In the context of \secref{sss:subset R}, let 
$$\sR_{q,\bO,\ol\BQ}\subset (\cG/\!/\on{Ad}(\cG))(\ol\BQ)$$
be the image of
$$q_\bO\cdot H^c_\bO(\ol\BQ)\subset \cG(\ol\BQ)$$
under the projection
$$\cG(\ol\BQ)\to (\cG/\!/\on{Ad}(\cG))(\ol\BQ).$$

By construction, this subset is $\on{Gal}(\ol\BQ/\BQ)$-invariant; hence it corresponds to a subset
$$\sR_{q,\bO,\BQ}\subset \on{Specm}(\CO_{\cG/\!/\on{Ad}(\cG),\BQ}).$$

\medskip

Let 
$$\sR^{\on{Weil}}_{q,\bO,\ol\BQ}\subset (\cG/\!/\on{Ad}(\cG))(\ol\BQ) \text{ and } \sR^{\on{Weil}}_{q,\bO,\BQ}\subset \on{Specm}(\CO_{\cG/\!/\on{Ad}(\cG),\BQ})$$
be the corresponding subsets, where we replace $H^c_\bO(\ol\BQ)$ by $H^{c,\on{Weil}}_\bO(\ol\BQ)$. 

\sssec{}

Let
$$\sFunct_{\on{cusp}}(\Bun_G(\BF_q),\BQ)_\bO \subset \sFunct_{\on{cusp}}(\Bun_G(\BF_q),\BQ)$$
be the subspace consisting of elements, whose support in 
$$\on{Specm}(\Sph(G)_{x,\BQ})\simeq \on{Specm}(\CO_{\cG/\!/\on{Ad}(\cG),\BQ})$$
is contained in $\sR_{q_x,\bO,\BQ}$ for every $x\in |X|$.

\medskip

Denote
$$\sFunct_{\on{cusp}}(\Bun_G(\BF_q),\ol\BQ_\ell)_\bO:=\ol\BQ_\ell\underset{\BQ}\otimes \sFunct_{\on{cusp}}(\Bun_G(\BF_q),\BQ)_\bO.$$

\medskip

Let
$$\sFunct_{\on{cusp}}(\Bun_G(\BF_q),\BQ)^{\on{Weil}}_\bO \subset \sFunct_{\on{cusp}}(\Bun_G(\BF_q),\BQ)_\bO$$
and 
$$\sFunct_{\on{cusp}}(\Bun_G(\BF_q),\ol\BQ_\ell)^{\on{Weil}}_\bO\subset \sFunct_{\on{cusp}}(\Bun_G(\BF_q),\ol\BQ_\ell)_\bO$$
be the corresponding subsets, where we replace $\sR_{q_x,\bO,\BQ}$ by $\sR^{\on{Weil}}_{q_x,\bO,\BQ}$. 

\sssec{}

We obtain that \conjref{c:Ram-Pet} is equivalent to the following:

\begin{conj} \label{c:Ramanujan Q} The inclusion
$$\underset{\bO}\oplus\, \sFunct_{\on{cusp}}(\Bun_G(\BF_q),\BQ)_\bO\to \sFunct_{\on{cusp}}(\Bun_G(\BF_q),\BQ)$$
is an isomorphism. 
\end{conj} 

\medskip

In its turn, the statement of \conjref{c:Ramanujan Q} is equivalent to the following: 

\begin{conj} \label{c:Ramanujan Q ell} 
The inclusion
$$\underset{\bO}\oplus\, \sFunct_{\on{cusp}}(\Bun_G(\BF_q),\ol\BQ_\ell)_\bO\to \sFunct_{\on{cusp}}(\Bun_G(\BF_q),\ol\BQ_\ell)$$
is an isomorphism. 
\end{conj} 

Note that the \emph{statements} of Conjectures \ref{c:Ramanujan Q} and \ref{c:Ramanujan Q ell} are \emph{not}
conditional on \conjref{c:Arthur}.

\begin{rem}

We perceive \conjref{c:Ramanujan Q ell} as the $\ell$-adic version of the Ramanujan-Arthur conjecture. 

\medskip

As was noted above, it is equivalent to its rational version, i.e., \conjref{c:Ramanujan Q}, and the latter implies
the version with $\BC$-coefficients, i.e., \conjref{c:Ram-Pet}.

\end{rem}

\sssec{}

Let us now also assume Conjectures \ref{c:filtr rat} and \ref{c:Arth alg prime}(a). We obtain that 
\conjref{c:filtr and Arthur and cusp C} is equivalent to the following: 

\begin{conj} \label{c:filtr and Arthur and cusp} Let $Y\subset \Nilp$ a closed Ad-invariant subset. Then the intersection
$$\sFunct_{\on{cusp}}(\Bun_G(\BF_q),\BQ)\cap \sFunct_c(\Bun_G(\BF_q),\BQ)_Y\subset \sFunct_{\on{cusp}}(\Bun_G(\BF_q),\BQ)$$
equals 
$$\underset{\bO\subset Y}\oplus\, \sFunct_{\on{cusp}}(\Bun_G(\BF_q),\BQ)_\bO.$$
\end{conj} 

\sssec{}

Note that while the \emph{statement} of \conjref{c:filtr and Arthur and cusp} depends on
\conjref{c:filtr rat}, its $\ol\BQ_\ell$-counterpart (which is equivalent to it) does not:

\begin{conj} \label{c:filtr and Arthur and cusp ell}
$$\sFunct_{\on{cusp}}(\Bun_G(\BF_q),\ol\BQ_\ell)\cap \sFunct_c(\Bun_G(\BF_q),\ol\BQ_\ell)_Y\subset \sFunct_{\on{cusp}}(\Bun_G(\BF_q),\ol\BQ_\ell)$$
equals 
$$\underset{\bO\subset Y}\oplus\, \sFunct_{\on{cusp}}(\Bun_G(\BF_q),\ol\BQ_\ell)_\bO.$$
\end{conj} 

\sssec{}

In addition, we propose:

\begin{conj} \label{c:mock Weil}
The inclusion
$$\sFunct_{\on{cusp}}(\Bun_G(\BF_q),\BQ)^{\on{Weil}}_\bO \subset \sFunct_{\on{cusp}}(\Bun_G(\BF_q),\BQ)_\bO$$
is an isomorphism.
\end{conj}

As above, the statement of \conjref{c:mock Weil} is conditional on \conjref{c:filtr rat}, whereas its $\ol\BQ_\ell$-counterpart is not:

\medskip

\begin{conj} \label{c:mock Weil Q-ell}
The inclusion
$$\sFunct_{\on{cusp}}(\Bun_G(\BF_q),\ol\BQ_\ell)^{\on{Weil}}_\bO \subset \sFunct_{\on{cusp}}(\Bun_G(\BF_q),\ol\BQ_\ell)_\bO$$
is an isomorphism.
\end{conj}

\ssec{Some converse implications}

In this subsection we assume that $G$ is semi-simple.

\medskip

In the preceding subsections, we deduced the Ramanujan-Arthur conjecture over $\ol\BQ_\ell$ from \conjref{c:Arthur},
which is of analytic nature. 

\medskip

In this subsection we will formulate \conjref{c:support O discr Q ell}, which is of \emph{algebraic nature}, which also implies
the Ramanujan-Arthur conjecture over $\ol\BQ_\ell$. 

\medskip

In its turn, \conjref{c:support O discr Q ell} follows from some previously formulated conjectures. 

\sssec{}

Let us again assume Conjectures \ref{c:Arthur},  \ref{c:filtr rat} and \ref{c:Arth alg prime}(a). 
Using \eqref{e:Arth alg} we obtain: 

\begin{conj} \label{c:support O discr C}
For every $x\in X$, the support of $\sFunct_{c,\on{discr}}(\Bun_G(\BF_q),\BC)_\bO$ in $(\cG/\!/\on{Ad}(\cG))(\BC)$ 
is contained in $\sR_{q_x,\bO,\BC}$.
\end{conj} 

The latter is equivalent to: 

\begin{conj} \label{c:support O discr Q}
For every $x\in X$, the support of $\sFunct_{c,\on{discr}}(\Bun_G(\BF_q),\BQ)_\bO$ in  
$$\on{Specm}(\Sph(G)_{x,\BQ})\simeq \on{Specm}(\CO_{\cG/\!/\on{Ad}(\cG),\BQ})$$
is contained in $\sR_{q_x,\bO,\BQ}$.
\end{conj} 

In fact, we propose the following slightly stronger conjectures:

\begin{conj} \label{c:support O discr C Weil} 
For every $x\in X$, the support of $\sFunct_{c,\on{discr}}(\Bun_G(\BF_q),\BC)_\bO$ in $(\cG/\!/\on{Ad}(\cG))(\BC)$ 
is contained in $\sR^{\on{Weil}}_{q_x,\bO,\BC}$.
\end{conj}

\begin{conj} \label{c:support O discr Q Weil} 
For every $x\in X$, the support of $\sFunct_{c,\on{discr}}(\Bun_G(\BF_q),\BQ)_\bO$ in  
$$\on{Specm}(\Sph(G)_{x,\BQ})\simeq \on{Specm}(\CO_{\cG/\!/\on{Ad}(\cG),\BQ})$$
is contained in $\sR^{\on{Weil}}_{q_x,\bO,\BQ}$.
\end{conj} 

\sssec{}

Let 
$$\sFunct_{c,\on{discr}}(\Bun_G(\BF_q),\ol\BQ_\ell)_\bO\subset \sFunct_{c}(\Bun_G(\BF_q),\ol\BQ_\ell)_\bO$$
be the subspace consisting of elements that are locally finite with respect to the action of $\Sph(G)_{x,\ol\BQ_\ell}$
for all $x\in X$.  

\medskip

This condition is equivalent to be locally finite with respect to the action of $\on{Exc}(X,G)$, and hence also equivalent to
being locally finite with respect to the action of $\Sph(G)_{x,\ol\BQ_\ell}$ for any given $x$
(see \cite[Theorem 4.4.2(a)]{GaLiRe}).

\medskip

Note that if we assume \conjref{c:filtr rat}, then 
$$\sFunct_{c,\on{discr}}(\Bun_G(\BF_q),\ol\BQ_\ell)_\bO:=\ol\BQ_\ell\underset{\BQ}\otimes \sFunct_{c,\on{discr}}(\Bun_G(\BF_q),\BQ)_\bO.$$

\sssec{}

The statement of \conjref{c:support O discr Q} is equivalent to the following:
\begin{conj} \label{c:support O discr Q ell}
For every $x\in X$, the support of $\sFunct_{c,\on{discr}}(\Bun_G(\BF_q),\ol\BQ_\ell)_\bO$ in 
$$\on{Specm}(\Sph(G)_{x,\ol\BQ_\ell})\simeq (\cG/\!/\on{Ad}(\cG))(\ol\BQ_\ell)$$ is contained in 
$$\sR_{q_x,\bO,\BQ}\subset \on{Specm}(\CO_{\cG/\!/\on{Ad}(\cG),\BQ}) \subset (\cG/\!/\on{Ad}(\cG))(\ol\BQ_\ell).$$ 
\end{conj} 

\medskip

Similarly, the statement of \conjref{c:support O discr Q Weil} is equivalent to the following strengthening 
of \conjref{c:support O discr Q ell}: 

\begin{conj} \label{c:support O discr Q ell Weil}
For every $x\in X$, the support of $\sFunct_{c,\on{discr}}(\Bun_G(\BF_q),\ol\BQ_\ell)_\bO$ in 
$$\on{Specm}(\Sph(G)_{x,\ol\BQ_\ell})\simeq (\cG/\!/\on{Ad}(\cG))(\ol\BQ_\ell)$$ is contained in 
$$\sR^{\on{Weil}}_{q_x,\bO,\BQ}\subset \on{Specm}(\CO_{\cG/\!/\on{Ad}(\cG),\BQ}) \subset (\cG/\!/\on{Ad}(\cG))(\ol\BQ_\ell).$$ 
\end{conj} 

As above, the \emph{statement} of \conjref{c:support O discr Q} (resp., \conjref{c:support O discr Q Weil})  
depends on \conjref{c:filtr rat},
while that of \ref{c:support O discr Q ell} (resp., \conjref{c:support O discr Q ell Weil}) does not. 

\sssec{} \label{sss:implications}

Using the embedding
%\begin{multline*} 
%(\sFunct_{\on{cusp}}(\Bun_G(\BF_q),\BQ)\cap \sFunct_c(\Bun_G(\BF_q),\BQ)_{\ol\bO})/ \\
%(\sFunct_{\on{cusp}}(\Bun_G(\BF_q),\BQ)\cap \sFunct_c(\Bun_G(\BF_q),\BQ)_{\ol\bO-\bO})\hookrightarrow 
%\sFunct_{c,\on{discr}}(\Bun_G(\BF_q),\BQ)_\bO,
%\end{multline*}
%and similarly, 
\begin{multline*} 
(\sFunct_{\on{cusp}}(\Bun_G(\BF_q),\ol\BQ_\ell)\cap \sFunct_c(\Bun_G(\BF_q),\ol\BQ_\ell)_{\ol\bO})/ \\
(\sFunct_{\on{cusp}}(\Bun_G(\BF_q),\ol\BQ_\ell)\cap \sFunct_c(\Bun_G(\BF_q),\ol\BQ_\ell)_{\ol\bO-\bO})\hookrightarrow 
\sFunct_{c,\on{discr}}(\Bun_G(\BF_q),\ol\BQ_\ell)_\bO,
\end{multline*}
we obtain: 

\medskip

\begin{enumerate}

\item The statement of \conjref{c:support O discr Q ell} implies that 
of Conjectures \ref{c:Ramanujan Q ell} and \ref{c:filtr and Arthur and cusp ell}
unconditionally (i.e., without assuming any of \ref{c:Arthur},  \ref{c:filtr rat} and \ref{c:Arth alg prime}(a)),
and hence also \conjref{c:Ramanujan Q}. 

\medskip

\item Assuming only \conjref{c:filtr rat}, the statement of \conjref{c:support O discr Q ell}
implies that of \conjref{c:filtr and Arthur and cusp}. 

\medskip

\item The statement of \conjref{c:support O discr Q ell Weil} implies \conjref{c:mock Weil Q-ell} unconditionally.

\medskip

\item Assuming only \conjref{c:filtr rat}, the statement of \conjref{c:support O discr Q Weil} implies \conjref{c:mock Weil}. 

\end{enumerate} 

\sssec{The hierarchy of conjectures}

Here is how we view the hierarchy of difficulty of the above conjectures (which is also the reason we 
care to explain what implies what):

\medskip

\begin{enumerate}

\item We view \conjref{c:filtr rat} as the hardest. But it is reasonable to expect that one can 
deduce it assuming some strong conjectures about motives (e.g., the Tate conjecture). 
However, even if we knew \conjref{c:filtr rat}, we would not know
how to prove \conjref{c:Arth alg prime} or \conjref{c:alg discr}. 

\medskip

\item We believe that \conjref{c:support O discr Q ell Weil} is within reach (see Remark \ref{r:GLR}). This would
imply the validity of \conjref{c:Ramanujan Q ell}, hence \conjref{c:Ramanujan Q} and hence also of \conjref{c:Ram-Pet} 
(i.e., the Ramanujan-Arthur conjecture).

\medskip

\item We are optimistic about the chances of deducing Arthur's conjecture (i.e., \conjref{c:Arthur}) given (the Ramanujan-Arthur)
\conjref{c:Ram-Pet}, by extending the method of \cite{KO} or \cite{LR}. 

\end{enumerate} 

\ssec{Relation to Arthur parameters} \label{ss:cl Ramanujan}

In this subsection we continue to assume that $G$ is semi-simple.

\sssec{} 

Let us relate a particular case of \conjref{c:filtr and Arthur and cusp ell} 
to \thmref{t:non-deg Ramanujan}. 

\medskip

Take $Y=\{0\}$. Note that, on the one hand, the intersection
$$\sFunct_{\on{cusp}}(\Bun_G(\BF_q),\ol\BQ_\ell)\cap \sFunct_c(\Bun_G(\BF_q),\ol\BQ_\ell)_{\{0\}}$$
is the subspace that we denoted earlier by
$$\sFunct_{\on{cusp}}(\Bun_G(\BF_q),\ol\BQ_\ell)_{\on{non-degen}}.$$

\medskip

On the other hand,
$$\sFunct_{\on{cusp}}(\Bun_G(\BF_q),\ol\BQ_\ell)_{\{0\}}$$
is the subspace of $\sFunct_{\on{cusp}}(\Bun_G(\BF_q),\ol\BQ_\ell)$ 
consisting of elements whose support in 
$$\on{Specm}(\CO_{\cG/\!/\on{Ad}(\cG),\ol\BQ_\ell}) = (\cG/\!/\on{Ad}(\cG))(\ol\BQ_\ell)$$
lies in the image of  
\begin{equation} \label{e:Weil 0}
\cG^c(\ol\BQ)\subset \cG(\ol\BQ_\ell)\to (\cG/\!/\on{Ad}(\cG))(\ol\BQ_\ell).
\end{equation} 

Denote this subspace by 
$$\sFunct_{\on{cusp}}(\Bun_G(\BF_q),\ol\BQ_\ell)_{\on{anlt-temp}},$$
where ``anlt" stands for ``analytic". 

\sssec{} 

Thus, the particular case of \conjref{c:filtr and Arthur and cusp ell} for $\bO=\{0\}$ says that we have an equality
$$\sFunct_{\on{cusp}}(\Bun_G(\BF_q),\ol\BQ_\ell)_{\on{non-degen}} = \sFunct_{\on{cusp}}(\Bun_G(\BF_q),\ol\BQ_\ell)_{\on{anlt-temp}}$$
as subspaces in $\sFunct_{\on{cusp}}(\Bun_G(\BF_q),\ol\BQ_\ell)$. 

\medskip

We claim that \thmref{t:non-deg Ramanujan} provides an inclusion in one direction, namely
$$\sFunct_{\on{cusp}}(\Bun_G(\BF_q),\ol\BQ_\ell)_{\on{non-degen}} \subset \sFunct_{\on{cusp}}(\Bun_G(\BF_q),\ol\BQ_\ell)_{\on{anlt-temp}}.$$

Namely, we claim that we have an inclusion 
\begin{equation} \label{e:irred temp}
\sFunct_c(\Bun_G(\BF_q),\ol\BQ_\ell)_{\on{irred}} \subset \sFunct_{\on{cusp}}(\Bun_G(\BF_q),\ol\BQ_\ell)_{\on{anlt-temp}}.
\end{equation} 

\sssec{} \label{sss:pure}

Indeed, for a point $x$, consider the corresponding map
$$\LS_\cG^{\on{arithm}}\to \cG/\on{Ad}(\cG),$$
given by restriction of Weil local systems from $X$ to $x$. 

\medskip

We need to show that the composition
$$\LS_{\cG,\on{irred}}^{\on{arithm}}\hookrightarrow \LS_\cG^{\on{arithm}}\to \cG/\on{Ad}(\cG)\to  \cG/\!/\on{Ad}(\cG)$$
takes values in the image of \eqref{e:Weil 0}. 

\medskip

Unraveling, this is the statement that an irreducible Weil local system $\sigma$ is pointwise pure of weight $0$.
However, this is well-known, see, e.g., the assertion of \cite[Proposition 25.4.6]{AGKRRV1}.

\sssec{}

We now consider the case of a general nilpotent orbit $\bO$. Let
\begin{multline*}
\sFunct_{c,\on{discr}}(\Bun_G(\BF_q),\ol\BQ_\ell)_\bO\subset \sFunct_{c}(\Bun_G(\BF_q),\ol\BQ_\ell)_\bO:= \\
=\sFunct_{c}(\Bun_G(\BF_q),\ol\BQ_\ell)_{\ol\bO}/\sFunct_{c}(\Bun_G(\BF_q),\ol\BQ_\ell)_{\ol\bO-\bO}
\end{multline*}
be the subspace consisting of elements that are locally finite with respect to the action of
$\Sph(G)_{x,\ol\BQ_\ell}$, $x\in |X|$. 

\medskip

Recall the notion of \emph{Arthur parameter}, see \cite[Sect. 12.2.2]{VLaf}. Generalizing \thmref{t:non-deg Ramanujan}, we propose:

\begin{conj} \label{c:locally finite O} \hfill

\medskip

\noindent{\em(a)} 
The support of $\sFunct_{c,\on{discr}}(\Bun_G(\BF_q),\ol\BQ_\ell)_\bO$ in $\LS^{\on{arithm,coarse}}_\cG(\ol\BQ_\ell)$ 
is contained in the (finite) set consisting of Langlands parameters associated to \emph{elliptic} Arthur parameters whose $SL_2$
component corresponds to $\bO$.

\medskip

\noindent{\em(b)} For a given \emph{elliptic} Arthur parameter, the direct summand of $\LS^{\on{arithm,coarse}}_\cG(\ol\BQ_\ell)$
set-theoretically supported at it is at most one-dimensional. 

\end{conj}

\begin{rem} \label{r:GLR}

This conjecture is in fact a theorem-in-progress and is the subject of the forthcoming work \cite{GaLaRa}. 
Note that for $\bO=\{0\}$, the statement of \conjref{c:locally finite O} is equivalent to that of \thmref{t:non-deg Ramanujan}. 

\end{rem} 

\sssec{}

Note that by the same argument as in \secref{sss:pure}, the statement of \conjref{c:locally finite O}(a) 
implies that of \conjref{c:support O discr Q ell}, and in fact that of  \conjref{c:support O discr Q ell Weil}. 

\medskip

Hence, by \secref{sss:implications}, we obtain: 

\begin{conj} \label{c:cusp O} Assume that  \conjref{c:locally finite O}(a) holds. Then: \hfill

\smallskip

\noindent{\em(a)} 
Conjectures \ref{c:Ramanujan Q ell} and \ref{c:filtr and Arthur and cusp ell} hold. 

\smallskip

\noindent{\em(b)} The support of $\sFunct_{\on{cusp}}(\Bun_G(\BF_q),\ol\BQ_\ell)_\bO$ in $\LS^{\on{arithm,coarse}}_\cG(\ol\BQ_\ell)$ 
is contained in the (finite) set consisting of Langlands parameters associated to \emph{elliptic} Arthur parameters whose $SL_2$
component corresponds to $\bO$.
\end{conj}

Note that \conjref{c:cusp O} coincides with \cite[Conjecture 12.7]{VLaf}.

\ssec{Relation to motivic Langlands parameters} 

\sssec{}

Let
$$\LS^{\on{mot,coarse}}_\cG \subset \LS^{\on{arithm,coarse}}_\cG$$
be as in \cite[Sect. 5.4.3]{GaLiRe}\footnote{See {\it loc. cit.}, Remark 5.4.6, for why this locus is called ``motivic".},
which is a reformulation of \cite[Sect. 6.6]{Dr}. 

\medskip

Note that $\LS^{\on{mot,coarse}}_\cG$ is isomorphic to the union of (infinitely many) copies of the point-scheme. 

\sssec{}

From \conjref{c:locally finite O}(a) we obtain:

\begin{conj} \label{c:motivic O Q-ell}
The support of $\sFunct_{c,\on{discr}}(\Bun_G(\BF_q),\ol\BQ_\ell)_\bO$ in $\LS^{\on{arithm,coarse}}_\cG(\ol\BQ_\ell)$ 
is contained in $\LS^{\on{mot,coarse}}_\cG(\ol\BQ_\ell)$. 
\end{conj}

\sssec{}

From \conjref{c:motivic O Q-ell} we obtain: 

\begin{conj} \label{c:cusp in mot}
The support of $\sFunct_{\on{cusp}}(\Bun_G(\BF_q),\ol\BQ_\ell)$ in $\LS^{\on{arithm,coarse}}_\cG(\ol\BQ_\ell)$ 
is contained in $\LS^{\on{mot,coarse}}_\cG(\ol\BQ_\ell)$. 
\end{conj}

\sssec{}

Recall that according to \cite[Sect. 5.4.6]{GaLiRe}, the scheme $\LS^{\on{mot,coarse}}_\cG$ 
(as well as its embedding into $\LS^{\on{arithm,coarse}}_\cG$) are defined over $\BQ$.

\medskip

Assume that \conjref{c:exc rat} holds, so that we can talk about the support of elements of 
 $\sFunct_{\on{cusp}}(\Bun_G(\BF_q),\sk)$ in $\LS^{\on{arithm,coarse}}_\cG(\sk)$ for $\BQ\subset \sk$. 
 
 \medskip
 
 We obtain that \conjref{c:cusp in mot} is equivalent to: 

\begin{conj}  \label{c:motivic Q}
The support of $\sFunct_{\on{cusp}}(\Bun_G(\BF_q),\ol\BQ)$ in $\LS^{\on{arithm,coarse}}_\cG(\ol\BQ)$
is contained in
$$\LS^{\on{mot,coarse}}_\cG(\ol\BQ)\subset \LS^{\on{arithm,coarse}}_\cG(\ol\BQ).$$
\end{conj} 

\sssec{}

Let us assume for a moment Conjectures \ref{c:filtr rat}, \ref{c:exc rat}, \ref{c:Arthur}, \ref{c:Arth alg prime} and \ref{c:alg discr}. 
Then from \conjref{c:motivic Q}, we obtain that the Hecke eigenvalues on $L^2_{\on{discr}}(\Bun_G(\BF_q))$
are motivic in the following sense:

\medskip

\begin{conj} \hfill

\smallskip

\noindent{\em(a)}  
Assume that \conjref{c:locally finite O}(a) holds. Then the space $L^2_{\on{discr}}(\Bun_G(\BF_q))$ splits as a direct sum 
$$\underset{\sigma\in \LS^{\on{mot,coarse}}_\cG(\BC)}\oplus\, L^2_{\on{discr}}(\Bun_G(\BF_q))_\sigma,$$
where for every $x\in X$, the support of $L^2_{\on{discr}}(\Bun_G(\BF_q))_\sigma$ in
$$\on{Specm}(\Sph(G)_{x,\BC})\simeq (\cG/\!/\on{Ad}(\cG))(\BC)$$ 
is a single point equal to the image of $\sigma$ under
$$\LS^{\on{mot,coarse}}_\cG(\BC)\to \LS^{\on{arithm,coarse}}_\cG(\BC) \to (\cG/\!/\on{Ad}(\cG))(\BC),$$
where the last arrow corresponds to the restriction along $x\to X$. 

\smallskip

\noindent{\em(b)}  
Assume additionally that \conjref{c:locally finite O}(b) holds. Then each $L^2_{\on{discr}}(\Bun_G(\BF_q))_\sigma$
is at most one-dimensional. 

\end{conj} 

\appendix

\section{Proof of \thmref{t:present}} \label{s:proof of present}

We will focus on the question of $\Frob$-invariance, and as a by-product explain how to 
guarantee properties stated in (a) and (b) of the theorem. 

\ssec{Analysis of the orbits}

\sssec{}

Recall that the connected components of $\LS_\cG^{\on{restr}}$ are in bijection with isomorphism classes
of semi-simple $\cG$-local systems on $X$. 

\medskip

We will distinguish three scenarios for $\Frob$-orbits of connected components of $\LS_\cG^{\on{restr}}$: 

\medskip

\noindent{(a)} The orbit consists of one element;

\smallskip

\noindent{(a')} The orbit is finite;

\smallskip

\noindent{(b)} The orbit is infinite.

\sssec{}

Note that in case (b) there is nothing to prove as far as $\Frob$-invariance is concerned: choose a connected component $\CZ_\kappa$  on the given orbit;
exhibit it as 
$$\underset{i}{``\on{colim}"}\, \CZ_{\alpha,\kappa},$$
and spread this presentation in the $\Frob$-invariant way to the other connected components on this orbit. 

\sssec{}

We claim that (a') reduces to case (a): 

\medskip

For a given connected component $\CZ_\kappa$, let $n$ be the smallest integer such that
$$\Frob^n(\CZ_\kappa)=\CZ_\kappa.$$

\medskip

Up to replacing $\BF_q\rightsquigarrow \BF_{q^n}$, by case (a), we can assume that $\CZ_\kappa$
admits a presentation 
$$\underset{i}{``\on{colim}"}\, \CZ_{\alpha,\kappa},$$
preserved by $\Frob^n$. 

\medskip

This presentation spreads uniquely to other connected components on the $\Frob$-orbit of $\CZ_\kappa$
in the $\Frob$-invariant way. 

\sssec{}

Thus, we only need to consider case (a).

\begin{rem}

For a $\Frob$-orbit on $\pi_0(\LS_\cG^{\on{restr}})$ of cardinality $>1$, let 
$$\LS_\cG^{\on{restr}}{}'\subset \LS_\cG^{\on{restr}}$$
be the corresponding union of connected components. It is easy to see that
$$\Tr(\Frob,\QCoh(\LS_\cG^{\on{restr}}{}'))=0.$$

This implies that for
$$\Shv_\Nilp(\Bun_G)':=\QCoh(\LS_\cG^{\on{restr}}{}')\underset{\QCoh(\LS_\cG^{\on{restr}})}\otimes \Shv_\Nilp(\Bun_G).$$
we have 
$$\Tr(\Frob,\Shv_\Nilp(\Bun_G)')=0.$$

This implies that scenario (a) is the only one relevant in any case.

\end{rem} 

\ssec{Reduction to the problem about the coarse moduli space}

\sssec{}

Thus, we fix a Frobenius-invariant connected component $\CZ_\kappa$.
Let $\CZ^{\on{coarse}}_\kappa$ be the corresponding coarse moduli space, see \cite[Sect. 5.4]{AGKRRV1}. 

\medskip

By assumption, $\CZ^{\on{coarse}}_\kappa$ is also Frobenius-invariant.

\sssec{}

By \cite[Theorem 5.4.2]{AGKRRV1}, $\CZ^{\on{coarse}}_\kappa$ is a formal affine scheme with the underlying
reduced subscheme equal to $\on{pt}$, and the map
$$\fr:\CZ_\kappa\to \CZ^{\on{coarse}}_\kappa$$
is a relative algebraic stack.

\medskip

Hence, to produce a desired presentation of $\CZ_\kappa$, it suffices to find the corresponding 
presentation of $\CZ^{\on{coarse}}_\kappa$.

\sssec{} \label{sss:V}

We claim that in order to do so, it suffices to find a finite-dimensional $\ol\BQ_\ell$-vector space $V$, equipped with
an action of the Frobenius, and a Frobenius-equivariant map
\begin{equation} \label{e:V}
V\to \fm_{\CZ^{\on{coarse}}_\kappa},
\end{equation} 
such that the composition 
$$V\to \fm_{\CZ^{\on{coarse}}_\kappa}\to \fm_{\CZ^{\on{coarse}}_\kappa}/\fm^2_{\CZ^{\on{coarse}}_\kappa}$$
is surjective, where 
$$\fm_{\CZ^{\on{coarse}}_\kappa}\subset A_{\CZ^{\on{coarse}}_\kappa}$$
is the maximal ideal in the classical complete local $\sfe$-algebra corresponding to $\CZ^{\on{coarse}}_\kappa$.

\medskip

Indeed, consider the formal completion $(V^*)^\wedge_0$ of the dual vector space $V^*$
at $0$. Note that the resulting map
\begin{equation} \label{e:Z to V}
\CZ^{\on{coarse}}_\kappa\to (V^*)^\wedge_0
\end{equation} 
is schematic. Indeed, it suffices to show that \eqref{e:Z to V} is schematic at the classical level
(see \cite[Sect. 3.1.3]{AGKRRV1} and references therein), while for every $n$
$$(\CZ^{\on{coarse}}_\kappa)^{\on{cl}} \underset{(V^*)^\wedge_0}\times (V^*)_n \simeq 
\on{Spf}(A_{\CZ^{\on{coarse}}_\kappa}/\fm_{\CZ^{\on{coarse}}_\kappa}^n)^{\on{cl}}\simeq \on{Spec}(A_{\CZ^{\on{coarse}}_\kappa}/\fm_{\CZ^{\on{coarse}}_\kappa}^n)^{\on{cl}},$$
where $(V^*)_n$ denotes the $n$th inifinitesimal neighborhood of $0$ in $V^*$. 

\medskip

Now, 
\begin{equation} \label{e:pres V}
(V^*)^\wedge_0\simeq \underset{n}{\on{colim}}\, \left(\{0\}\underset{\wt\Sym{}^n(V^*)}\times V^*\right),
\end{equation} 
where\footnote{Although this does not matter in characteristic $0$, we distinguish $\wt\Sym^n(W)$
from $\Sym^n(W):=(W^{\otimes n})_{\Sigma_n}$.}:

\medskip

\begin{itemize}

\item For a vector space $W$, we denote:
$$\wt\Sym{}^n(W):=(W^{\otimes n})^{\Sigma_n},$$ where $\Sigma_n$ is the symmetric group (note that $\wt\Sym{}^n(V^*)\simeq (\Sym^n(V))^*$); 

\medskip

\item The map $W\to \wt\Sym{}^n(W)$ is $w\mapsto w^{\otimes n}$.

\end{itemize}

\begin{rem}

Note that if $\dim(V)>1$, the terms in the presentation \eqref{e:pres V} have a non-trivial derived structure. 
The more familiar presentation of $(V^*)^\wedge_0$ as an ind-scheme is
$$(V^*)^\wedge_0\simeq  \underset{n}{\on{colim}}\, (V^*)_n,$$
for $(V^*)_n$ as above. Note that the two presentations agree at the classical level. However, in the latter presentation,
the maps $(V^*)_n\to (V^*)^\wedge_0$ are \emph{not} regular embeddings.

\end{rem} 

\sssec{}

In order to show the existence of \eqref{e:V}, we will prove:

\begin{thm} \label{t:eigs coarse}
The eigenvalues of $\Frob$ on $\fm_{\CZ^{\on{coarse}}_\kappa}/\fm^2_{\CZ^{\on{coarse}}_\kappa}$
are Weil numbers of weight $\leq -1$.
\end{thm}

\ssec{Proof of the existence of \texorpdfstring{$V$}{V}}

Let us show how \thmref{t:eigs coarse} implies the existence of \eqref{e:V}. 

\sssec{}

Consider the following general situation. Let $A$ be a complete local algebra over an algebraically closed field with maximal ideal $\fm$,
and equipped with an automorphism $\phi$.

\medskip

Suppose that the eigenvalues of $\phi$ on $\fm/\fm^2$ have the following property: 

\medskip

There exists an integer $n$ such that for all $n'\geq n$, the $n'$-fold product of scalars that appear as eigenvalues
of $\phi$ on $\fm/\fm^2$ is \emph{not} an eigenvalue on $\fm/\fm^2$. 

\sssec{}

We claim:

\begin{prop} \label{p:V}
Under the above circumstances, we can find a vector space $V$, equipped with an automorphism
$\phi_V$ and a map $V\to \fm$, compatible with the automorphisms, and such that the composition
$$V\to \fm\to \fm/\fm^2$$
is surjective.
\end{prop}

\sssec{}

Note that \thmref{t:eigs coarse} implies that the assumption of \propref{p:V} holds 
for the action of $\Frob$ on $\fm_{\CZ^{\on{coarse}}_\kappa}/\fm^2_{\CZ^{\on{coarse}}_\kappa}$.

\medskip

So, \thmref{t:eigs coarse} and \propref{p:V} imply the existence of the desired substacks. 

\sssec{Proof of \propref{p:V}} \label{sss:who is V}

Let $n$ be as in the proposition. We let $V\subset \fm/\fm^n$ be the subspace spanned
by generalized $\phi$-eigenspaces corresponding to the eigenvalues that appear in $\fm/\fm^2$.

\medskip

We claim that the embedding $V\hookrightarrow \fm/\fm^n$ can be (uniquely) lifted to a map
$$V\to \fm,$$
in a way compatible with the automorphisms.

\medskip

Since $A$ is complete, we can perform the lifting by induction. Thus, we assume that we have found
a lifting 
$$V\to \fm/\fm^{n'}, \quad n'\geq n,$$
and we wish to lift it further to a map
$$V\to \fm/\fm^{n'+1}.$$.

I.e., we have to split a short exact sequence
\begin{equation} \label{e:SES}
0\to \fm^{n'}/\fm^{n'+1}\to V'\to V\to 0
\end{equation}
in a way compatible with the endomorphisms.

\medskip

Now, the assumption of the proposition implies that the eigenvalues of $\phi$
on $\fm^{n'}/\fm^{n'+1}$ are disjoint from those on $V$.

\medskip

Hence, \eqref{e:SES} admits a unique splitting.

\qed[Proof of \propref{p:V}]

\sssec{Proof of \thmref{t:present}, point \em{(a)}} \label{sss:Chev}

To make the construction invariant with respect to $\tau$, we proceed as follows:

\medskip

It is easy to see that we only have to consider the case of $\kappa$ such that $\CZ_\kappa$ 
is itself $\tau$-invariant. In this case, we take $V$ constructed above, and replace it by
$$V\oplus V^\tau.$$

\sssec{Proof of \thmref{t:present}, point \em{(b)}} 

Since $\LS_\cG^{\on{arithm}}$ has finitely many connected 
components\footnote{This is the first part of the proof \cite[Theorem 24.1.4]{AGKRRV1}, see Sect. 25.1.2 in {\it loc.cit.},
which is different in nature from the rest of the argument},
it suffices to consider the case of one Frobenius-invariant connected component
$\CZ_\kappa$. Furthermore, the problem reduces to one about $\CZ^{\on{coarse}}_\kappa$.

\medskip

It suffices to show that in the presentation \eqref{e:pres V}, the maps
\begin{equation} \label{e:pres V bis}
(\{0\}\underset{\wt\Sym^n(V^*)}\times V^*)^{\Frob}\to (V^*)^{\Frob}
\end{equation} 
are isomorphisms. 

\medskip

In fact, we claim that both sides in \eqref{e:pres V bis} are isomorphic to $\on{pt}$. Using the isomorphism
\eqref{e:pres V}, it is sufficient to prove this for the left-hand side.

\medskip

We claim that both
$$(\wt\Sym^n(V^*))^{\Frob} \text{ and } (V^*)^{\Frob}$$ 
are isomorphic to the $\on{pt}$.

\medskip

Indeed, this follows from the fact that the eigenvalues of $\Frob$ on $V$ and on $\Sym^n(V)$ 
are $\neq 1$.

\ssec{Structure of Frobenius-equivariant orbits}

The rest of this section is devoted to the proof of \thmref{t:eigs coarse}. 

\sssec{}

Let $\sigma\in \CZ_\kappa$ be the unique closed point (see \cite[Sect. 3.7]{AGKRRV1}). 
By {\it loc. cit.}, the resulting map
$$\iota:\on{pt}/\on{Aut}(\sigma)\to \CZ_\kappa$$
is a closed embedding. 

\medskip

For an integer $n$, let $\CZ^{(n)}_\kappa$ denote
the classical $n$-th infinitesimal neighborhood of $\on{pt}/\on{Aut}(\sigma)$ in $\CZ_\kappa$.

\medskip

It follows from \cite[Theorem 1.4.5]{AGKRRV1} that $\CZ^{(n)}_\kappa$ is an algebraic stack of finite type.

\sssec{}

Let $\CZ^{\on{coarse},(1)}_\kappa$ be the 1st infinitesimal neighborhood of $\on{pt}$ in $\CZ^{\on{coarse}}_\kappa$,
and consider the fiber product: 
$$\CZ^{\on{coarse},(1)}_\kappa\underset{\CZ^{\on{coarse}}_\kappa}\times \CZ^{(n)}_\kappa,$$
understood in the sense of classical algebraic geometry. 

\medskip

We will prove:

\begin{prop} \label{p:inj}
For $n$ large enough, the map 
$$\Gamma(\CZ^{\on{coarse},(1)}_\kappa,\CO_{\CZ^{\on{coarse},(1)}_\kappa})\to
\Gamma(\CZ^{\on{coarse},(1)}_\kappa\underset{\CZ^{\on{coarse}}_\kappa}\times \CZ^{(n)}_\kappa,
\CO_{\CZ^{\on{coarse},(1)}_\kappa\underset{\CZ^{\on{coarse}}_\kappa}\times \CZ^{(n)}_\kappa})$$
is injective.
\end{prop}

\sssec{}

Let us accept this proposition temporarily and proceed with the proof of \thmref{t:eigs coarse}.

\ssec{Estimating the Frobenius eigenvalues}

\sssec{} \label{sss:Frob acts}

By assumption, the isomorphism class of $\sigma$ is $\Frob$-invariant. In particular, the stack
$\on{pt}/\on{Aut}(\sigma)$ and hence $\CZ^{(n)}_\kappa$ acquire an action of $\Frob$. Furthermore, the map
$$\CZ^{\on{coarse},(1)}_\kappa\underset{\CZ^{\on{coarse}}_\kappa}\times \CZ^{(n)}_\kappa\to \CZ^{\on{coarse},(1)}_\kappa$$
is Frobenius-equivariant. 

\medskip

Hence, by \propref{p:inj}, in order to prove \thmref{t:eigs coarse}, it suffice to show that the eigenvalues
of $\Frob$ on 
$$\Gamma(\CZ^{\on{coarse},(1)}_\kappa\underset{\CZ^{\on{coarse}}_\kappa}\times \CZ^{(n)}_\kappa,
\CO_{\CZ^{\on{coarse},(1)}_\kappa\underset{\CZ^{\on{coarse}}_\kappa}\times \CZ^{(n)}_\kappa})$$
\emph{modulo constants} have weights $\leq -1$.

\sssec{}

Note that since $\CZ_\kappa$ is the quotient of an affine formal scheme by a reductive group, for an embedding
of closed substacks
$$\CZ'_\kappa\subset \CZ''_{\kappa}\subset \CZ_\kappa,$$
the induced map
$$\Gamma(\CZ''_{\kappa},\CO_{\CZ''_{\kappa}})\to \Gamma(\CZ'_{\kappa},\CO_{\CZ'_{\kappa}})$$
is surjective.

\sssec{}

Consider the closed embedding 
$$\CZ^{\on{coarse},(1)}_\kappa\underset{\CZ^{\on{coarse}}_\kappa}\times \CZ^{(n)}_\kappa\to \CZ^{(n)}_\kappa,$$

We obtain that it suffices to show that the eigevalues of $\Frob$ on 
$$\Gamma(\CZ^{(n)}_\kappa,\CO_{\CZ^{(n)}_\kappa})$$
\emph{modulo constants} have weights $\leq -1$.

\sssec{}

The fact that the isomorphism class of $\sigma$ is $\Frob$-invariant means that $\sigma$ admits a Weil structure;
let us choose one (we will specify a particularly convenient choice in \secref{sss:wt 0} below). 

\medskip

The choice of Weil structure on $\sigma$ equips $\on{Aut}(\sigma)$ with an action of $\Frob$, so that
the resulting action on $\on{pt}/\on{Aut}(\sigma)$ equals one from \secref{sss:Frob acts}.

\medskip

We obtain that the Frobenius eigenvalues on $\Gamma(\CZ^{(n)}_\kappa,\CO_{\CZ^{(n)}_\kappa})$ modulo
constants are contained in the set of scalars obtained as products (of cardinalities $\leq n$) 
of the Frobenius eigenvalues on
$$\on{Fib}(T^*_\sigma(\CZ_\kappa)\to T^*_\sigma(\on{pt}/\on{Aut}(\sigma))\simeq (H^1(X,\cg_\sigma))^*.$$

\sssec{} \label{sss:wt 0}

By \cite[Corollary 3.7.4]{AGKRRV1}, $\sigma$ is semi-simple. Hence, by by \cite[Proposition 25.4.6]{AGKRRV1}, we can choose 
the Weil structure on $\sigma$ so that is pure of weight $0$. 

\medskip

With this choice, $\cg_\sigma$, viewed as a local system of vector spaces, is pure of weight $0$. Hence,
by Weil-II, the Frobenius eigenvalues on $H^1(X,\cg_\sigma)$ have weight $1$. 

\qed[\thmref{t:eigs coarse}]

\ssec{Proof of \propref{p:inj}}

\sssec{}

Let $\CZ^{\on{rigid}_x}_\kappa$ be the rigidified version of $\CZ_\kappa$, i.e.,
$$\CZ^{\on{rigid}_x}_\kappa:=\CZ_\kappa\underset{\on{pt}/\cG}\times \on{pt},$$
where $\CZ_\kappa\to \on{pt}$ corresponds to the choice of a point $x\in X$.

\medskip

We have
$$\CZ_\kappa\simeq \CZ^{\on{rigid}_x}_\kappa/\cG.$$

\sssec{}

Consider the corresponding (affine) subscheme
$$\CZ^{\on{rigid}_x,(n)}_\kappa\subset  \CZ^{\on{rigid}_x}_\kappa.$$

\medskip

It suffices to show that the map
$$\Gamma(\CZ^{\on{coarse},(1)}_\kappa,\CO_{\CZ^{\on{coarse},(1)}_\kappa})\to
\Gamma(\CZ^{\on{coarse},(1)}_\kappa\underset{\CZ^{\on{coarse}}_\kappa}\times \CZ^{\on{rigid}_x,(n)}_\kappa,
\CO_{\CZ^{\on{coarse},(1)}_\kappa\underset{\CZ^{\on{coarse}}_\kappa}\times \CZ^{\on{rigid}_x,(n)}_\kappa})$$
is injective for $n$ sufficiently large. 

\sssec{}

Let $\CZ^{\on{rigid}_x,(\infty)}_\kappa$ denote the formal completion of $\CZ^{\on{rigid}_x}_\kappa$ along
the preimage of $\on{pt}/\on{Aut}(\sigma)$. 

\medskip 

Since $\Gamma(\CZ^{\on{coarse},(1)}_\kappa,\CO_{\CZ^{\on{coarse},(1)}_\kappa})$ is finite-dimensional,
it suffices to show that the map 
$$\Gamma(\CZ^{\on{coarse},(1)}_\kappa,\CO_{\CZ^{\on{coarse},(1)}_\kappa})\to
\Gamma(\CZ^{\on{coarse},(1)}_\kappa\underset{\CZ^{\on{coarse}}_\kappa}\times \CZ^{\on{rigid}_x,(\infty)}_\kappa,
\CO_{\CZ^{\on{coarse},(1)}_\kappa\underset{\CZ^{\on{coarse}}_\kappa}\times \CZ^{\on{rigid}_x,(\infty)}_\kappa})$$
is injective. 

\sssec{}

Note that since $\cG$ is reductive, the map
$$\Gamma(\CZ^{\on{coarse},(1)}_\kappa,\CO_{\CZ^{\on{coarse},(1)}_\kappa})\to
\Gamma(\CZ^{\on{coarse},(1)}_\kappa\underset{\CZ^{\on{coarse}}_\kappa}\times \CZ^{\on{rigid}_x}_\kappa,
\CO_{\CZ^{\on{coarse},(1)}_\kappa\underset{\CZ^{\on{coarse}}_\kappa}\times \CZ^{\on{rigid}_x}_\kappa})$$
is injective, with the image being the subspace of $\cG$-invariants in the right-hand side. 

\medskip

Therefore, it remains to show that the restriction map
\begin{multline*} 
\Gamma(\CZ^{\on{coarse},(1)}_\kappa\underset{\CZ^{\on{coarse}}_\kappa}\times \CZ^{\on{rigid}_x}_\kappa,
\CO_{\CZ^{\on{coarse},(1)}_\kappa\underset{\CZ^{\on{coarse}}_\kappa}\times \CZ^{\on{rigid}_x}_\kappa})\to \\
\to \Gamma(\CZ^{\on{coarse},(1)}_\kappa\underset{\CZ^{\on{coarse}}_\kappa}\times \CZ^{\on{rigid}_x,(\infty)}_\kappa,
\CO_{\CZ^{\on{coarse},(1)}_\kappa\underset{\CZ^{\on{coarse}}_\kappa}\times \CZ^{\on{rigid}_x,(\infty)}_\kappa})
\end{multline*}
is injective \emph{on $\cG$-invariant elements}. 

\sssec{}

Since the map from a Noetherian local ring to its completion is injective, 
it suffices to show that if $f$ is a non-zero $\cG$-invariant element in 
$$\Gamma(\CZ^{\on{coarse},(1)}_\kappa\underset{\CZ^{\on{coarse}}_\kappa}\times \CZ^{\on{rigid}_x}_\kappa,
\CO_{\CZ^{\on{coarse},(1)}_\kappa\underset{\CZ^{\on{coarse}}_\kappa}\times \CZ^{\on{rigid}_x}_\kappa}),$$
then it does not vanish on any Zariski neighborhood of the preimage of $\on{pt}/\on{Aut}(\sigma)$. 

\medskip

However, this follows from the Hilbert-Mumford criterion: for any point 
$$z\in \CZ^{\on{coarse},(1)}_\kappa\underset{\CZ^{\on{coarse}}_\kappa}\times \CZ^{\on{rigid}_x}_\kappa,$$ we can find
a copy of $\BG_m\subset \cG$ that contracts this point onto the preimage of $\on{pt}/\on{Aut}(\sigma)$: this is because
this preimage is the unique closed $\cG$-orbit in $\CZ^{\on{rigid}_x}_\kappa$.

\end{document}